%% file: new_main.tex
\documentclass{article}
\title{Bondal--Orlov reconstruction for tame stacks \\ with trivial generic stabilizer}
\author{Daigo Ito and Noah Olander}
\date{}
\input{style}

\begin{document}
\maketitle

\begin{abstract}
    We generalize the Bondal--Orlov Reconstruction Theorem to smooth, proper, tame algebraic stacks with generically trivial stabilizer, whose canonical bundles are \emph{nowhere torsion}: They do not become trivial after pullback along any finite morphism from a projective curve. Our main tools are coherent Tannaka duality as proved by Lurie and by Hall and Rydh, and Grothendieck duality for proper tame stacks as developed by Hall and Priver. We additionally use the notion of nowhere torsion line bundle on a scheme to prove a version of the Bondal--Orlov Reconstruction Theorem for finite type, separated, Gorenstein schemes which are not necessarily proper, which generalizes work of Ballard, Favero, Ito, and Matsui.
\end{abstract}

\tableofcontents
\section{Introduction}

Let $k$ be a field and $X$ a finite type, separated scheme over $k$. Many open problems in algebraic geometry attempt to determine how much information about $X$ can we recover from $D^b_{\operatorname{Coh}}(X)$, its bounded derived category of coherent sheaves, viewed as a $k$-linear triangulated category. A major early step in this program was the following result of Bondal and Orlov.

\begin{theorem}[\cite{bondal_orlov_2001}]
    Assume $X$ is a smooth, projective variety over $k$ with canonical bundle $\omega_X$. Suppose either $\omega_X$ or $\omega_X^{-1}$ is ample. If $Y$ is a smooth 
    variety over $k$ and $\Phi : D^b_{\operatorname{Coh}}(X) \cong D^b_{\operatorname{Coh}}(Y)$ is a $k$-linear, exact equivalence of triangulated categories, then there are a $k$-isomorphism $f : X \to Y$, a line bundle $\ecal{L}$ on $Y$, and an integer $n$ such that $\Phi$ is isomorphic to the functor
    $
  K \mapsto Rf_*(K) \otimes^{\mathbb{L}}\ecal{L}[n].  
    $
\end{theorem}

This result has been generalized many times in the literature. 
In \cite[Corollary 6.3]{ballard2011derived}, it is shown that ``smooth'' can be replaced with ``Gorenstein,'' but $Y$ is assumed projective. In \cite[Theorem 3.15]{FAVERO20121955}, a version is proved where $X$ need not be proper, but the restriction of $\omega_X$ to every complete closed subvariety $Z \subset X$ is assumed to be either ample or anti-ample, and $Y$ is assumed to be divisorial. Finally, in \cite{ito2025polarizations} (which builds on the work \cite{ito2024new}), a version is proved in which $X$ is proper and Gorenstein and $\omega_X$ is a  ``$\otimes$-generating'' line bundle, and $Y$ is allowed to be any variety over $k$. Here a line bundle $\ecal{L}$ on $X$ is said to be $\otimes$-generating if its tensor powers compactly generate the derived category $D_{qc}(X)$ of $X$. In the work \cite{itoolander2025derived} of the authors of this paper, this condition is elucidated: If $X$ is proper, $\ecal{L}$ is $\otimes$-generating if and only if for every closed subvariety $Z \subset X$, either $\ecal{L}|_Z$ or $\ecal{L}^{-1}|_Z$ is big. In a slightly different direction, there is also a generalization to orbifolds by Kawamata \cite{kawamata2002francia}, which we will discuss in greater detail below. 

The goal of this paper is to generalize all these results, starting with schemes and later working with stacks. We say a line bundle $\ecal{L}$ on $X$ is \emph{nowhere torsion} if whenever $\ecal{F}$ is a coherent sheaf on $X$ whose support is proper and $\ecal{F} \otimes \ecal{L} \cong \ecal{F}$, then the support of $\ecal{F}$ has dimension zero. 

\begin{theorem}(See Theorem \ref{theorem-reconstruction-for-schemes})
\label{thm-reconstructionforschemesintroversion}
    Suppose $X, Y$ are finite type, separated, connected, Gorenstein schemes over a field $k$. Suppose $\omega_X$ is a nowhere torsion line bundle. Let $\Phi: D^b_{\operatorname{Coh}}(X) \to D^b_{\operatorname{Coh}}(Y)$ be a $k$-linear, exact equivalence of categories. Then there exist a $k$-isomorphism $f : X \to Y$, a line bundle $\ecal{L}$ on $Y$, and an integer $n$ such that $\Phi(K) \cong Rf_*(K) \otimes^{\mathbb{L}}\ecal{L}[n]$ for every $K \in D^b_{\operatorname{Coh}}(X)$. If additionally, $\dim X> 0$ and $X$ is proper  over $k$, then there is an isomorphism of functors $\Phi\cong Rf_*(-)\otimes^{\mathbb{L}}\ecal{L}[n]$.
\end{theorem}

Our methods are very similar to Favero's from \cite{FAVERO20121955}. We expect this result to be unsurprising, perhaps even folklore to experts, and similar conditions occur in \cite[Lemma 2.2]{Krahblowups} and \cite[Proposition 3.4, Corollary 3.5]{kawatani2010group}. However, we still hope people find it useful. We also provide in \S\ref{subsection-examples} examples of varieties to which our result applies but previous versions do not.

The bulk of the paper is devoted to extending the Bondal--Orlov Reconstruction Theorem to certain algebraic stacks. Let $\mathcal{X}$ be a smooth, separated, tame algebraic stack over $k$ with trivial generic stabilizer. When $k$ has characteristic zero, this is equivalent to asking that $\mathcal{X}$ be a smooth, separated, Deligne--Mumford stack with trivial generic stabilizer (to many, this would simply be called an orbifold).  Such stacks, as well as their derived categories, arise naturally even in the study of quasi-projective varieties. 

\begin{example}
    Suppose $k = \mathbb{C}$ and $G \subset \operatorname{SL}_3(\mathbb{C})$ is a finite subgroup. Let $G$ act on $\mathbb{A}^3_{\mathbb{C}}$ through the standard representation. Set $\mathcal{X} = [\mathbb{A}^3_{\mathbb{C}}/G]$ and let $X = \mathbb{A}^3_{\mathbb{C}}/G$ be its coarse space. By \cite{BridgelandKingReid2001}, there is a projective crepant resolution $Y \to X$ of $X$ together with an exact equivalence of categories
$
D^b_{\operatorname{Coh}}(\mathcal{X}) \cong D^b_{\operatorname{Coh}}(Y).
$

Going one step further, if $G$ is abelian, then it is proved in \cite{CrawIshii2004} that for \emph{any} projective crepant resolution $Y' \to X$, the variety $Y'$ is related to $Y$ by a sequence of flops, and there is an exact, $k$-linear equivalence of categories 
$
D^b_{\operatorname{Coh}}(Y) \cong D^b_{\operatorname{Coh}}(Y').
$
Note that there can be many non-isomorphic $Y'$, for example when 
$$
\mathbb{Z}/6 \cong G = \{\begin{pmatrix}
    \zeta & 0 & 0 \\
    0 & \zeta^2 & 0 \\
    0 & 0 & \zeta^3
\end{pmatrix} : \zeta \in \mu_6(\mathbb{C})\},
$$
there are 5 distinct projective crepant resolutions of $X$, see \cite{cacciatori2009dbranes}. 

We can also view the coarse space morphism $\mathcal{X} \to X$ as a crepant resolution of singularities since it is proper and an isomorphism over an open whose complement has codimension $2$, and $\mathcal{X}$ is smooth over $\mathbb{C}$. Among all resolutions of singularities of $X$, this one is distinguished: It is the canonical stack morphism associated to $X$, see \cite[Proposition 2.8]{Vistoli1989} for the general construction. 

There are similar crepant resolutions when $G = (\bb Z/2\bb Z)^2$ acts on $\bb A^3$, as described in \cite{donagi2017global}*{\S5.1}, and on $(\bb P^1)^3$, see Example \ref{example-multiple-crepant}. We point this out because the  results below require properness.
\end{example}

It is natural to ask whether $\mathcal{X}$ above can also be characterized uniquely among all smooth tame stacks with equivalent derived category. Our main result says that in the proper case and under an assumption on the canonical bundle, this holds. 

\begin{theorem}(See Theorem \ref{theorem-bo-tame-stack})
\label{thm-bo-tamestack-introversion}
    Let $\mathcal{X}, \mathcal{Y}$ be smooth, proper, connected tame stacks over $k$ of positive dimension with trivial generic stabilizers. Suppose $D^b_{\operatorname{Coh}}(\mathcal{X}) \cong D^b_{\operatorname{Coh}}(\mathcal{Y})$ as $k$-linear triangulated categories. If the canonical bundles $\omega_{\mathcal{X}}$ and $\omega_{\mathcal{Y}}$ are nowhere torsion, then $\mathcal{X} \cong \mathcal{Y}$.
\end{theorem}

``Nowhere torsion'' here is defined exactly as in the case of schemes. 
This result can be interpreted as saying that among all Fourier--Mukai partners of an $\mathcal{X}$ as in the theorem, 
there is at most one having nowhere torsion canonical bundle. If it exists, this one should be viewed as special. Our result strengthens \cite[Theorem 6.1]{kawamata2002francia}, which proves that the coarse spaces of $\mathcal{X}, \mathcal{Y}$ are isomorphic under some additional hypotheses, including that $k = \mathbb{C}$, the stabilizer groups of $\mathcal{X}, \mathcal{Y}$ are cyclic with character groups generated by the canonical bundles, and the coarse spaces have ample or anti-ample canonical sheaves. 

\begin{examples} \ 
\begin{enumerate}
     \item The stacky weighted projective space $\mathcal{P}(1,1,2)$ has the Hirzebruch surface $\mathbb{F}_2$ as a Fourier--Mukai partner, see \cite{auroux2008mirror}*{Theorem 2.9} and Example \ref{example-f2}. We will see that the first has good canonical bundle, but the second does not: The canonical bundle is trivial on the unique $-2$ curve.
     \item In Example \ref{example-multiple-crepant}, based on \cite{donagi2017global}, we observe that a quotient stack $[(\bb P^1)^3/(\bb Z/2\bb Z)^2]$, which has nowhere torsion canonical bundle, has at least two distinct, non-stacky Fourier--Mukai partners, given as distinct crepant resolutions of $(\bb P^1)^3/(\bb Z/2\bb Z)^2$ which are related by flops. 
\end{enumerate}
\end{examples}

For more examples of stacks to which Theorem \ref{thm-bo-tamestack-introversion} applies, see Section \ref{subsection-stackyexamples}.


Importantly, the classical Bondal--Orlov Reconstruction Theorem also computes the group of autoequivalences of the derived category of a smooth projective variety with ample or anti-ample canonical bundle. In the setting of Theorem \ref{thm-bo-tamestack-introversion} we are also able to do this to some extent, but the answer is more complicated. Let $\operatorname{Std}(\mathcal{X}) \subset \operatorname{Aut} D^b_{\operatorname{Coh}}(\mathcal{X})$ denote the subgroup of standard autoequivalences, 
 that is, the subgroup generated by pushforwards by automorphisms, tensoring with a line bundle, and shifting by an integer. Thus $\operatorname{Std}(\mathcal{X}) \cong (\operatorname{Pic}(\mathcal{X}) \rtimes \operatorname{Aut}(\mathcal{X})) \times \mathbb{Z}[1]$.

\begin{theorem}(See Construction \ref{construction-localsubgroup})
\label{theorem-autoequivsintroversion}
    Let $\cal X$ be as in Theorem \ref{thm-bo-tamestack-introversion}. There is a split surjective homomorphism $\operatorname{Aut} D^b_{\operatorname{Coh}}(\mathcal{X}) \twoheadrightarrow \operatorname{Std}(\mathcal{X})$. Denote its kernel $\operatorname{Loc}(\mathcal{X})$, i.e.,
    $$
    \operatorname{Aut} D^b_{\operatorname{Coh}}(\mathcal{X}) = \operatorname{Loc}(\mathcal{X}) \rtimes \operatorname{Std}(\mathcal{X}).
    $$ 
    Furthermore, the subgroup $\operatorname{Loc}(\mathcal{X})$ is determined locally: 
    There is a natural injective homomorphism 
    \[
    \operatorname{Loc}(\mathcal{X}) \hookrightarrow \prod _{x \in |\mathcal{X}| \text{ closed}}  \operatorname{Aut} D^b_{\operatorname{Coh},x}(\mathcal{X}). \qedhere 
    \]
    %
\end{theorem}

We call $\operatorname{Loc}(\mathcal{X}) \subset \operatorname{Aut} D^b_{\operatorname{Coh}}(\mathcal{X})$ the subgroup of \emph{local auto-equivalences}, and we characterize its elements in 
Construction \ref{construction-localsubgroup}. 
When $\cal{X} = X$ is a scheme, we have $\operatorname{Loc}(X) = 1$ by Theorem \ref{thm-reconstructionforschemesintroversion}, but it can be more complicated when $\mathcal{X}$ is a stack. For example, there is a spherical object in $D^b_{\operatorname{Coh}}(\mathcal{P}(1,1,2))$ whose support is equal to the unique stacky point. The associated spherical twist is a non-trivial element of $\operatorname{Loc}(\mathcal{X})$. See Example \ref{example-spherical-twist-on-p112}. Nevertheless, due to the local nature of the group $\operatorname{Loc}(\mathcal{X})$, we are sometimes able to compute it explicitly. We carry this out in \ref{subsection-stackysurfaceexamples} for some classes of stacky surfaces, relying heavily on the work \cite{ishii2005autoequivalences}, which uses the derived McKay correspondence and canonical stacks to compute the group of autoequivalences of the derived category of the minimal resolution of a projective surface with $A_n$-singularities. The work loc. cit. illustrates the point that even a smooth projective variety may have a special stacky Fourier--Mukai partner which we can exploit to learn things about the original variety.

For particularly well-behaved stacks $\mathcal{X}$, we are able to prove a strong form of reconstruction, namely that every auto-equivalence of $D^b_{\operatorname{Coh}}(\mathcal{X})$ is standard and every Fourier--Mukai partner $\mathcal{Y}$ of $\mathcal{X}$ is isomorphic to $\mathcal{X}$ (with no assumption on $\omega_\mathcal{Y}$). Let us say a line bundle $\ecal{L}$ on $\mathcal{X}$ is \emph{point-wise $\otimes$-generating} if for every closed point $x$ of $\mathcal{X}$ with residual gerbe $i_x : \mathcal{G}_x \to \mathcal{X}$, every coherent sheaf on $\mathcal{G}_x$ is a quotient of a finite direct sum of tensor powers of $i_x^*\ecal{L}$.

\begin{theorem}
\label{thm-canonicalgeneratesintroversion}
    Let $\mathcal{X}$ be as in Theorem \ref{thm-bo-tamestack-introversion}. Assume additionally that the coarse space of $\mathcal{X}$ is a scheme, and that $\omega_{\mathcal{X}}$ is locally tensor generating. Let $\mathcal{Y}$ be a smooth proper connected tame stack over $k$ with trivial generic stabilizer. Suppose $\Phi : D^b_{\operatorname{Coh}}(\mathcal{X}) \to D^b_{\operatorname{Coh}}(\mathcal{Y})$ is a $k$-linear, exact equivalence. Then there are a $k$-isomorphism $f : \mathcal{X} \to \mathcal{Y}$, a line bundle $\ecal{L}$ on $\mathcal{Y}$, and an integer $n$ such that $\Phi \cong
    Rf_*(-) \otimes^{\mathbb{L}}\ecal{L}[n].
    $
\end{theorem}

Our main tools for proving this result are coherent Tannaka duality, see \cite{Hall-Rydh-coherent} or \cite{lurie2004tannaka}, and Grothendieck duality for proper tame stacks, as developed in \cite{hall2024generalizedbondalorlovfaithfulnesscriterion}. To prove Theorems \ref{thm-bo-tamestack-introversion} and \ref{theorem-autoequivsintroversion}, we show that our stacks $\mathcal{X}, \mathcal{Y}$ have large opens satisfying the hypotheses of Theorem \ref{thm-canonicalgeneratesintroversion}. We argue that these opens are isomorphic using coherent Tannaka duality, and argue separately that $\mathcal{X}$ and $\mathcal{Y}$ have isomorphic coarse moduli spaces using a Fourier--Mukai kernel, which exists by \cite{peng2024equivalences}. Combining these arguments, we are able to show $\mathcal{X} \cong \mathcal{Y}$, and with more care, we are able to obtain results about autoequivalences.

\section{Nowhere torsion line bundles on schemes}

In this part of the paper, we prove Theorem \ref{thm-reconstructionforschemesintroversion}. In Section \ref{subsection-nowheretorsiononschemes}, we give the definition and some basic properties of nowhere torsion line bundles on schemes. In Section \ref{subsection-reconstructionforschemes}, we give the proof of our reconstruction result. Finally, in Section \ref{subsection-examples}, we give examples of nowhere torsion line bundles on varieties, paying particular attention to varieties with nowhere torsion canonical bundles, as these give applications of our reconstruction theorem. See the introduction to that section for more details.

\subsection{Definition and basic properties}
\label{subsection-nowheretorsiononschemes}

Let $X$ be a finite type, separated scheme over a field $k$. Let $\ecal{L}$ be a line bundle on $X$.

\begin{definition}
    We shall say $\ecal{L}$ is \emph{nowhere torsion} if for every coherent sheaf $\ecal{F}$ on $X$ whose support is proper over $k$, $\ecal{F} \otimes \ecal{L} \cong \ecal{F}$ implies the support of $\ecal{F}$ is a finite set of closed points. We shall sometimes say $\ecal{L}$ is \emph{somewhere torsion} if it is not nowhere torsion.
\end{definition}

\begin{example}
    An (anti-)ample line bundle on $X$ is nowhere torsion. More generally, a $\otimes$-generating line bundle on $X$ is nowhere torsion (see \cite{itoolander2025derived} and Example \ref{example-tensorgeneratingisnowheretorsion}). 
\end{example}

The name is justified by Proposition \ref{prop-justificationofname} below.

\begin{lemma}
\label{lemma-nowhere torsionfinitesurj}
Let $f : Y \to X$ be a finite surjective morphism of (finite type, separated) schemes. Then $f^*\ecal{L}$ is nowhere torsion if and only if $\ecal{L}$ is nowhere torsion.
\end{lemma}

\begin{proof}
    Suppose $\ecal{L}$ is somewhere torsion. Then there is a coherent sheaf $\ecal{F}$ on $X$ with positive-dimensional, proper support such that $\ecal{F} \otimes \ecal{L} \cong \ecal{F}$. Then $f^*\ecal{F}$ has positive-dimensional, proper support (equal to the preimage of the support of $\ecal{F}$) and $f^*\ecal{F} \otimes f^*\ecal{L} \cong f^*\ecal{F}$. Hence $f^*\ecal{L}$ is somewhere torsion. 

    Conversely, suppose $f^*\ecal{L}$ is somewhere torsion. Then there is a coherent sheaf $\ecal{G}$ on $Y$ with positive-dimensional, proper support such that $\ecal{G} \otimes f^*\ecal{L} \cong \ecal{G}$. But then $f_*\ecal{G}$ has positive-dimensional, proper support (equal to the scheme-theoretic image of the support of $\ecal{G}$) and by the projection formula, $f_*\ecal{G} \otimes \ecal{L} \cong f_*\ecal{G}$. 
\end{proof}

Recall that $\ecal{L}$ is said to be \emph{torsion} if it is torsion as an element of the group $\operatorname{Pic}(X)$.

\begin{prop}
\label{prop-justificationofname}
   The following are equivalent:
    \begin{enumerate}
    \item $\ecal{L}$ is somewhere torsion.
    \item There is a complete curve $C \subset X$ such that, if $f : C^\nu \to X$ is the composition of the normalization $C^\nu \to C$ and the inclusion $C \subset X$, then $f^*\ecal{L}$ is a torsion line bundle on $C^\nu$.
    \item There is a finite morphism $f : C \to X$ where $C$ is a normal projective curve over $k$ and such that $f^*\ecal{L}$ is a torsion line bundle on $C$.
    \item There is a finite morphism $f : C \to X$ where $C$ is a projective scheme of dimension $1$ over $k$ and such that $f^*\ecal{L}\cong \ecal{O}_C$.
    \item There is a finite morphism $f : C \to X$ where $C$ is a normal projective curve over $k$ and such that $f^*\ecal{L}\cong \ecal{O}_C$. \qedhere
    \end{enumerate} 
\end{prop}

In particular, a necessary condition for $\ecal{L}$ to be somewhere torsion is that there exists a complete curve $C \subset X$ such that $\operatorname{deg}(\ecal{L}|_{C}) = 0$. Note however that this is not a sufficient condition, and also that being a nowhere torsion line bundle is not a numerical property: Any non-torsion line bundle of degree zero on a smooth projective curve $C/k$ is nowhere torsion of degree zero. 

\begin{proof}
    (i) $\implies$ (ii): If $\ecal{L}$ is somewhere torsion, then there exists some coherent sheaf $\ecal{F}$ on $X$ with proper support of positive dimension such that $\ecal{F} \otimes \ecal{L} \cong \ecal{F}$. Choose a complete curve $C \subset X$ contained in the support of $\ecal{F}$, and let $f : C^\nu \to X$ be as in the statement. Then $f^*\ecal{F}$ is a coherent sheaf on the normal curve $C^\nu$ so we can write $f^*\ecal{F}  = \ecal{E} \oplus \ecal{E}'$ where $\ecal{E}$ is a vector bundle and $\ecal{E}'$ is torsion. We also have $f^*\ecal{F}\otimes f^*\ecal{L} \cong f^*\ecal{F}$ and from this we see that 
    $$\ecal{E} \otimes f^*\ecal{L} \cong \ecal{E}$$
    as both are isomorphic to $f^*\ecal{F}/torsion$. Also, $\ecal{E}$ is not zero since we chose $C \subset \operatorname{Supp}(\ecal{F})$. Let $r > 0$ be its rank. Taking determinants gives $\operatorname{det}(\ecal{E}) \otimes f^*\ecal{L}^{\otimes r} \cong \operatorname{det}(\ecal{E})$, hence $f^*\ecal{L}^{\otimes r} \cong \ecal{O}_{C^\nu}$, as needed.

    Clearly (ii) $\implies$ (iii). Let us show (iii) $\implies$ (iv). Let $f : C \to X$ be a finite morphism with $C$ a normal projective curve and such that $f^*\ecal{L}$ is torsion, say of order $r>0$. Let $D \to C$ be the corresponding $\mu_r$-torsor, given by
    $$
    D = \underline{\operatorname{Spec}}_{C}(\ecal{O}_C \oplus \ecal{L} \oplus \cdots \oplus \ecal{L}^{\otimes r-1}).
    $$
    Then $D \to C$ is finite and the pullback $(D \to X)^*\ecal{L}$ is trivial, so (iv) holds.
    
    (iv) $\implies$ (v): If $f : C \to X$ is as in the statement of (iv), then consider the composition $C' \to C \to X$ where $C' \to C$ is the normalization of a reduced irreducible component of $C$.

    (v) $\implies$ (i): Let $f : C \to X$ be as in the statement of (v). Set $\ecal{F} = f_*\ecal{O}_C$. Then by the projection formula,
    $$\ecal{F} \otimes \ecal{L} \cong f_*f^*\ecal{L} \cong f_*\ecal{O}_C \cong \ecal{F},
    $$
    and the support of $\ecal{F}$ is proper of dimension $1$.
\end{proof}

\begin{corollary}
\label{corollary-powersofnowheretorsiononscheme}
    Let $0 \neq n \in \mathbb{Z}$. Then $\ecal{L}$ is nowhere torsion if and only if $\ecal{L}^{\otimes n}$ is nowhere torsion. 
\end{corollary}

\begin{proof}
    This is clear from condition (iii) above.
\end{proof}

\begin{example} \label{example-tensorgeneratingisnowheretorsion} Recall that we say a line bundle $\ecal L$ on $X$ is \textit{$\otimes$-generating} if the tensor powers classically generate (i.e. split generate) $\perf X$. That is, $$\perf X = \bra{\ecal L^{\otimes n} \mid n \in \bb Z}.$$ For example, any (anti-)ample line bundle is $\otimes$-generating. Since $\otimes$-generation is preserved under affine pullback and on a normal complete curve, $\otimes$-generation is equivalent to having non-zero degree \cite{itoolander2025derived}*{Proposition 5.1}, we see that any $\otimes$-generating line bundle is nowhere torsion. 
    On the other hand, a line bundle of degree zero on a smooth projective curve which is not torsion is nowhere torsion, but not $\otimes$-generating. 
\end{example}

For a proper variety, it follows from the main result of \cite{itoolander2025derived} that a nef and $\otimes$-generating line bundle is ample. On the other hand, any non-torsion degree $0$ line bundle on an elliptic curve is nef and nowhere torsion, but not ample. For nowhere torsion line bundles, we instead have the following result. 

\begin{lemma}
Suppose $\ecal{L}$ is nowhere torsion and semi-ample, then for any integral proper curve $C \subset X$,  $\deg(\ecal{L}|_C) > 0$. If $X$ is moreover proper, then $\ecal L$ is strictly nef and hence ample. 
\end{lemma}

\begin{proof}
Suppose $C \subset X$ is an integral proper curve with $\operatorname{deg}(\ecal{L}|_{C}) = 0$. Then there is $n > 0$ and $s \in \Gamma(X, \ecal{L}^{\otimes n})$ such that $s$ doesn't vanish identically on $C$. Then since $\ecal{L}$ has degree zero on $C$ we must have $C \subset X_s$. Then $\ecal{L}^{\otimes n}|_{C} \cong \ecal{O}_C$, which contradicts $\ecal L$ being nowhere torsion. Therefore, if $X$ is proper, then the semi-ample line bundle $\ecal L$ is strictly nef by definition and hence ample for example by \cite{campana2008strictly}*{Lemma 1.4} (whose proof does not need projectivity).
\end{proof}
\begin{example}
    We give an example of a line bundle $\ecal L$ on a non-proper variety that is nowhere torsion and semi-ample, but not ample. Let $Y$ be any variety whose structure sheaf is $\otimes$-generating but not ample (e.g. see \cite{itoolander2025derived}*{Lemma 5.30}). Note that $Y$ does not contain any proper integral curve $C$ since otherwise the structure sheaf of such a curve is $\otimes$-generating, which is absurd. Let $X = Y \times \bb P^1$ with projections $p: X \to \bb P^1$ and $q: X \to Y$ and set $\ecal L = p^*\ecal O_{\bb P^1}(1)$. First, note that $\ecal L$ is semi-ample as it is a pull-back of an ample line bundle, and $\ecal L$ is not ample since $\ecal L|_{Y\times \{x\}} \cong \ecal O_Y$ is not ample for any closed point $x \in \bb P^1$ by supposition. To see that $\ecal L$ is nowhere torsion, note that if $C \subset X$ is a proper integral curve, then $C \subset X \overset{q}{\to} Y$ needs to be constant, as $Y$ contains no proper curve, and hence $C$ is a fiber of the projection $q: X \to Y$. In particular, $\ecal L|_C \cong \ecal O_{\bb P^1}(1)$ is non-torsion. 
\end{example}

\begin{remark}
    By the abundance conjecture, if $X$ is a Gorenstein projective variety with canonical singularities, then we expect that if the canonical bundle $\omega_X$ is nowhere torsion and nef, then it is ample. On the other hand, see Example \ref{example-controlled canonical} for examples of normal projective Gorenstein varieties with canonical bundle being (i) nef and nowhere torsion but not strictly nef and (ii) strictly nef but not ample. Note that they do not contradict the abundance conjecture as these examples do not have canonical singularities. 
\end{remark}

\begin{lemma}
    Let $X$ and $Y$ be proper schemes over a field. Suppose that $\ecal L$ is a nef and nowhere torsion line bundle on $X$ and $\ecal M$ is a strictly nef line bundle on $Y$. Then, $\ecal L \boxtimes \ecal M$ is nowhere torsion. 
\end{lemma}
\begin{proof}
Take a finite morphism $(f,g): C \to X \times Y$ from a proper normal integral curve. If $g$ is constant, then $(f,g)^*(\ecal L \boxtimes \ecal M) \cong f^* \ecal L$ is not torsion since $\ecal L$ is nowhere torsion. If $g$ is nonconstant, then 
\[
\deg (f,g)^*(\ecal L \boxtimes \ecal M) = \deg f^* \ecal L + \deg g^* \ecal M > 0,
\]
as $\ecal L$ is nef and $\ecal M$ is strictly nef, and in particular $(f,g)^*(\ecal L \boxtimes \ecal M)$ is not torsion. 
\end{proof}
\begin{example}
    Let $\ecal L$ be a non-torsion line bundle of degree $0$ on an elliptic curve $E$. In particular, $\ecal L$ is nef and nowhere torsion. Hence, $\ecal L \boxtimes \ecal O_{\bb P^1}(1)$ is nowhere torsion on $E \times \bb P^1$. On the other hand, while $\ecal L\inv$ is also nef and nowhere torsion on $E$, $\ecal L \boxtimes \ecal L \inv$ on $E \times E$ restricts trivially to the diagonal $\Delta \subset E\times E$. 
\end{example}

\subsection{Reconstruction for schemes with nowhere torsion canonical bundle}
\label{subsection-reconstructionforschemes}

The main goal of this section is to prove Theorem \ref{theorem-reconstruction-for-schemes}, which is a generalization of \cite[Theorem 3.15]{FAVERO20121955} and is also our main reason for studying nowhere torsion line bundles. We follow the same proof strategy as the one in loc. cit. Let $X$ be a finite type, separated scheme over a field $k$. Let $\perf_{prop}(X) \subset \perf (X)$ be the full triangulated subcategory consisting of objects whose support is proper over $k$. 
The following is shown for varieties in \cite[Theorem 3.12]{FAVERO20121955}.
\begin{lemma} \label{lemma-characterizepropersupport}
    We have 
    \begin{align*}\perf_{prop}(X) &= \{K \in \perf (X) : \forall L \in \perf(X) ,  \operatorname{dim}_k \operatorname{Hom}(K, L)< \infty \} \\& = \{K \in \perf (X) : \forall L \in \perf(X) ,  \operatorname{dim}_k \operatorname{Hom}(L, K)< \infty \}  \qedhere 
    \end{align*}
\end{lemma}

We will need two additional lemmas to prove this. 

\begin{lemma}
\label{lemma-existsanaffinecurve}
    Suppose $X$ is not proper (and hence positive-dimensional) over $k$. Then $X$ contains a closed subscheme $C \subset X$ which is an affine curve. 
\end{lemma}

\begin{proof}
    If $X$ is quasi-projective over $k$, then $X \subsetneq \overline{X}$ is a dense open subscheme of a projective scheme $\overline{X}$ over $k$. Choose a closed point $y \in \overline{X}\setminus X$ and a curve $\overline{C} \subset \overline{X}$ containing $y$ which intersects $X$ non-trivially, and set $C = \overline{C} \cap X$. This works since a curve over a field is always either affine or projective (e.g., \cite[\href{https://stacks.math.columbia.edu/tag/0A28}{Tag 0A28}]{stacks-project}). Note moreover that the proof shows that if $U \subset X$ is a dense open, then we are able to choose $C$ so that $C \cap U \neq \emptyset$.

    In general, choose a proper, birational morphism $Y \to X$ with $Y$ quasi-projective. Then $Y$ cannot be proper since $X$ would be as well, being separated and the image of a proper scheme over $k$. Thus by the previous paragraph, there exists a closed subscheme $D \subset Y$ with $D$ an affine curve over $k$ and such that $D$ is not contained in the exceptional locus of $Y \to X$. 
    The composition $D \to X$ is proper and quasi-finite since $D$ is a curve, hence finite. The image $C$ of $D$ in $X$ is then a closed subscheme which is a curve. It is an affine curve being the image of an affine scheme under a finite morphism. 
\end{proof}

\begin{lemma}\label{lemma-cohomologywithaffinesupport}
    Let $\ecal{F} \in \operatorname{Coh}(X)$ have support equal to an affine curve $C$. Then $\operatorname{dim}_kH^0(X, \ecal{F}) = \infty$ and $H^q(X, \mathcal{F}) = 0$ for $q>0$.
\end{lemma}

\begin{proof}
    The coherent sheaf $\ecal{F}$ has a finite filtration $0 = \ecal{F}_0 \subset \ecal{F}_1 \subset \cdots \subset \ecal{F}_n = \ecal{F}$ such that for every $p$ there is a closed immersion $i_p : Z_p \to X$ with $Z_p$ integral and a non-zero coherent sheaf of ideals $\ecal{I}_p \subset \ecal{O}_{Z_p}$ such that $\ecal{F}_{p}/\ecal{F}_{p-1} \cong i_{p, *}\ecal{I}_p$ (\cite[\href{https://stacks.math.columbia.edu/tag/01YF}{Tag 01YF}]{stacks-project}). Then $\operatorname{Supp}\ecal{F} = \bigcup_{p = 1}^{n} Z_p$ hence $Z_p$ are closed subsets of $C$ and therefore affine for all $p$, and there exists a $p_0$ such that $Z_{p_0} = C$. By induction on $p$, one sees that $H^q(X, \ecal{F}_p) = 0$ for all $q > 0$. For $p = n$, this proves the vanishing statement. It also follows from this that there is a filtration
    $$
    0 = H^0(X, \ecal{F}_0) \subset H^0(X, \ecal{F}_1) \subset \cdots \subset H^0(X, \ecal{F}_n) = H^0(X, \ecal{F})
    $$
    with $H^0(X, \ecal{F}_p/ \ecal{F}_{p-1}) \cong H^0(Z_p, \ecal{I}_p)$. It will therefore be enough to show that $H^0(C, \ecal{I}_{p_0})$ is infinite-dimensional over $k$. We have an exact sequence
    $$
    0 \to H^0(C, \ecal{I}_{p_0}) \to H^0(C, \ecal{O}_C) \to H^0(C, \ecal{O}_C/\ecal{I}_{p_0}) \to 0
    $$
    since $C$ is affine. The last term is finite-dimensional since $\ecal{I}_{p_0} \subset \ecal{O}_C$ is a non-zero coherent ideal sheaf on a(n integral) curve so it defines a closed subscheme which is finite over $k$. The middle term is infinite-dimensional since $C$ is an affine curve. Thus we conclude.
\end{proof}

\begin{proof}[Proof of Lemma \ref{lemma-characterizepropersupport}]
    If $K \in \perf_{prop}(X)$, then $\operatorname{Hom}(K, L) = H^0(X, K^\vee \otimes^{\mathbb{L}} L)$ and the support of $K^\vee \otimes^{\mathbb{L}} L$ is proper over $k$ so this is finite-dimensional over $k$. The same proof shows $\operatorname{Hom}(L, K)$ 
    is finite-dimensional over $k$. Conversely, suppose $K \in \perf(X)$ has support which is not proper over $k$. Applying Lemma
\ref{lemma-existsanaffinecurve} to the support of $K$, we see that there is a closed subset $C \subset \operatorname{Supp}K$ such that $C$ is an affine curve over $k$. The category $D_{qc, C}(X)$ has a compact generator which is a perfect complex $L$ on $X$ with support equal to $C$ 
by \cite{rouquier2008dimensions}*{Theorem 6.8}. 
Then $K^\vee \otimes^{\mathbb{L}} L$ and $L^\vee \otimes ^{\mathbb{L}} K$ are both perfect complexes on $X$ with support equal to $C$, and so there are integers $i, j$ such that the coherent sheaves $H^i(K^\vee \otimes^{\mathbb{L}} L)$ and $H^j(L^\vee \otimes ^{\mathbb{L}} K)$ have support equal to $C$. But then 
$$
\operatorname{Hom}(K, L[i]) = H^i(X, K^{\vee} \otimes ^{\mathbb{L}}L) = H^0(X, H^i(K^\vee \otimes ^{\mathbb{L}}L)) 
$$
which is infinite-dimensional over $k$ by Lemma \ref{lemma-cohomologywithaffinesupport}. For the second equality, we use the vanishing statement in Lemma \ref{lemma-cohomologywithaffinesupport}. 
Similarly, 
$$
\operatorname{Hom}(L[-j], K) = H^0(X, H^j(L^\vee \otimes^{\mathbb{L}}K))
$$
is infinite-dimensional over $k$.
\end{proof}

The $k$-scheme $X$ has a dualizing complex $\omega_X^\bullet \in D^b_{\operatorname{Coh}}(X)$ which can be defined as $p^!\ecal{O}_{\operatorname{Spec}k}$ where $p : X \to \operatorname{Spec}k$ is the structure morphism, see \cite[\href{https://stacks.math.columbia.edu/tag/0AU3}{Tag 0AU3}]{stacks-project}. When $X$ is Cohen--Macaulay of pure dimension $d$ we have $\omega_X^\bullet = \omega_X[d]$ for some $\omega_X \in \operatorname{Coh}(X)$ called the dualizing sheaf on $X$. When $X$ is furthermore Gorenstein, $\omega_X$ is an invertible $\ecal{O}_X$-module.

\begin{lemma}
\label{lemma-serrefunctorpropersupport}
    Suppose $X$ is Gorenstein. Then $\perf_{prop}(X) \subset \perf(X)$ has a Serre functor which is isomorphic to $-\otimes ^{\mathbb{L}} \omega_X^\bullet$.
\end{lemma}

\begin{proof}
    The key point is that for $K \in D^b_{\operatorname{Coh}}(X)$ with support proper over $k$, there are natural-in-$K$ isomorphisms
    $$
    \operatorname{Hom}_X(K, \omega_X^\bullet) = \operatorname{Hom}_k(R\Gamma(X, K), k).
    $$
    This follows from the adjunction between $Rp_!$ and $p^!$, see \cite[\href{https://stacks.math.columbia.edu/tag/0G51}{Tag 0G51}]{stacks-project}, and the fact that $Rp_!K = Rp_*K = R\Gamma(X, K)$ for $K$ with proper support over $k$, which follows immediately from the definitions, see \cite[\href{https://stacks.math.columbia.edu/tag/0G4Z}{Tag 0G4Z}]{stacks-project}. Then since $X$ is Gorenstein, the functor $-\otimes^{\mathbb{L}} \omega_X^\bullet$ preserves $\perf_{prop}(X)$, and we have for $K, L$ in $\perf_{prop}(X)$,
    \[
    \operatorname{Hom}(K, L \otimes^{\mathbb{L}}\omega_X^\bullet) = \operatorname{Hom}(R \mathcal{H}om(L, K), \omega_X^\bullet) = \operatorname{Hom}(R\operatorname{Hom}(L, K) , k) = \operatorname{Hom}(L, K)^*.\qedhere
    \]
\end{proof}

\begin{lemma}
\label{lemma-perfectcoherentsheaf}
    Let $Y$ be a Noetherian scheme. Let $y \in Y$ be a closed point such that $\ecal{O}_{Y,y}$ is Cohen--Macaulay. Then there exists a coherent sheaf $\ecal{F}$ on $Y$ whose support is equal to $\{y\}$ and such that $\ecal{F}$ is quasi-isomorphic to a perfect complex on $Y$. Furthermore, we may choose $\ecal{F}$ so that the ring 
    $\operatorname{Hom}(\ecal{F}, \ecal{F})$ is commutative and has exactly two idempotents. 
\end{lemma}

\begin{proof}
    By definition, there is a regular sequence $f_1, \dots , f_d \in \ecal{O}_{Y,y}$ such that $V(f_1, \dots , f_d) = \{\mathfrak{m}_y\}$  as closed sets in $\operatorname{Spec}\ecal{O}_{Y,y}$. Note that the Koszul complex on the elements $f_1, \dots , f_d$ is a resolution of the $\ecal{O}_{Y,y}$-module $\ecal{O}_{Y,y}/(f_1, \dots , f_d)$. One checks that if $\ecal{F}$ is the pushforward of the structure sheaf along $i : \operatorname{Spec}\ecal{O}_{Y,y}/(f_1, \dots , f_d) \to Y$, then $\ecal{F}$ works. Indeed, the morphism $i$ is a regular closed immersion so $\ecal{F}$ is a perfect complex and its endomorphism ring is 
    $\ecal{O}_{Y,y}/(f_1, \dots , f_d)$ which is a commutative local ring. 
\end{proof}

\begin{lemma}
\label{lemma-samedimension}
    Let $X, Y$ be finite type, separated schemes over a field $k$ which are equidimensional and Gorenstein. Suppose there is a $k$-linear, exact equivalence of categories $\Phi: \perf_{prop}(X) \cong \perf_{prop}(Y)$. Then $\dim X = \dim Y$.
\end{lemma}

\begin{proof}
   The argument is the same as the proof of \cite[Proposition 4.1]{HuyBook}. We may assume $X$ is non-empty. By Lemma \ref{lemma-perfectcoherentsheaf}, since Gorenstein schemes are Cohen--Macaulay, there exists a perfect coherent sheaf $\ecal{F}$ on $X$ whose support is a single closed point $x$. Then since Serre functors commute with equivalences and $\ecal{F} \otimes \omega_X \cong \ecal{F}$, we have
   $$
  \Phi(\ecal{F})[\dim X] \cong  \Phi(\ecal{F} \otimes ^{\mathbb{L}} \omega_X)[\dim X] \cong \Phi(\ecal{F})\otimes^{\mathbb{L}}\omega_Y[\dim Y]
   $$
   of $\Phi(\ecal{F}) \cong \Phi(\ecal{F}) \otimes^{\mathbb{L}}\omega_Y[\dim Y - \dim X]$. If the left hand side has cohomology sheaves in degrees $[a,b]$ then the right hand side has cohomology sheaves in $[a + \dim Y - \dim X, b + \dim Y -\dim X]$ and hence $\dim X = \dim Y$.
\end{proof}

\begin{lemma}
\label{lemma-characterizingperf}
    Let $Y$ be a Noetherian scheme. An object $K \in D^b_{\operatorname{Coh}}(Y)$ is perfect if and only if for every $L \in D^b_{\operatorname{Coh}}(Y)$, we have
    $
    \operatorname{Ext}^i(K, L) = 0$ for $i \gg 0.
    $
\end{lemma}

\begin{proof}
    If $K$ is perfect, then $R\operatorname{Hom}(K, L) = R \Gamma( Y, K^\vee \otimes^{\mathbb{L}}L)$ is bounded since $K^\vee \otimes^{\mathbb{L}}L \in D^b_{\operatorname{Coh}}(Y)$. Conversely, if $K$ is not perfect, then by \cite[\href{https://stacks.math.columbia.edu/tag/068W}{Tag 068W}]{stacks-project}, there is a closed point $y \in Y$ such that, if $ : \operatorname{Spec}\kappa(y) \to Y$ is the inclusion, we have $Li^*(K) \in D(\operatorname{Spec}\kappa(y))$ is \emph{not} bounded below. Then let $L = Ri_*(\ecal{O})$ be the skyscraper sheaf at the origin placed in degree zero. We have
    $$
    \operatorname{Ext}^i(K, L) = \operatorname{Ext}^i_{\operatorname{Spec}\kappa(y)}(Li^*(K), \ecal{O}) = \operatorname{Hom}_{\operatorname{Spec}\kappa(y)}(H^{-i}(Li^*(K)), \ecal{O}) \neq 0
    $$
    for infinitely many $i > 0$, as needed.
\end{proof}

\begin{lemma}
\label{lemma-itsasheaf}
    Let $Y$ be a Noetherian scheme. Let $K \in D^b_{\operatorname{Coh}}(Y)$. Assume the following conditions hold.
    \begin{enumerate}
        \item The support of $K$ consists of finitely many closed points.
        \item We have $\operatorname{Ext}^i(K, K) = 0$ for all $i< 0$.
        \item The (possibly non-commutative) ring $\operatorname{Hom}(K, K)$ has exactly two central idempotents. 
    \end{enumerate}
    Then there are an integer $n$ and a coherent sheaf $\ecal{F}$ on $Y$ whose support consists of exactly one closed point such that $K \cong \ecal{F}[n]$.
\end{lemma}

\begin{proof}
    The proof is standard, see \cite[proof of Lemma 4.5]{HuyBook}.
\end{proof}

\begin{lemma}
\label{lemma-recognizingvectorbundles}
    Let $Y$ be a Noetherian scheme. Suppose given for every closed point $y \in Y$ a coherent sheaf $\ecal{F}_y$ on $Y$ whose support is equal to $\{y\}$. Then the following hold.
    \begin{enumerate}
        \item For $0 \neq K \in D^b_{\operatorname{Coh}}(Y),$ we have
        $$
        \operatorname{sup}\{n \in \mathbb{Z}:H^n(K) \neq 0\} = \operatorname{sup}\{n \in \mathbb{Z} : \exists y \in Y \text{ closed such that }\operatorname{Hom}(K, \ecal{F}_y[-n]) \neq 0\}.
        $$
        \item If $\ecal{E}$ is a vector bundle on $Y$ and $y \in Y$ is a closed point, then $\operatorname{Ext}^i(\ecal{E}, \ecal{F}_y)\neq 0$ only if $i = 0$.
        \item Suppose $Y$ is connected, $K \in \perf(Y)$, and for every closed point $y \in Y$, there is a unique integer $i$ such that $\operatorname{Ext}^i(K, \ecal{F}_y) \neq 0$. Then there exist an integer $n$ and vector bundle $\ecal{E}$ on $Y$ such that $K = \ecal{E}[n]$. (Hence by $(ii)$ the unique integer $i$ is equal to $n$). \qedhere
    \end{enumerate}
\end{lemma}

\begin{proof}
For (i), if $K \in D^{< n}(Y)$, then certainly $\operatorname{Hom}(K, \ecal{F}_y[-n]) = 
0$ for every $y \in Y$ closed. If $K \in D^{\leq n}(Y)$ but $H^n(K) \neq 0$, then there exists a closed point $y \in Y$ contained in the support of $H^n(K)$ and for this $y$ there are maps $H^n(K) \twoheadrightarrow \ecal{O}_y \hookrightarrow \ecal{F}_y$ where we use that $\ecal{O}_{y}$ is the unique simple object of the finite length category of coherent sheaves on $Y$ supported on $\{y\}$. Hence
$$
\operatorname{Hom}(K, \ecal{F}_y[-n]) = \operatorname{Hom}(\tau_{\geq n}K, \ecal{F}_y[-n]) = \operatorname{Hom}(H^n(K), \ecal{F}_y) \neq 0,
$$
as needed.

For (ii), if $y \in Y$ is any closed point, we have $\operatorname{Ext}^i(\ecal{E}, \ecal{F}_y) = H^i(X, \ecal{E}^\vee \otimes \ecal{F}_y) = 0$ unless $i = 0$.

For (iii), let $y \in Y$ be a closed point and $L \in \perf(Y)$ be an an object with $y \in \operatorname{Supp}L$. We will show by induction on the length of $\ecal{F}_y$ that the greatest (resp. the least) integer $i$ such that $\operatorname{Ext}^i(L, \ecal{F}_y) \neq 0$ is equal to the greatest (resp. the least) integer $i$ such that $\operatorname{Ext}^i(L, \ecal{O}_y) \neq 0$ (Note that these integers exist since $L$ is perfect and $y \in \operatorname{Supp}L$). Then we conclude by \cite[\href{https://stacks.math.columbia.edu/tag/068V}{Tag 068V}]{stacks-project} and the fact that a perfect complex with tor-amplitude in $[a,a]$ is a vector bundle sitting in degree $a$.

When the length of $\ecal{F}_y$ is one, we have $\ecal{F}_y \cong \ecal{O}_y$ so this is clear. Now suppose the length of $\ecal{F}_y$ is positive. Then there is a short exact sequence of coherent sheaves
$$
0 \to \ecal{G} \to \ecal{F}_y \to \ecal{O}_y \to 0
$$
supported at $\{y\}$ and the length of $\ecal{G}$ is one less than the length of $\ecal{F}_y$. There is a long exact sequence of $\operatorname{Ext}$ groups
$$
\cdots \to \operatorname{Ext}^{i-1}(L, \ecal{O}_y) \to \operatorname{Ext}^i(L, \ecal{G}) \to \operatorname{Ext}^i(L, \ecal{F}_y) \to \operatorname{Ext}^i(L, \ecal{O}_y) \to \operatorname{Ext}^{i+1}(L, \ecal{G}) \to \cdots .
$$
By induction, the greatest (resp. the least) $i$ such that $\operatorname{Ext}^i(L, \ecal{O}_y) \neq 0$ is equal to the greatest (resp. the least) $i$ such that $\operatorname{Ext}^i(L, \ecal{G}) \neq 0$, and we see from the exact sequence that this must also be the greatest (resp. the lest) $i$ such that $\operatorname{Ext}^i(L, \ecal{F}_y) \neq 0$, as needed.
\end{proof}

Compare the following to \cite[Lemma 3.5]{FAVERO20121955}.

\begin{prop}
\label{prop-pointstopoints}
    Suppose $X, Y$ are finite type, separated, connected 
    schemes over a field $k$. 
    Let $\Phi: D^b_{\operatorname{Coh}}(Y) \to D^b_{\operatorname{Coh}}(X)$ be a $k$-linear, exact equivalence of categories. 
    Suppose given for every closed point $y \in Y$ a coherent sheaf $\ecal{F}_y$ on $Y$ whose support is equal to $\{y\}$. Assume that for every closed point $y \in Y$, there exists a coherent sheaf $\ecal{G}_y$ on $X$ whose support consists of a single closed point and an integer $n_y$ such that $\Phi(\ecal{F}_y) \cong \ecal{G}_y[n_y]$. Then there exist an isomorphism $f : Y \to X$, $\ecal{L} \in \operatorname{Pic}(X)$, and integer $n$ such that $\Phi(\ecal{F}) \cong Rf_*(\ecal{F}) \otimes^{\mathbb{L}}\ecal{L}[n]$ for every coherent sheaf $\ecal{F}$ on $Y$. If additionally, $\dim X> 0$ and $X$ is proper  over $k$, then there is an isomorphism of functors $\Phi \cong Rf_*(-)\otimes^{\mathbb{L}}\ecal{L}[n]$.
\end{prop}

\begin{proof}
  For a closed point $y \in Y$, write $x = x(y)$ for the closed point of $X$ such that the support of $\ecal{G}_y$ is equal to $\{x\}$. By Lemma \ref{lemma-recognizingvectorbundles}, every object of $D^b_{\operatorname{Coh}}(Y)$ has a non-zero map to a shift of one of the sheaves $\ecal{F}_y$, hence the same is true for the family of sheaves $\ecal{G}_y$ on $X$, and we deduce that every closed point $x$ occurs as $x (y)$ for some closed point $y \in Y$ (otherwise consider the coherent sheaf $\ecal{O}_x$). In fact, $y \mapsto x(y)$ is a bijection from the set of closed points of $Y$ to the set of closed points at $x$ because if $y_1 \neq y_2$ are closed points of $Y$, then $\operatorname{Ext}^i(\ecal{F}_{y_1}, \ecal{F}_{y_2})= 0$ for all integers $i$, hence $\operatorname{Ext}^i(\ecal{G}_{x(y_1)}, \ecal{G}_{x(y_2)})= 0$ for all integers $i$, and from this we see that $x(y_1) \neq x(y_2)$ because any non-zero coherent sheaves supported at a single closed point of $X$ map to one another. 

  We claim that all the integers $n_y$ are equal. For this, note that $\Phi$ takes perfect complexes to perfect complexes by Lemma \ref{lemma-characterizingperf}, and so $\Phi(\ecal{O}_Y)$ is a perfect complex on $X$. Since $\Phi$ is fully faithful, it satisfies for every closed point $y \in Y$,
  $$
  \operatorname{Ext}^{i+n_y}(\Phi(\ecal{O}_Y), \ecal{G}_{x(y)}) = \operatorname{Ext}^i(\Phi(\ecal{O}_Y), \ecal{G}_{x(y)}[n_y]) = \operatorname{Ext}^i(\ecal{O}_Y, \ecal{F}_y) \neq 0
  $$
  if and only if $i = 0$. We conclude by part (iii) of Lemma \ref{lemma-recognizingvectorbundles} that there are a vector bundle $\ecal{E}$ on $X$ and integer $n$ such that $\Phi(\ecal{O}_Y) \cong \ecal{E}[n]$, and furthermore that all the integers $n_y$ are equal to $n$, as claimed.

    \emph{Now to prove the result, we may and do replace $\Phi$ by its composition with the shift by $-n$ functor to assume $n_y = 0$ for all closed points $y \in Y$.}

  Then by part (i) of Lemma \ref{lemma-recognizingvectorbundles}, we see using fully faithfullness of $\Phi$ that for every object $0 \neq K \in D^b_{\operatorname{Coh}}(Y)$, we have
  \begin{equation}
  \label{equn-sametopcohomology}
  \operatorname{sup}\{n \in \mathbb{Z} : H^n(K) \neq 0\} =   \operatorname{sup}\{n \in \mathbb{Z} : H^n(\Phi(K)) \neq 0\}.
  \end{equation}
  We claim that for an object $K \in D^b_{\operatorname{Coh}}(Y)$, we have $K \in \operatorname{Coh}(Y)[0] \subset D^b_{\operatorname{Coh}}(Y)$ if and only if $\Phi(K) \in \operatorname{Coh}(X)[0] \subset D^b_{\operatorname{Coh}}(X)$. The proof is the same as ``Step 2'' of the proof of Proposition 3.2.3 of \cite{olander2022resolutions}. If $0 \neq \ecal{F}$ is a coherent sheaf on $Y$, then we know $\Phi(\ecal{F}) \in D^{\leq 0}(X)$ and $H^0(\Phi(\ecal{F})) \neq 0$. Hence there is a distinguished triangle 
  $$
  \tau_{<0}\Phi(\ecal{F}) \to \Phi(\ecal{F}) \to H^0(\Phi(\ecal{F})) \to \tau_{<0}\Phi(\ecal{F})[1]
  $$
  and if $\Phi(\ecal{F})$ is not a coherent sheaf in degree zero, then the left arrow is not zero. Applying $\Phi^{-1}$, we obtain a non-zero morphism $\Phi^{-1}(\tau_{<0}\Phi(\ecal{F})) \to \ecal{F}$, but $\Phi^{-1}(\tau_{<0}\Phi(\ecal{F}))  \in D^{<0}(Y)$ by (\ref{equn-sametopcohomology}), so this is a contradiction. Conversely, if $K \in D^b_{\operatorname{Coh}}(Y)$ is such that $\Phi(K)$ is a coherent sheaf sitting in degree zero, then we see that $K \in D^{\leq 0}(Y)$, hence there is a distinguished triangle
  $$
  \tau_{<0}K \to K \to H^0(K) \to \tau_{<0}K[1],
  $$
  and if $K$ is not a coherent sheaf sitting in degree zero, then the first arrow is not zero. Applying $\Phi$, we obtain a non-zero morphism $\Phi(\tau_{<0}K) \to \Phi(K)$ from an object in degrees $<0$ to an object in degree zero, a contradiction.

  Thus $\Phi$ and $\Phi^{-1}$ induce quasi-inverse equivalences of categories $\operatorname{Coh}(Y) \leftrightarrow \operatorname{Coh}(X)$. By Gabriel's Theorem (see \cite[\href{https://stacks.math.columbia.edu/tag/0GPK}{Tags 0GPK and 0GPL}]{stacks-project} or \cite{Gabriel1962}), the restriction of $\Phi$ to $\operatorname{Coh}(Y)$ is of the form $f_*(-)\otimes \ecal{L}$ for some isomorphism $f : Y \to X$ of $k$-schemes and line bundle $\ecal{L}$ on $X$, as needed. 

  If in addition, $X$ is proper over $k$ of dimension $>0$, then so is $Y$ since $X \cong Y$, so by \cite[Lemma 3.3.2]{olander2022resolutions}, the category of coherent sheaves on $Y$ has an almost ample set. Thus by \cite[Proposition 3.3]{CanonacoStellari2014}, if $F, G$ are two exact functors from $D^b_{\operatorname{Coh}}(Y)$ to a triangulated category whose restrictions to the full subcategory $\operatorname{Coh}(Y)[0] \subset D^b_{\operatorname{Coh}}(Y)$ are isomorphic, and $F$ is fully faithful, then $F \cong G$. Applying this to the functors $\Phi$ and $Rf_*(-) \otimes ^{\mathbb{L}}\ecal{L}$ completes the proof. 
\end{proof}

\begin{theorem}\label{theorem-reconstruction-for-schemes}
    Suppose $X, Y$ are finite type, separated, connected, Gorenstein schemes over a field $k$. Suppose $\omega_X$ is a nowhere torsion line bundle. Let $\Phi: D^b_{\operatorname{Coh}}(Y) \to D^b_{\operatorname{Coh}}(X)$ be a $k$-linear, exact equivalence of categories. Then there exist an isomorphism $f : Y \to X$, $\ecal{L} \in \operatorname{Pic}(X)$, and integer $n$ such that $\Phi(\ecal{F}) \cong Rf_*(\ecal{F}) \otimes^{\mathbb{L}}\ecal{L}[n]$ for every coherent sheaf $\ecal{F}$ on $Y$. If additionally, $\dim X> 0$ and $X$ is proper  over $k$, then there is an isomorphism of functors $\Phi \cong Rf_*(-)\otimes^{\mathbb{L}}\ecal{L}[n]$.
\end{theorem}
\begin{proof}
 First, by Lemmas \ref{lemma-characterizingperf} and \ref{lemma-characterizepropersupport}, we see that $\Phi$ restricts to $k$-linear, exact equivalences of categories $\perf(Y) \to \perf (X)$ and $\perf_{prop}(Y) \to \perf_{prop}(X)$. By Lemma \ref{lemma-samedimension}, we have $\dim X = \dim Y$ (note that Gorenstein schemes are Cohen--Macaulay and connected Cohen--Macaulay schemes are equidimensional). Write $d$ for $\dim X = \dim Y$.
    
    By Lemma \ref{lemma-serrefunctorpropersupport} and the fact that Serre functors commute with equivalences, we obtain that 
    $$
    \Phi(K \otimes^{\mathbb{L}}\omega_Y)[d] \cong \Phi(K) \otimes^{\mathbb{L}} \omega_X[d],
    $$
  or equivalently 
  $$
  \Phi(K \otimes^{\mathbb{L}}\omega_Y) \cong \Phi(K) \otimes^{\mathbb{L}} \omega_X
  $$
  whenever $K \in \perf_{prop}(Y)$. For every closed point $y \in Y$, choose a perfect coherent sheaf $\ecal{F}_y$ on $Y$ whose support is equal to $\{y\}$ and whose endomorphism ring has exactly two central idempotents, which exists by Lemma \ref{lemma-perfectcoherentsheaf}. Then applying the equation above with $K = \ecal{F}_y$, we obtain that $\Phi(\ecal{F}_y) \cong \Phi(\ecal{F}_y) \otimes^{\mathbb{L}} \omega_X$, and hence $H^i(\Phi(\ecal{F}_y)) \cong H^i(\Phi(\ecal{F}_y)) \otimes \omega_X$ for every integer $i$. Since $\omega_X$ is nowhere torsion, it follows that the support of $\Phi(\ecal{F}_y)$ is finite. Then since $\Phi$ is fully faithful, we may apply Lemma \ref{lemma-itsasheaf} to the object $\Phi(\ecal{F}_y) \in D^b_{\operatorname{Coh}}(X)$ and deduce that $\Phi(\ecal{F}_y) \cong \ecal{G}_y[n_y]$ for some integer $n_y$ and coherent sheaf $\ecal{G}_y$ supported at a single closed point $x = x(y) \in X$. Then we conclude by Proposition \ref{prop-pointstopoints}.
\end{proof}
{\begin{remark}
    Note that this result is compatible with the conjecture that the existence of flops implies derived equivalence \cite{bondal1995semiorthogonal}*{Conjecture in p.40} since if we have a floppable curve $C \subset X$, then $\omega_X|_C$ is torsion and hence $\omega_X$ is somewhere torsion. Indeed, by the usual definition of flops (e.g. \cite{kollar1989flops}*{Definition 2.1}, which is the convention used in \cite{bondal1995semiorthogonal}), a flopping contraction $f:X \to Y$ (with $Y$ normal $\bb Q$-Gorenstein) is in particular crepant in the sense that $f^*K_Y$ is $\bb Q$-linearly equivalent to $K_X$. Hence, taking $m$ so that $mK_Y$ is Cartier, we have $\omega_X^{\otimes m}|_C \cong f^*\ecal O_Y(mK_Y)|_C \cong \ecal O_C$, as $C$ is contracted by $f$.
\end{remark}
We can also recover $X$ from $\perf$ in the situation of Theorem \ref{theorem-reconstruction-for-schemes} by the following observation. We quickly recall the following results. See the original paper \cite{krause2020completing} by Henning Krause with appendices by Tobias Barthel and Bernhard Keller for full expositions. 
\begin{construction}\label{construction-sequential-completion} First, we recall notions regarding sequential completion. 
\begin{itemize}
    \item For a category $\cal C$, a sequence $X_0 \to X_1 \to \cdots$ of objects in $\cal C$ is a \textit{Cauchy sequence} if for any $C \in \cal C$, there exists $N_C \in \bb N$ such that for any $j \geq i > N_C$, $\hom_\cal C(C, X_i) \to \hom_\cal C(C,X_j)$ is an isomorphism.
    \item We say a morphism $F_\bullet:X_\bullet \to Y_\bullet$ of Cauchy sequences is \textit{eventually invertible} if for any $C \in \cal C$, there exists $n_{C,F} \in \bb N$ such that for any $i > n_{C,F}$, $\hom_{\cal C}(C, X_i) \to \hom_\cal C(C,Y_i)$ is an isomorphism.
    \item Define the \textit{sequential completion} $\hat {\cal C}$ of $\cal C$ to be the localization of the category of Cauchy sequences in $\cal C$ with respect to eventually invertible morphisms. The canonical functor $\cal C \to \hat{\cal C}$ sending an object $X$ of $\cal C$ to the constant sequence $X \overset{\id}{\to} X \overset{\id}{\to}\cdots$ and the canonical functor $\hat {\cal C} \to \func(\cal C^\op, \sf{Set})$ taking Cauchy sequences to their colimits (in $\func(\cal C^\op, \sf{Set})$) are fully faithful by \cite{krause2020completing}*{Proposition 2.4}. 
\end{itemize}
When $\cal C$ is an additive category, the sequential completion $\hat{\cal C}$ is canonically additive and the canonical embedding $\cal C \inj \hat{\cal C}$ is also additive by \cite{krause2020completing}*{Lemma 3.1}. Similarly, the canonical embedding $\hat{\cal C} \inj \func(\cal C^\op, \set)$ factors through the additive fully faithful embedding $\hat{\cal C} \inj \func_\bb Z(\cal C^\op, \ab)$ (by taking colimits) noting the forgetful functor $\ab \to \set$ preserves filtered colimits and the canonical functor $\func_\bb Z(\cal C^\op, \ab) \to \func(\cal C^\op, \sf{Set})$ is fully faithful.  

Now, \cite{krause2020completing}*{Theorem 8.2} claims that for a Noetherian scheme $Z$ (over $\spec \bb Z$), the image of the canonical embedding $$\widehat{\perf(Z)}^b \inj \func_\bb Z(\perf(Z)^\op, \ab)$$ (taking Cauchy sequences to their colimits) is identified with the image of the fully faithful embedding $$D^b_{\operatorname{Coh}}(Z) \inj \func_\bb Z(\perf(Z)^\op, \ab), \quad \ecal F \mapsto \hom_{D^b_{\operatorname{Coh}}(Z)}(-, \ecal F)|_{\perf(Z)}$$ (as triangulated categories). Here, $\widehat{\perf(Z)}^b$ denotes the additive subcategory of $\widehat{\perf(Z)}$ on objects with "bounded cohomology", which is naturally a triangulated category by \cite{krause2020completing}*{Lemma 4.6, Theorem 4.7, Definition 5.1, Lemma 7.3, Proposition 8.1} and  whose definition a priori requires an additional choice of a cohomology functor $H: \perf(Z) \to \ab$, but by \cite{krause2020completing}*{Theorem B.1} (see also \cite{krause2020completing}*{Corollary B.2}), when $Z$ is a separated Noetherian scheme, the objects in $\widehat{\perf(Z)}^b$ admit an intrinsic characterization in terms of the triangulated category structure of $\perf(Z)$. Therefore, if we have an additive exact equivalence $\Phi: \perf(X) \to \perf(Y)$ for separated Noetherian schemes $X,Y$, then it canonically lifts to an additive exact equivalence $\hat \Phi: D^b_{\operatorname{Coh}}(X) \simeq D^b_{\operatorname{Coh}}(Y)$, which is naturally isomorphic to the restriction of the additive equivalence $\func_\bb Z(\perf(X)^\op, \ab) \to \func_\bb Z(\perf(Y)^\op, \ab)$. 
\end{construction}

\begin{lemma}\label{lemma-k-linear-lift}
    Let $X, Y$ be separated Noetherian schemes over a field $k$. Then any $k$-linear exact equivalence $\Phi:\perf(X) \to \perf(Y)$ lifts to a $k$-linear exact equivalence $\tilde \Phi: D^b_{\operatorname{Coh}}(X) \to D^b_{\operatorname{Coh}}(Y)$. Moreover, if $\tilde \Phi$ is a Fourier--Mukai transform with kernel $\ecal K \in D_{qc}(X \times Y)$, then so is $\Phi$ with the same kernel. 
\end{lemma}
\begin{proof}
    The latter claim clearly follows from the former, so it suffices to show the additive lift $\hat \Phi$ of $\Phi$ in Construction \ref{construction-sequential-completion} is a $k$-linear lift, which boils down to showing that the embedding $\perf(X) \inj \widehat{\perf(X)}$, the identification $\widehat{\perf(X)}^b \simeq D^b_{\operatorname{Coh}}(X)$, and the induced equivalence $\hat{\Phi}: \widehat{\perf(X)}^b \to \widehat{\perf(Y)}^b$ are all $k$-linear. Indeed, noting the additive case, the embedding is clearly $k$-linear, $\widehat{\perf(X)}^b$ and $D^b_{\operatorname{Coh}}(X)$ are equivalent as $k$-linear subcategories of $\func_k(\perf(X), \sf{Vect}_k)$, and the induced functor is the restriction of the $k$-linear equivalence $\func_k(\perf(X)^\op, \sf{Vect}_k) \to \func_k(\perf(Y)^\op, \sf{Vect}_k)$. 
\end{proof}
\begin{corollary}
    Suppose $X, Y$ are finite type, separated, connected, Gorenstein schemes over a field $k$. Suppose $\omega_X$ is a nowhere torsion line bundle. Let $\Phi: \perf(Y) \to \perf(X)$ be a $k$-linear, exact equivalence of categories. Then there exist an isomorphism $f : Y \to X$, $\ecal{L} \in \operatorname{Pic}(X)$, and integer $n$ such that $\Phi(\ecal{F}) \cong Rf_*(\ecal{F}) \otimes^{\mathbb{L}}\ecal{L}[n]$ for every coherent sheaf $\ecal{F}$ quasi-isomorphic to a perfect complex on $Y$. If additionally, $\dim X> 0$ and $X$ is proper  over $k$, then there is an isomorphism of functors $\Phi \cong Rf_*(-)\otimes^{\mathbb{L}}\ecal{L}[n]$.
\end{corollary}
\begin{proof}
    Immediately follows from Theorem \ref{theorem-reconstruction-for-schemes} and Lemma \ref{lemma-k-linear-lift}. 
\end{proof}

\subsection{Examples} 
\label{subsection-examples} In this section, we work over an algebraically closed field $k$ of arbitrary characteristic, though we sometimes assume $k$ has characteristic zero. We give some examples of nowhere torsion line bundles on varieties over $k$. Especially interesting to us are examples of Gorenstein varieties with nowhere torsion canonical bundle. 
Two of these sets of examples were essentially already known: The blow up of $\mathbb{P}^2_k$ in $n$ very general points on a smooth cubic $E \subset \mathbb{P}^2_k$ has nowhere torsion canonical bundle (see \ref{subsubsection-blowups}), as does a $\mathbb{P}^1$-bundle $\mathbb{P}_E(\ecal{E})$ over an elliptic curve $E$ as long as $\ecal{E} \not \cong \ecal{L} \oplus \ecal{M}$ for line bundles $\ecal{L}, \ecal{M}$ such that $\ecal{L}^{-1} \otimes \ecal{M}$ is torsion (see \ref{subsubsection-p1bundles}). These results were first obtained in \cite{Krahblowups} and \cite{uehara2017fourier}, respectively, but they did not fit into the framework of any of the previously known versions of the Bondal--Orlov Reconstruction Theorem, so the authors had to give their own arguments that such varieties could be reconstructed from their derived categories. We are additionally able to extend Uehara's result to some abelian varieties, giving some new varieties which can be reconstructed from their derived categories. 

We also give new examples of singular projective Gorenstein varieties which can be reconstructed from their derived categories. In \ref{subsubsection-abundance}, we give examples of Gorenstein varieties whose canonical bundles are nowhere torsion and nef but not semi-ample. Our examples are all singular, and indeed, the abundance conjecture predicts that any such variety must be singular. We actually construct many such examples -- one for each pair $(B, \ecal{L})$ where $B$ is simple abelian variety over $k$ and $\ecal{L}$ is an element of $\operatorname{Pic}^0(B)$ which is not torsion, see Example \ref{example-controlled canonical}. Bootstrapping, we show in Lemma \ref{lemma-blowupofminimalnowhere torsion} that a blowup of any of these varieties $X$ in a finite set of non-singular points also has nowhere torsion canonical bundle, thus producing yet more examples. 

Finally, one might ask if any proper variety over $k$ which admits a big line bundle also admits a nowhere torsion line bundle (or even a $\otimes$-generating line bundle). In \ref{subsubsection-yesbigbutnonowheretorsion}, we use Ferrand's pinching construction \cite{Ferrand-Conducteur} to show this need not be the case.

\subsubsection{Blow-ups} 
\label{subsubsection-blowups}

We first ask for which blowups $X$ of $\mathbb{P}^2$ in finitely many points is it true that $\omega_X$ is nowhere torsion? 

\begin{example}
\begin{enumerate}
    \item If $X$ is the blow up of $\mathbb{P}^2$ in $\leq 8$ general points, then $\omega_X^{-1}$ is ample, hence $\omega_X$ is nowhere torsion.
    \item If $X$ is the blow up of $\mathbb{P}^2$ in $n$ points, where either (a) all the points lie on a line and $n > 3$ or (b) all the points lie on a conic and $n > 6$, then $\omega_X^{-1}$ has negative degree on the strict transform of the line (resp. the conic), but $\omega_X^{-1}$ is big and $\otimes$-generating, hence nowhere torsion, see \cite[Theorem 1.5]{itoolander2025derived}.
    \item  It is proved in \cite[Lemma 3.2]{Krahblowups} that if $k = \mathbb{C}$ and $X$ is the blow up of $\mathbb{P}^2$ in $9$ very general points, then $\omega_X$ is nowhere torsion. Here one can see that neither $\omega_X$ nor $\omega_X^{-1}$ is big, hence $\omega_X$ is not $\otimes$-generating. The authors contemplate in \cite[Remark 3.3]{Krahblowups} whether it could be true that $\omega_X$ is nowhere torsion when $X$ is a blow up of $\mathbb{P}^2$ in $n >9$ very general points as well. 
    \item There exist blow ups $X$ of $\mathbb{P}^2$ in $9$ points for which $\omega_X$ is somewhere torsion. Let $E \subset \mathbb{P}^2$ be a smooth cubic and choose points $p_1, \dots , p_9$ on $E$ such that $3h-\sum_{i = 1}^9 p_i$ is a torsion element of the class group of $E$, where $h$ is the class of $\ecal{O}_E(1)$. Then if $X$ is the blow up of $\mathbb{P}^2$ in $p_1, \dots , p_9$, and $\widetilde{E} \subset X$ is the strict transform of $E$, then $\omega_X|_{\widetilde{E}} = \ecal{O}_E(-3h+\sum p_i)$ is torsion in the Picard group of $E$. \qedhere
\end{enumerate}
\end{example}

We prove a very slight generalization of \cite[Lemma 3.2]{Krahblowups} here, using the same argument.

\begin{lemma}
    Let $E \subset \mathbb{P}^2$ be a smooth cubic. Let $\pi : X \to \mathbb{P}^2$ be the blow up in very general points $p_1, \dots, p_n$ on $E$ where $n$ is any positive integer. Then $\omega_X$ is nowhere torsion, so $X$ is determined by $D^b_{\operatorname{Coh}}(X)$.
\end{lemma}

\begin{proof}
    Note that $E$ has genus $1$ by the adjunction formula, and the strict transform $\widetilde{E}$ of $E$ is an anti-canonical divisor on $X$. Let $C \subset X$ be a curve such that $(\omega_X \cdot C) = 0 = (\widetilde{E} \cdot C)$. Assume $C \neq \widetilde{E}$. Then $C$ and $\widetilde{E}$ are disjoint, hence $\pi(C) \cap E \subset \{p_1, \dots , p_n\}$. This gives a relation $d \cdot h = \sum a_i p_i$ in the class group of $E$ where $d, a_i$ are non-negative integers with $d > 0$ and $h$ is the class of $\ecal{O}_E(1)$. The set of tuples $(q_1, \dots , q_n) \in E^n$ such that there exist integers $d, a_i \geq 0$ with $d>0$ such that $d h = \sum a_i q_i$ in the class group of $E$ is a countable union (indexed by tuples $(d, a_1, \dots , a_n) \in \mathbb{Z}_{>0} \times \mathbb{Z}_{\geq 0}^n$ say) of proper closed subsets of $E^n$. Since $(p_1, \dots , p_n)$ are very general, there is no such relation, so we have a contradiction. 

    We have shown that for every curve $C \subset X$ such that $C \neq \widetilde{E}$, we have $(\omega_X \cdot C) = - (\widetilde{E} \cdot C) < 0$. It remains to show $\omega_X|_{\widetilde{E}}$ is not torsion. If $n \neq 9$ there is nothing to show since $(\omega_X \cdot \widetilde{E}) = - (\omega_X^2) \neq 0$. If $n = 9$ (the case handled by Hu and Krah), we have $\omega_X |_{\widetilde{E}} \cong \ecal{O}_E(-3h+\sum p_i)$. Since $p_1, \dots , p_n$ are very general, there are only countably many torsion points on an elliptic curve, and the fibers of $E^n \to \operatorname{Pic}^n_E$ are proper closed subsets of $E^n$, this is not a torsion line bundle on $E$.
\end{proof}



In \cite{itoolander2025derived}*{Theorem 5.26}, we proved that the blow up of a smooth proper variety with ample canonical bundle in a finite set of points has $\otimes$-generating canonical bundle, and is therefore determined by its derived category. The following generalizes this.

\begin{lemma}
\label{lemma-blowupofminimalnowhere torsion}
    Assume $X$ is a proper Gorenstein variety over a field. If $\omega_X$ is nef and nowhere torsion, and $\pi : Y \to X$ is a blowup of $X$ in a finite set of 
    regular points, then $\omega_Y$ is nowhere torsion. In particular, $Y$ is determined by its derived category. 
\end{lemma}

\begin{remark}
    The abundance conjecture predicts that a smooth proper variety with nef canonical bundle has semi-ample canonical bundle, and a semi-ample and nowhere torsion line bundle on a proper variety is ample. However, when $X$ is allowed to be Gorenstein but singular, it is possible for $\omega_X$ to not be ample, and in these cases Lemma \ref{lemma-blowupofminimalnowhere torsion} gives new examples of varieties which are determined by their derived categories. See Example \ref{example-controlled canonical} below.
\end{remark}

\begin{proof}
    Let $C \subset Y$ be a curve. Suppose $\omega_Y|_{C}$ has degree zero. We have a formula $\omega_Y = \pi^*\omega_X(\sum (\operatorname{dim}(X)-1) E_i)$ where $E_i$ are the exceptional divisors of the blowup. Thus if $C \subset E_i$ for some $i$ then $\omega_Y|_{C}$ is anti-ample since $\omega_Y|_{E_i} = \ecal{O}_{E_i}(1-\operatorname{dim}(X))$ is, a contradiction. Hence $C \to X$ is birational onto its image. If $C \cap E_i \neq \emptyset$ for some $i$ then $C \cdot E_i > 0$ and from this we see $\omega_Y|_{C}$ has positive degree, a contradiction. Hence $C$ is disjoint from the $E_i$ and maps isomorphically to its image in $X$. But then we see $\omega_Y|_{C} = \omega_X|_{C}$ and since $\omega_X$ is nowhere torsion we see $\omega_Y|_{C}$ is not torsion. We have proven that $\omega_Y$ is nowhere torsion. 
\end{proof}
We may drop the nefness assumption if we blow up points more carefully.  
\begin{lemma}
    Assume $X$ is a Gorenstein proper variety over a field. If $\omega_X$ is nowhere torsion and $\pi: Y \to X$ is a blowup in regular closed points $p_1, \dots , p_n$ which are not in the stable base locus of $\omega_X$, then $\omega_Y$ is nowhere torsion. 
\end{lemma}

Recall that the stable base locus of a line bundle $\ecal{L}$ on a scheme $X$ is the closed complement of the open set
$$
\bigcup _{n > 0, s \in \Gamma(X, \ecal{L}^{\otimes n})}X_s.
$$

\begin{proof}
    Again we have $\omega_Y = \pi^*\omega_X \otimes \ecal{O}((\operatorname{dim}(X) - 1)\cdot \sum E_i)$ where $E_i = \pi^{-1}(p_i)$. Suppose there is an integral curve $C \subset Y$ with normalization $C^\nu$ such that $(C^\nu \to Y)^*\omega_Y$ is torsion. The same argument as in Lemma \ref{lemma-blowupofminimalnowhere torsion} shows $C$ cannot be disjoint from all the $E_i$, and also $C$ cannot be contained in any $E_i$. Therefore, $C \cdot E_i \geq 0$ for all $i$ and $C \cdot E_j > 0$ for some $j$. By assumption, there is an integer $r > 0$ and a section $s \in \Gamma(X, \omega_X^{\otimes r})$ such that $p_j \in X_s$. Then $\operatorname{deg}(\pi^*\omega_X^{\otimes r}|_ C)\geq 0$ since $\pi^*s$ is not identically zero on $C$. This implies $\operatorname{deg}(\omega_X|_{C}) \geq 0$. Putting all this together yields $\operatorname{deg}(\omega_Y |_{C}) > 0$, a contradiction.
\end{proof}

\subsubsection{$\bb P^1$-bundles over abelian varieties} 
\label{subsubsection-p1bundles}


The following generalizes \cite{itoolander2025derived}*{\href{https://arxiv.org/pdf/2507.17681v2}{Example 5.15}} for elliptic curves to (simple) abelian varieties. 
Recall the following computations from \cite{itoolander2025derived}*{\S5.4}. 
\begin{construction}\label{construction-p1bundleoverabelian}
    Let $A$ be a proper Gorenstein variety with trivial canonical bundle $\omega_A \cong \ecal O_A$ and let $\ecal L$ be a line bundle on $A$. Let $\pi:X = \bb P_A(\ecal O_A \oplus \ecal L) \to A$ be the projective bundle. Since $\omega_A \cong \ecal O_A$, we have
    \[
    \omega_X \cong \ecal O_X(-2) \otimes \pi^* \ecal L. 
    \]
    Let $s_\infty \in H^0(X, \ecal O_X(1))$ be the section corresponding to $\ecal O_X =  \pi^* \ecal O_A \to \pi^*(\ecal O_A \oplus \ecal L) \to \ecal O_X(1),$
    and let $s_0 \in H^0(X, \ecal O_X(1) \otimes \pi^*\ecal L\inv)$ be the section corresponding to $ \pi^*\ecal L \to \pi^*(\ecal O_A \oplus \ecal L) \to \ecal O_X(1).$
    By \cite{itoolander2025derived}*{Lemma 5.27}, the zero scheme $Z(s_0)$ of $s_0$ agrees with the section 
    \[
    Z(s_0) = S_0 \subset X
    \]
    corresponding to the quotient $\ecal O_A \oplus \ecal L \to \ecal O_A$, and similarly, we have
    \[
    Z(s_\infty) = S_\infty \subset X 
    \]
    where $S_\infty$ corresponds to $\ecal O_A \oplus \ecal L \surj \ecal L$. Therefore, we have $s_0 \otimes s_\infty \in H^0(X, \omega_X\inv)$ with $Z(s_0\otimes s_\infty) = S_0 \sqcup S_\infty$. Moreover, we have
    \[
    X_{s_0\otimes s_\infty} = X \setminus (S_0 \sqcup S_\infty) \cong \operatorname{\cal S \it{pec}}_A ( \oplus_{n \in \bb Z} \ecal L^{\otimes n}) \overset{p}{\to} A
    \]
    is the punctured total space of $\ecal L$ over $A$. In particular, we have $ \ecal O_{X_{s_0\otimes s_\infty}} \cong p^* \ecal L.$
\end{construction}
\begin{prop}\label{prop-main-projective-bundle-over-abelian}
    Use the notation of Construction \ref{construction-p1bundleoverabelian}. 
    \begin{enumerate} 
        \item If $\omega_X$ is nowhere torsion, then $\ecal L$ is not torsion. 
        \item If $A$ is a simple abelian variety and if either $\ecal L \in \pic^0(A)$ or the Picard number of $A$ is $1$, then we have that $\omega_X$ is nowhere torsion $\iff$ $\ecal L$ is not torsion. \qedhere
    \end{enumerate}
\end{prop}
\begin{proof}
    For (i), if $\omega_X$ is nowhere torsion, then $\ecal L \cong \omega_X|_{S_0}$ is not torsion.

    For (ii), suppose $A$ is a simple abelian variety and $\omega_X$ is somewhere torsion. Then, there is a smooth projective curve $C$ with a finite morphism $g:C \to X$ such that $g^*\omega_{X}$ is torsion. First, assume that $h:=\pi\circ g: C \to A$ is constant so the image is $p \in A$. Then, the image of $C$ under $g$ is the fiber $\bb P^1$ of $\pi$ at $p$. Thus, since $h^*\ecal L = \ecal O_C$, we have $\deg (g^* \omega_{X}) = \deg g^*\ecal O_{\bb P^1}(-2) = -2 \deg g \neq 0$, which is a contradiction. Hence, $h: C\to A$ needs to be nonconstant. 
    
    Now suppose that $\ecal L \in \pic^0(A)$. Then, writing $\ecal Q = g^* \ecal O_{X}(1)$, we have $$\deg \ecal Q = - \frac{1}{2}\deg (g^*\omega_{X}\otimes h^*\ecal L\inv) = 0.$$ Also note that the tautological quotient $\pi^*\ecal E \surj \ecal O_{X}(1)$ restricts to a surjection $h^*\ecal E \surj \ecal Q$. In particular, we have either $\ecal Q \cong \ecal O_C$ or $\ecal Q \cong h^*\ecal L$ since $\deg \ecal Q = \deg h^*\ecal L = 0$ and any non-trivial line bundle of degree $0$ on a smooth projective curve has no global section. Hence, we have either $g^*\omega_{X} \iso h^*\ecal L$ or $g^*\omega_{X} \iso h^*\ecal L\inv$, respectively. Therefore, in either case, $h^*\ecal L$ is torsion. We are done by the following claim. \\
\underline{Claim.} For a simple abelian variety $A$ (over an algebraically closed field) 
and a nonconstant morphism $h: C \to A$ from a smooth projective curve $C$, $\ecal L \in \pic^0(A)$ is torsion if (and only if) $h^*\ecal L$ is torsion. \\
\underline{Proof.} It suffices to show the pullback homomorphism
\[
h^*: \pic^0_A \to \pic^0_C =:J_C
\]
has finite kernel, and since the dual abelian variety $\pic^0_A$ is also simple, it suffices to show $h^*$ is non-zero. But $h^*$ is dual to the unique morphism $\operatorname{Alb}_C = J_C^t \to A$ such that the composition $C \to \operatorname{Alb}_C \to A$ is $h$, so it must be non-zero as $h$ is non-constant.


Next, suppose the Picard number of $A$ is $1$. If $\ecal L \in \pic^0(A)$, then the claim follows from the arguments above. If $\ecal L \not \in \pic^0(A)$, then $\ecal L$ is either ample or anti-ample, i.e, $\otimes$-generating, which contradicts (i).
\end{proof}

We fix the following notations for projective bundles on an elliptic curve. 
\begin{notation}
    Let $E$ be an elliptic curve over an algebraically closed field of characteristic zero and $\ecal E$ a vector bundle of rank $n$ on $E$. Let $\pi:\bb P_E(\ecal E) \to E$ denote the corresponding projective bundle. In this case, we have $\omega_{\bb P_E(\ecal E)} \cong \pi^* \det \ecal E \tens \ecal O_{\bb P_E(\ecal E)}(-n)$. For a morphism $f: C \to \bb P_E(\ecal E)$ from a smooth projective curve $C$, which corresponds to a quotient $(\pi \circ f)^*\ecal E \to \ecal Q:= f^*\ecal O_{\bb P_E(\ecal E)}(1)$, we have
\[
f^* \omega_{\bb P_E(\ecal E)} \cong  (\pi\circ f)^*\det \ecal E \otimes \ecal Q^{\otimes (-n)}. 
\]
Let $\ecal F_r$ denotes the unique indecomposable vector bundle of rank $r$ and degree $0$ with $H^0(E,\ecal F_r) \neq 0$ (cf. \cite{atiyah1957vector}*{\href{https://math.berkeley.edu/~nadler/atiyah.classification.pdf}{Theorem 5}}). 
\end{notation}
Recall the following from \cite{itoolander2025derived}*{\href{https://arxiv.org/pdf/2507.17681v2}{Example 5.15}}.
\begin{lemma}
    Suppose that $\rank \ecal E = 2$. Then $\omega_{\bb P_E(\ecal E)}$ is $\otimes$-generating if and only if $\ecal E$ is unstable, i.e., $\ecal E \cong \ecal M \oplus \ecal N$ with $\deg \ecal M \inv \otimes \ecal N \neq 0$.
\end{lemma}
The following shows that as long as $\ecal M \inv \otimes \ecal N$ is not torsion, $\omega_{\bb P_E(\ecal E)}$ is nowhere torsion, providing examples where Theorem \ref{theorem-reconstruction-for-schemes} applies, while none of the reconstructions in \cites{ballard2011derived, FAVERO20121955, ito2025polarizations} works.
\begin{prop}
    Suppose that $\rank \ecal E = 2$. Then $\omega_{\bb P_E(\ecal E)}$ is nowhere torsion if and only if 
    \[
    \ecal E \cong \ecal M \oplus \ecal N
    \]
    for some line bundles $\ecal M$ and $\ecal N$ on $E$ with $\ecal M \otimes \ecal N\inv$ being not torsion. In particular, if $\bb P_E(\ecal O_E \oplus \ecal L)$ has a non-isomorphic Fourier--Mukai partner (or more generally a non-standard autoequivalence), then $\ecal L$ is a torsion line bundle.  
\end{prop}
\begin{remark}
    This gives another partial proof of \cite{uehara2017fourier}*{Theorem 1.1}. 
\end{remark}
\begin{proof}

First, by Proposition \ref{prop-main-projective-bundle-over-abelian} (iii), if $\ecal E$ is decomposable, then the claim holds. 
    Hence, it remains to show that if $\ecal E$ is indecomposable, then $\omega_{\bb P_E(\ecal E)}$ is somewhere torsion.  First, suppose $\deg \ecal E$ is even. Then, we may assume $\deg \ecal E = 0$ by twisting by a line bundle. By Atiyah's classification \cite{atiyah1957vector}*{\href{https://math.berkeley.edu/~nadler/atiyah.classification.pdf}{Theorem 5}}, there is a line bundle $\ecal L$ of degree $0$ such that $\ecal E \cong \ecal F_2 \otimes \ecal L$ which fits into a short exact sequence
    \[
    0 \to \ecal L \to \ecal E \to \ecal L \to 0.
    \]
    Thus, $\det \ecal E \cong \ecal L^{\otimes 2}$ and for a section $f: E \to \bb P_E(\ecal E)$ corresponding to the quotient $\ecal E \surj \ecal L$, we have
    \[
    f^*\omega_{\bb P_E(\ecal E)} \cong \ecal L^{\otimes 2} \otimes \ecal L^{\otimes (-2)} \cong \ecal O_E
    \]
    and therefore $\omega_{\bb P_E(\ecal E)}$ is somewhere torsion. 

    Next, suppose $\deg \ecal E$ is odd. Then, we may assume $\deg \ecal E = 1$ by twisting by a line bundle. In particular, $\ecal E$ is stable and hence simple. By \cite{atiyah1957vector}*{\href{https://math.berkeley.edu/~nadler/atiyah.classification.pdf}{Corollary to Theorem 7}}, $\ecal E$ is also semihomogeneous. Thus, by \cite{mukai1978semi}*{\href{https://projecteuclid.org/journals/kyoto-journal-of-mathematics/volume-18/issue-2/Semi-homogeneous-vector-bundles-on-an-abelian-variety/10.1215/kjm/1250522574.full}{Proposition 7.3}}, there is an isogeny $r: E' \to E$ such that $r^* \ecal E \cong \ecal L \oplus \ecal L$ for a line bundle $\ecal L$ on an elliptic curve $E'$. Let $f: E' \to \bb P_E(\ecal E)$ be the morphism corresponding to a quotient $r^*\ecal E \surj \ecal L$. Then,
    \[
    f^*\omega_{\bb P_E(\ecal E)} \cong \ecal L^{\otimes 2} \otimes \ecal L^{\otimes (-2)} \cong \ecal O_{E'}
    \]
    and therefore $\omega_{\bb P_E(\ecal E)}$ is somewhere torsion.

    The claim after ``In particular'' follows by Theorem \ref{theorem-reconstruction-for-schemes}.
\end{proof}

\subsubsection{Varieties with nowhere torsion and nef, but not ample canonical bundle}
\label{subsubsection-abundance}
The abundance conjecture predicts that for a normal variety $X$ with Gorenstein klt singularities, a nef canonical bundle is semi-ample. Since a nowhere torsion and semi-ample line bundle on a proper variety is ample, proper Gorenstein varieties with nowhere torsion and nef but not ample canonical bundle are expected to have singularities worse than klt. Indeed, ill behavior of the canonical bundle of varieties with bad singularities is well-known to experts (see e.g. \cite{lazic2013around}*{Example 5.1} due to Koll\'ar). We use a similar idea to construct a proper Gorenstein variety with nowhere torsion and nef but not ample canonical bundle.

\begin{construction}\label{conctruction-controlledcanonical}
    Let $B$ be a smooth projective variety. Let $\ecal L$ be a line bundle on $B$. Set 
    \[
    \pi: Y := \bb P_B(\ecal O_B^{\oplus 3} \oplus \ecal L) \to B.
    \]
    Let $S_0 \subset Y$ be the section corresponding to the canonical quotient $\ecal O_B^{\oplus 3} \oplus \ecal L \to \ecal L$. Take a line bundle $\ecal M$ on $B$ with a non-zero section
    \[
    s \in H^0(B \times \bb P^2, \ecal M \boxtimes \ecal O_{\bb P^2}(5)) \cong H^0 (B, \ecal M \otimes \sym^5(\ecal O_B^{\oplus 3})).
    \]
    Set 
    \[
    X = X(\ecal M, s) = V(\tilde s) \subset Y
    \]
    where $\tilde s$ is the image of $s$ under the inclusion map
    \[
    H^0(B \times \bb P^2, \ecal M\boxtimes \ecal O_{\bb P^2}(5)) = H^0 (B, \ecal M \otimes \ecal \sym ^5 (\ecal O_{B}^{\oplus 3} )) \hookrightarrow H^0(B, \ecal M \otimes \sym^5(\ecal O^{\oplus 3}_B \oplus \ecal L)) =  H^0(Y, \pi^*\ecal M \otimes \ecal O_{Y}(5)). 
    \]
    In particular, since $X$ is an effective Cartier divisor on a smooth projective variety, it is a projective Gorenstein scheme. By the adjunction formula, we have 
    \begin{align*}
        \omega_X & \cong( \omega_Y \otimes \ecal O_Y(X))|_X \\
        &\cong (\ecal O_Y(-4) \otimes \pi^*(\omega_B \otimes \ecal L) \otimes \pi^*\ecal M \otimes \ecal O_Y(5))|_X \\
        &\cong (\ecal O_Y(1) \otimes \pi^*(\omega_B \otimes \ecal L \otimes \ecal M))|_X. 
    \end{align*}
    For the section $B \cong S_0 \subset Y$ corresponding to the canonical quotient $\ecal O_B^{\oplus 3} \oplus \ecal L \surj \ecal L$, we have $S_0 \subset X$ and
    \[
    \omega_X|_{S_0} \cong \ecal L \otimes \omega_B \otimes \ecal L \otimes \ecal M = \omega_B \otimes \ecal L^{\otimes 2} \otimes \ecal M.
    \]
    In particular, for a property $\sf{P}$ of a line bundle that is preserved under the restriction to a closed subscheme (e.g., (anti-)ample, (anti-)nef, $\otimes$-generating, nowhere torsion, etc.), if $\omega_B \otimes \ecal L^{\otimes 2} \otimes \ecal M$ does not satisfy $\sf{P}$, neither does $\omega_X$. 
\end{construction}
{\begin{remark}\label{remark-choiceof5}
    Let $X_d = X_d(\ecal{M}, s)$ denote a variety constructed as in Construction \ref{conctruction-controlledcanonical}, but with $\ecal O_Y(5)$ replaced by $\ecal O_Y(d)$. Then, by the same adjunction computation, we have $\omega_{X_d} \cong (\ecal O_Y(d-4) \otimes \pi^*(\omega_B \otimes \ecal L \otimes \ecal M))|_{X_d}$. For $b \in B$ with $s_b \neq 0$, the fiber $(X_d)_b \subset Y_b \cong \bb P^3$ is the projective cone over a curve $V(s_b) \subset \bb P^2$ of degree $d$ with vertex at $(S_0)_b$. Take a ruling line $l \cong \bb P^1$ of the projective cone $(X_d)_b$. Then, since $\ecal O_Y(1)|_l = \ecal O_l(1)$ and $\pi^*(\omega_B \otimes \ecal L \otimes \ecal M)|_l = \ecal O_l$, we have $\omega_{X_d}|_l = \ecal O_l(d-4)$. Hence, if $\omega_{X_d}$ is nef, then $d \geq 4$ and if $\omega_{X_d}$ is nef and nowhere torsion, then $d > 4$. 
\end{remark}
\begin{lemma}
    If $V(s) \subset B \times \bb P^2$ is smooth (resp. irreducible), then $X$ is normal (resp. irreducible).
\end{lemma}
\begin{proof}
    First, note that the natural projection 
    \[
    p: Y \setminus S_0 \to B \times \bb P^2
    \]
    is an $\bb A^1$-bundle and hence so is $X \setminus S_0 = p\inv(V(s)) \to V(s)$. Now, suppose $V(s)$ and hence $X \setminus S_0$ is smooth. Then, since $\codim _X S_0 = 2$, $X$ is regular in codimension 1. Since $X$ is moreover Gorenstein and hence $S_2$, $X$ is normal as desired. Finally, suppose $V(s)$ and hence $X \setminus S_0$ is irreducible. Then $X\setminus S_0$ is dense in $X$ since $X$ has pure codimension $1$ in $Y$, hence $X$ is irreducible.
\end{proof}

\begin{prop} \ 
    \begin{enumerate}
        \item If both $\omega_B\otimes \ecal L \otimes \ecal M$ and $\omega_B\otimes \ecal L^{\otimes 2} \otimes \ecal M$ are nef (resp. strictly nef), then so is $\omega_X$.
        \item If $\omega_B\otimes \ecal L \otimes \ecal M$ is nowhere torsion and nef and $\omega_B\otimes \ecal L^{\otimes 2} \otimes \ecal M$ is nowhere torsion, then $\omega_X$ is nowhere torsion. \qedhere 
    \end{enumerate}
\end{prop}
\begin{proof}
    Let $g : C \to X$ be a finite morphism with $C$ a normal projective curve. Write $f$ for the composition $C \to X \to B$. The composition $C \to X \hookrightarrow Y = \mathbb{P}_B(\ecal{O}_B^{\oplus 3} \oplus \ecal{L})$ then corresponds to a quotient
    \begin{equation}
    \label{equn-thequotient}
    \ecal O_{C}^{\oplus 3} \oplus f^* \ecal L \surj g^*\ecal O_Y(1) =:\ecal Q.
    \end{equation}
    Also, recall $\omega_X \cong (\ecal O_Y(1) \otimes \pi^*(\omega_B \otimes \ecal L \otimes \ecal M))|_X,$ so 
    \[
    g^*(\omega_X) = \ecal Q \otimes f^*(\omega_B \otimes \ecal L \otimes \ecal M). 
    \]
    If $f$ is constant, then $g$ factors through a fiber $F \cong \bb P^3$ of $\pi$ so 
    \[
    \deg g^*(\omega_X) = \deg \ecal Q =  \deg (C \to F)^*\ecal O_{F}(1)> 0.
    \]
    Now, suppose $f$ is non-constant and hence finite. 
    First, if there is a non-trivial component $\ecal O_{C} \to \ecal Q$ in (\ref{equn-thequotient}), then $\deg \ecal Q \geq 0$ and hence
    \[
    \deg g^*(\omega_X) \geq \deg f^*(\omega_B \otimes \ecal L \otimes \ecal M).
    \]
    Similarly, if the component $f^*\ecal L \to \ecal Q$ is non-trivial, then $\deg \ecal{Q} \otimes f^*\ecal L\inv \geq 0$ and hence
    \[
    \deg g^*(\omega_X) \geq \deg f^*(\omega_B \otimes \ecal L^{\otimes 2} \otimes \ecal M).
    \]
    Therefore, (i) follows. For (ii), if $\deg \ecal Q>0$, then $\deg g^*(\omega_X) \geq \deg \ecal Q >0$ as $\omega_B \otimes \ecal L \otimes \ecal M$ is nef, so $g^*(\omega_X)$ is not torsion. If $\deg \ecal Q = 0$, then since there is a surjective map (\ref{equn-thequotient}), we have either $\ecal Q \cong \ecal O_{C}$ or $\ecal Q \cong f^*\ecal L$. In either case, $g^*(\omega_X) = \ecal Q \otimes f^*(\omega_B \otimes \ecal L \otimes \ecal M)$ is not torsion by supposition. If $\deg \ecal Q<0$, then $\ecal O_{ C} \to \ecal Q$ is trivial, so $f^*\ecal L \to \ecal Q$ is surjective and hence an isomorphism. Thus,
    $
    g^*(\omega_X) = f^*(\omega_B \otimes \ecal L^{\otimes 2} \otimes \ecal M)
    $
    is not torsion by supposition. 
\end{proof}
Now, we can produce plenty of examples with desired properties of canonical bundles.
\begin{example}\label{example-controlled canonical} \ 
    \begin{enumerate}
        \item We first consider an example of a normal projective Gorenstein variety whose canonical bundle is nowhere torsion and nef but not strictly nef (and hence not ample and thus not $\otimes$-generating). Let $B$ be an elliptic curve, $\ecal L \in \pic^0(B)$ a non-torsion line bundle, and $\ecal M = \ecal O_B$. 
        Take $s \in H^0(B \times \bb P^2, \ecal O_B \boxtimes \ecal O_{\bb P^2}(5)) \cong H^0(\bb P^2, \ecal O_{\bb P^2}(5))$ to be a smooth quintic in $\bb P^2$ so that $X$ is a normal projective Gorenstein variety. Then, $\omega_B \otimes \ecal L \otimes \ecal M \cong \ecal L$ is nef and nowhere torsion and $\omega_B \otimes \ecal L^{\otimes 2} \otimes \ecal M \cong \ecal L^{\otimes 2}$ is nef nowhere torsion but not strictly nef, so $\omega_X$ is nef and nowhere torsion, but not strictly nef. This works more generally with $B$ any simple abelian variety and $\ecal{L}$ any non-torsion element of $\operatorname{Pic}^0(B)$ since such an $\ecal{L}$ is nef and nowhere torsion but not strictly nef, see the proof of Proposition \ref{prop-main-projective-bundle-over-abelian}. 
        \item We can also construct an example of a projective Gorenstein variety whose canonical bundle is strictly nef (and hence nowhere torsion) but not ample (and hence not $\otimes$-generating). Let $B$ be any smooth projective variety with a strictly nef line bundle $\ecal L_0$ that is not ample and take an ample line bundle $\ecal A$ on $B$ such that $\ecal M:= \ecal L_0 \otimes \omega_B\inv \otimes \ecal A^{\otimes 2}$ is very ample. Set $\ecal L = \ecal A \inv$. Take $s \in H^0(B \times \bb P^2, \ecal M \boxtimes \ecal O_{\bb P^2}(5))$ with $V(s)$ smooth irreducible, which exists by Bertini's theorem, noting that $\ecal M \boxtimes \ecal O_{\bb P^2}(5)$ is very ample. Now, since 
        \[
        \omega_B \otimes \ecal L \otimes \ecal M \cong \ecal L_0 \otimes \ecal A
        \]
        is strictly nef and 
        \[
        \omega_B \otimes \ecal L^{\otimes 2} \otimes \ecal M \cong \ecal L_0
        \]
        is strictly nef but not ample, $\omega_X$ is strictly nef but not ample.  \qedhere
    \end{enumerate}
\end{example}

\begin{remark}
    These examples do not contradict the abundance conjecture. For simplicity, let us show the variety constructed in Example \ref{example-controlled canonical} (i) is not log canonical and in particular not klt. Let $C \subset \bb P^2$ denote the fixed smooth quintic. First, take a nonempty open subset $U \subset B$ where $\ecal L$ trivializes. Then, over $U$, we have
    \[
    Y_U \cong U \times \bb P^3, \quad X_U \cong U \times \bar C, \quad ({S_0})_U \cong U \times \{v\}
    \]
    where $\bar C \subset \bb P^3$ is the projective cone over $C$ by definition and $v$ is the vertex of the cone. 
    Now, let us consider the blow-up $b: \tilde Y =  \operatorname{Bl}_{S_0} Y \to Y$ with exceptional divisor $E$ and let $\rho: \tilde X = \operatorname{Bl}_{S_0} X \to X$ be the strict transform of $X$ with exceptional divisor $E_X = E \cap \tilde X$ (cf. \cite[\href{https://stacks.math.columbia.edu/tag/080E}{Tag 080E}]{stacks-project}). Then, $\tilde X$ is smooth since it can be checked locally on $B$ and we have $\tilde X|_U \cong U \times \operatorname{Bl}_v\bar C$, which is indeed smooth as $\operatorname{Bl}_v\bar C$ is a $\bb P^1$-bundle over $C$. Now, since $\operatorname{codim}_Y{S_0}= 3$, we have $b^* \omega_Y = \omega_{\tilde Y} \otimes \ecal O_{\tilde Y}(-2E).$ On the other hand, since $X$ has multiplicity $5$ along $S_0$, we have $ b^* \ecal O_Y(X) \cong \ecal O_{\tilde Y}(\tilde X)  \otimes \ecal O_{\tilde Y}(5E).$ Therefore, by the adjunction formula,
     \[   \omega_{\tilde X}  \cong \rho^* \omega_X \otimes \ecal O_{\tilde X}(-3E_X).\]
        
    Therefore, $E$ over $X$ has discrepancy $- 3 < -1$ and hence $X$ is not log canonical. Note that this discrepancy computation is the same as the one for the exceptional divisor of the resolution $\operatorname{Bl}_v \bar C \to \bar C$. In particular, in order to have log canonical (resp. klt) singularity, we need to consider a smooth cubic (resp. conic) instead of a quintic. However, in these cases, we do not get a nef canonical bundle by Remark \ref{remark-choiceof5}, so this construction is also compatible with the abundance conjecture. 
\end{remark}
\subsubsection{Varieties with no nowhere torsion line bundle}
\label{subsubsection-yesbigbutnonowheretorsion}

It is not difficult to give an example of a proper variety with no nowhere torsion line bundle. Namely, let $X$ be obtained by gluing together a non-intersecting pair of a line and a smooth rational curve  of degree $2$ in $\mathbb{P}^3_k$ along an isomorphism, see \cite[Section 6.2]{Ferrand-Conducteur}. Let $\mathbb{P}^1_k \cong C \subset X$ be the image of the two curves. Then every line bundle $\ecal{L}$ on $X$ has degree zero on $C$ and thus $\ecal{L}|_C \cong \ecal{O}_C$ is torsion. It is not difficult to see however that this proper variety $X$ has no big line bundles, as there is a finite birational morphism $f: \mathbb{P}^3_k \to X$ and the argument shows that for every line bundle $\ecal{L}$ on $X$, we have $f^*\ecal{L}\cong \ecal{O}_{\mathbb{P}^3_k}$ which is not big, hence $\ecal{L}$ is not big. It is more interesting to ask whether there exists a proper variety $X$ with a big line bundle but no nowhere torsion line bundle. We construct such a variety now.

\begin{example} 
Fix $m \geq 1$. Consider the Hirzebruch surface $\pi : \bb F_m = \mathbb{P}(\ecal{O}_{\mathbb{P}^1} \oplus \ecal{O}_{\mathbb{P}^1}(-n)) \to \mathbb{P}^1$. Let $F \subset \bb F_m$ be a fiber of $\pi$ and $S_+, S_- \subset \bb F_m$ be the sections of $\pi$ corresponding to the two direct summands. We have $S_+^2 = n, S_-^2 = -n$, and $S_+ \cap S_- = \emptyset.$ Let $X = \bb P^1 \times \bb F_m$ and consider pairwise disjoint smooth rational curves $L_1 = \{p\} \times F, L_2 = \{p'\}\times S_+, L_3 = \bb P^1 \times \{q\}$ in $X$. Consider finite morphisms $f_i: L_i \to \bb P^1$ with $\deg f_1 = \deg f_3 = 1$ and $\deg f_2 = m$, which gives a finite morphism $f: \sqcup L_i \to \bb P^1$. By \cite[Theorem 5.4]{Ferrand-Conducteur} or \cite[\href{https://stacks.math.columbia.edu/tag/0E25}{Tag 0E25}]{stacks-project}, there exists a pushout
$$
\begin{tikzcd}
\sqcup L_i \ar[r, hook, "i"] \ar[d, "f"] & X \ar[d, "\nu"] \\
\bb P^1 \ar[r, hook, "j"] & Y
\end{tikzcd}
$$
in the category of schemes. 
Since  $Y$ is separated (\cite[\href{https://stacks.math.columbia.edu/tag/0E26}{Tag 0E26}]{stacks-project}) and locally of finite type (\cite[\href{https://stacks.math.columbia.edu/tag/0E27}{Tag 0E27}]{stacks-project}) and $\nu: X \to Y$ is (finite) surjective, $Y$ is proper. 

Write $c = j \circ f = \nu \circ i$. By construction of the pushout, we have $\ecal{O}_Y = j_*\ecal{O}_{\bb P ^1} \times_{c_*\ecal{O}_{\sqcup L_i}} \nu_*\ecal{O}_X.$ Taking units gives a short exact sequence
$$
0 \to \ecal{O}_Y^\times \to j_*(\ecal{O}_{\bb P^1}^\times) \times \nu_*(\ecal{O}_X^\times) \to c_*(\ecal{O}_{\sqcup L_i}^\times) \to 0
$$
of abelian sheaves on $Y$. Taking cohomology (and noting that for a finite morphism of schemes $g: T \to S$, we have $R^1g_*(\ecal{O}_T^\times) = 0$ hence $H^1(S, g_*(\ecal{O}_T^\times )) = \operatorname{Pic}(T)$), we obtain an exact sequence
\[\begin{tikzcd}[row sep = 1em]
	0 & {H^0(Y, \ecal{O}_Y)^\times} & {H^0(\mathbb P^1, \ecal{O}_{\mathbb P^1})^\times \times H^0(X, \ecal{O}_X)^\times} & {\prod _i H^0(L_i, \ecal{O}_{L_i})^\times} \\
	& {\operatorname{Pic}(Y)} & {\operatorname{Pic}(\mathbb{P}^1) \times \operatorname{Pic}(X)} & {\prod _i \operatorname{Pic}(L_i).}
	\arrow[from=1-1, to=1-2]
	\arrow[from=1-2, to=1-3]
	\arrow[from=1-3, to=1-4]
	\arrow[from=1-4, to=2-2, out=0, in=180, looseness=2.5, overlay]
	\arrow[from=2-2, to=2-3]
	\arrow[from=2-3, to=2-4]
\end{tikzcd}\hspace*{1.5cm}\]

That is, a line bundle $\ecal L'$ on $Y$ is equivalent to a pair $(\ecal O_{\bb P^1}(n), \ecal L)$ of line bundles on $\bb P^1$ and $X$ together with isomorphisms $\phi_i: \ecal L|_{L_i} \iso f_i^*\ecal O_{\bb P^1}(n)$. 
We have $\pic(X) = \pic(\bb P^1) \times \pic(\bb F_m)$, so write $\ecal L = \ecal O_{\bb P^1}(k) \boxtimes \ecal O_{\bb F_m}(aS_+ + bF)$. Then, $\ecal L|_{L_1} = \ecal O_{L_1}(a)$, $\ecal L|_{L_2} = \ecal O_{L_2}(am+b)$, and $\ecal L|_{L_3} = \ecal O_{L_3}(k)$. On the other hand, we have $f_i^*\ecal{O}_{\bb P^1}(n) = \ecal{O}_{\bb P^1}(n)$ if $i = 1,3$ and $=\ecal{O}_{\bb P^1}(m \cdot n)$ if $i = 2$. Hence for $(\ecal O_{\bb P^1}(n), \ecal L)$ to come from a line bundle on $Y$, we must have $k = a = n$ and $b = 0$. Therefore, letting $\ecal M$ denote a line bundle on $Y$ corresponding to the data $(\ecal O_{\bb P^1}(1), \ecal O_{\bb P^1}(1) \boxtimes \ecal O_{\bb F_m}(S_+))$ with a fixed isomorphism over $\sqcup \ecal L_i$, the long exact sequence above simplifies to the (split) exact sequence
\[
0 \to (k^\times)^2 \to \pic(Y) \to \bb Z \cdot [\ecal M] \to 0. 
\]
Now, we claim that the line bundle $\ecal M$ is big. 
Note that $\nu: X \to Y$ is finite and birational, so $\ecal M$ is big 
if and only if so is $\nu^*\ecal M = \ecal O_{\bb P^1}(1) \boxtimes \ecal O_{\bb F_m}(S_+)$. This is true because the exterior product of big line bundles is big.

On the other hand, for any line bundle $\ecal{M}'$ on $Y$, we have seen that $\nu^*\ecal{M}' \cong \ecal O_{\bb P^1}(k) \boxtimes \ecal O_{\bb F_m}(aS)$ for some integers $a,k$. But then $(\nu^*\ecal{M}')|_{\{p\} \times S_-} \cong \ecal{O}_{\{p\} \times S_-}$, so $\nu^*\ecal{M}'$ is somewhere torsion. This implies $\ecal{M}'$ is somewhere torsion by Lemma \ref{lemma-nowhere torsionfinitesurj}.
\end{example} 

If we allow algebraic spaces, we have the following example. See the next section for more discussion of algebraic spaces and stacks.

\begin{example}
    Hironaka constructed a smooth proper algebraic space $X$ over any field $k$ with a curve $\mathbb{P}^1_k \cong C \subset X$ which is algebraically equivalent to zero. See \cite[Example B.3.4.2]{Har77}. Then any line bundle $\ecal{L}$ on $X$ satisfies $\ecal{L}|_{C} \cong \ecal{O}_C$, hence $\ecal{L}$ is somewhere torsion (the definition of nowhere torsion extends in the obvious way to algebraic spaces). Such an $X$ automatically has a big line bundle: Take any affine open subscheme $U \subset X$, which exists since $X$ has a dense open subscheme, and let $D \subset X$ be the complement, as a reduced closed subscheme. Since $X$ is smooth, $D \subset X$ is an effective Cartier divisor, and $\ecal{O}_X(D)$ is big, see \cite[Section 2.2]{itoolander2025derived}. 
\end{example}

\section{Nowhere torsion line bundles on stacks}

In this section, we define nowhere torsion line bundles on tame algebraic stacks with the ultimate goal of proving a reconstruction theorem for a smooth proper tame algebraic stack $\mathcal{X}$ over a field, with nowhere torsion canonical bundle and which are generically representable by schemes. We achieve this goal later, in Chapter \ref{section-BOforstacks}. The same definition of nowhere torsion line bundle as in the case of schemes works well here, see Section \ref{subsection-nowheretorsiononstack}, but the rest of the story requires new ideas. 

Again, we would like to characterize ``skyscraper sheaves of points'' on $\mathcal{X}$ purely in terms of its derived category. Here a skyscraper sheaf at a point should be interpreted as the pushforward of a simple object from the residual gerbe at a closed point. These objects are studied in \cite{hall2024generalizedbondalorlovfaithfulnesscriterion} and we recall the basic results in Section \ref{subsection-generalizedpoints}. We give our definition of point-like object in Section \ref{subsection-pointlike}. As in the case of schemes, this definition is our attempt to characterize skyscraper sheaves at closed points of $\mathcal{X}$ purely in terms of its derived category, but here there can be more objects which look like skyscraper sheaves, so we are forced to add more conditions in our definition, see Remark \ref{remark-point-like}. 

Even with these new conditions, unlike the case of schemes, it is not necessarily the case that point-like objects with zero-dimensional support are skyscraper sheaves. This cannot be fixed by strengthening the definition of point-like object. Rather, it is an unavoidable consequence of a phenomenon which occurs for stacks but not for schemes: There can be automorphisms of $D^b_{\operatorname{Coh}}(\mathcal{X})$ which preserve all the subcategories $D^b_{\operatorname{Coh}, x}(\mathcal{X})$ where $x$ is a closed point, and which act non-trivially on some of these subcategories. A guiding example is given in Example \ref{example-spherical-twist-on-p112}. These so-called \emph{local} automorphisms of $D^b_{\operatorname{Coh}}(\mathcal{X})$ play a fundamental role in our main reconstruction theorem in the next section.


\subsection{Conventions on algebraic stacks}

We follow the conventions on algebraic stacks used in the Stacks Project. We shall mostly deal with algebraic stacks with finite inertia, usually even finite diagonal. An algebraic stack $\mathcal{X}$ with finite inertia has a coarse moduli space $\pi : \mathcal{X} \to X$ by the main result of \cite{rydh2013existence} which generalizes the main result of \cite{K-M}. If $\mathcal{X}$ is also locally of finite type over a locally Noetherian scheme, then $X$ is locally Noetherian and $\pi$ is proper. An algebraic stack with finite inertia and coarse moduli space $\pi : \mathcal{X} \to X$ is \emph{tame} if the functor $\pi _* : \operatorname{QCoh}(\mathcal{X}) \to \operatorname{QCoh}(X)$ is exact. If $G/S$ is a finite locally free group scheme over a scheme, we shall call $G$ \emph{tame} if $BG$ is a tame algebraic stack, or equivalently, if $G$ is linearly reductive: The $G$-invariants functor from the category of $G$-equivariant quasi-coherent sheaves on $S$ to the category of quasi-coherent sheaves on $S$ is exact. 

For derived categories of algebraic stacks, we follow the conventions of \cite{hall2017perfect}. Thus if $\mathcal{X}$ is an algebraic stack, then $D_{qc}(\mathcal{X})$ is the full subcategory of the derived category of $\ecal{O}_{\mathcal{X}}$-modules on the lisse-\'etale site of $\mathcal{X}$ whose objects have quasi-coherent cohomology sheaves. For a morphism of algebraic stacks $f : \mathcal{X} \to \mathcal{Y}$, there is an adjunction $(Lf^*, Rf_*) : D_{qc}(\mathcal{X}) \leftrightarrow D_{qc}(\mathcal{Y})$, but the details are subtle due to the lack of functoriality of the lisse-\'etale site. In particular, $Rf_*$ is simply defined as the right adjoint of $Lf^*$ (here we differ from the conventions of \cite{hall2017perfect} in that we write $Rf_*$ instead of $Rf_{qc, *}$). This functor is particularly well-behaved when $f$ is a \emph{concentrated morphism} of algebraic stacks. Examples of such morphisms are quasi-compact and quasi-separated morphisms which are representable by algebraic spaces, and the projection $\mathcal{X} \times _k \mathcal{Y} \to \mathcal{Y}$ where $\mathcal{X}$ is a quasi-compact and quasi-separated tame algebraic stack over a field $k$ and $\mathcal{Y}$ is an arbitrary algebraic stack.

We note that our main results are all about Noetherian algebraic stacks with finite diagonal, and for such stacks we have $D_{qc}(\mathcal{X}) \cong D(\operatorname{QCoh}\mathcal{X})$ by \cite[Theorem A]{hall2017perfect} and \cite[Theorem 1.2]{HallNeemanRydh2019}. We encourage the reader to work exclusively with such stacks in this paper. 

We often denote $i_x : \ecal{G}_x \to \mathcal{X}$ the residual gerbe to an algebraic stack $\mathcal{X}$ at a point $x \in |\mathcal{X}|$. If $x : \operatorname{Spec} k \to \mathcal{X}$ is a morphism from the spectrum of a field, then $G_x$ denotes the $k$-group scheme $\underline{\operatorname{Aut}}_k(x)$. If $G$ is a group scheme over a field $k$, then we sometimes write $\bb X(G)$ for the group $\operatorname{Hom}(G, \mathbb{G}_{m, k})$ of characters of $G.$

\subsection{Nowhere torsion line bundles on stacks}
\label{subsection-nowheretorsiononstack}

Let $k$ be a field. Let $\mathcal{X}$ be a finite type algebraic stack over $k$ with finite diagonal. In particular, $\mathcal{X}$ is separated over $k$, and $\mathcal{X}$ has finite inertia so admits a coarse space morphism $\mathcal{X} \to X$. Also, there is a finite, surjective, generically flat morphism $Y \to \mathcal{X}$ where $Y$ is a scheme by \cite[Theorem B]{rydh_approx} (see also \cite[Theorem 2.7]{ehkv}).


Recall that the \textit{support} of a coherent sheaf $\cal F$ on $\cal X$ is the closed subset 
$\supp \ecal F \subset |\cal X|$ defined by the property that 
$x: \operatorname{Spec}K \to \mathcal{X}$ is in $\supp \cal F$ if and only if $x^* \ecal F \neq 0$ (and this does not depend on the choice of a representative $x$). We will say a closed subset $T \subset |\mathcal{X}|$ is \textit{proper} if the unique reduced closed substack $\cal Z \subset \cal X$ with $|\mathcal{Z}| = T$ is proper ($\iff$ any closed substack $\cal Z \subset \cal X$ with $|\mathcal{Z}| = T$ is proper).

\begin{definition}
    Let $\ecal L$ be a line bundle on $\cal X$. We say $\ecal L$ is \textit{nowhere torsion} if for any coherent sheaf $\ecal F$ on $\cal X$ whose support is proper over $k$, $\ecal F \cong \ecal F \otimes \ecal L$ implies $\dim \supp \ecal F = 0$.  
\end{definition}

Again, we will call a line bundle which is not nowhere torsion \emph{somewhere torsion}.

\begin{construction}
   Suppose we have a proper irreducible closed subset $T \subset |\cal X|$ of dimension $1$. Let $\cal C \subset \cal X$ be the reduced closed substack corresponding to $T$. Then, $\cal C$ 
   satisfies the same assumptions as $\mathcal{X}$, so 
   there exists a finite surjective (generically flat) morphism $C \to \cal C$ from a scheme $C$. Note that $C$ is a proper scheme over $k$ of dimension $1$ so is projective over $k$, and that we may assume $C$ is integral by replacing $C$ by a reduced irreducible component which dominates $\mathcal{C}$. Composing with the normalization $C^\nu \to C$, we obtain a finite morphism $f: C^\nu \to \cal X$ from a normal projective curve $C^\nu$ over $k$ with image equal to $T$.
\end{construction}

\begin{prop}
   Let $\ecal{L}$ be a line bundle on $\mathcal{X}$. The following are equivalent:
    \begin{enumerate}
    \item $\ecal{L}$ is somewhere torsion.
    \item There is a finite morphism $f : C \to \mathcal{X}$ where $C$ is a normal projective curve over $k$ and such that $f^*\ecal{L}$ is a torsion line bundle on $C$.
    \item There is a finite morphism $f : C \to \mathcal{X}$ where $C$ is a projective scheme of dimension $1$ over $k$ and such that $f^*\ecal{L}\cong \ecal{O}_C$.
    \item There is a finite morphism $f : C \to \mathcal{X}$ where $C$ is a normal projective curve over $k$ and such that $f^*\ecal{L}\cong \ecal{O}_C$. \qedhere
    \end{enumerate} 
\end{prop}
\begin{proof}
    (i) $\implies$ (ii): If $\ecal{L}$ is somewhere torsion, then there exists a coherent sheaf $\ecal{F}$ on $\cal X$ with proper support of positive dimension such that $\ecal{F} \otimes \ecal{L} \cong \ecal{F}$. Choose a proper irreducible closed subset
     $T \subset \operatorname{Supp} \ecal{F}$ of dimension $1$ 
     and let $f : C^\nu \to \cal X$ be a finite morphism as in the construction above. Then $f^*\ecal{F}$ is a coherent sheaf on the normal projective curve $C^\nu$ so we can write $f^*\ecal{F}  = \ecal{E} \oplus \ecal{E}'$ where $\ecal{E}$ is a vector bundle and $\ecal{E}'$ is torsion. We also have $f^*\ecal{F}\otimes f^*\ecal{L} \cong f^*\ecal{F}$ and from this we see that 
    $$\ecal{E} \otimes f^*\ecal{L} \cong \ecal{E}$$
    as both are isomorphic to $f^*\ecal{F}/torsion$. Also, $\ecal{E}$ is not zero since we chose $\cal C \subset \operatorname{Supp}(\ecal{F})$. Let $r > 0$ be its rank. Taking determinants gives $\operatorname{det}(\ecal{E}) \otimes f^*\ecal{L}^{\otimes r} \cong \operatorname{det}(\ecal{E})$, hence $f^*\ecal{L}^{\otimes r} \cong \ecal{O}_{C^\nu}$, as needed.

    The proof that (ii) $\implies$ (iii) $\implies$ (iv) $\implies$ (i) is the same as the proof of Proposition \ref{prop-justificationofname}.
\end{proof}
\begin{corollary}
\label{corollary-powersofnowheretorsiononstack}
    Let $\ecal L$ be a line bundle on 
    $\cal X$. Then, $\ecal L$ is nowhere torsion if and only if $\ecal L^{\otimes n}$ is nowhere torsion for some $0 \neq n \in \bb Z$ if and only if $\ecal L^{\otimes n}$ is nowhere torsion for every $0 \neq n \in \bb Z$. 
\end{corollary}
\begin{proof}
    Follows from the equivalence of (i) and (ii) in the previous proposition.
\end{proof}

\begin{lemma}
\label{lemma-properquasifinitepullback}
    Let $f: \cal Y \to \cal X$ be a proper, quasi-finite, surjective morphism of finite type algebraic stacks over $k$ with finite diagonal. Then a line bundle $\ecal L$ on $\cal X$ is nowhere torsion if and only if $f^* \ecal L$ is nowhere torsion on $\cal Y$. 
\end{lemma}

\begin{proof}
    If $f^*\ecal{L}$ is somewhere torsion, then there is a finite morphism $g: C \to \mathcal{Y}$ with $C$ a normal projective curve such that $(f \circ g)^*\ecal{L}= g^*(f^*\ecal{L})$ is torsion. Then $f \circ g : C \to \mathcal{X}$ is quasi-finite, proper, and representable since $\mathcal{X}$ is an algebraic stack and $C$ is a scheme, hence finite, and so we see that $\ecal{L}$ is somewhere torsion.

    Conversely, if $\ecal{L}$ is somewhere torsion, then there is a coherent sheaf $\ecal{F}$ on $\mathcal{X}$ whose support is proper of positive dimension and such that $\ecal{F} \otimes \ecal{L} \cong \ecal{F}$. Then $f^*\ecal{F}$ has proper positive dimensional support on $\mathcal{Y}$ and satisfies $f^*\ecal{F} \otimes f^*\ecal{L} \cong f^*\ecal{F},$ so $f^*\ecal{L}$ is somewhere torsion.
\end{proof}

The coarse space morphism $\pi : \mathcal{X} \to X$ satisfies the hypotheses above, so we obtain the following.

\begin{corollary}
\label{corollary-coarsespacetest}
    Let $\pi : \cal X \to X$ be the coarse moduli space. Let $\ecal L$ be a line bundle on $X$. Then $\ecal L$ is nowhere torsion if and only if $\pi^*\ecal{L}$ is nowhere torsion.
\end{corollary}

\begin{remark}
\label{remark-testforgoodnessonstack}
    For every line bundle $\ecal{L}$ on $\mathcal{X}$, there is an integer $n >1$ such that $\ecal{L}^{\otimes n} \cong \pi^*\ecal{M}$ for a line bundle $\ecal{M}$ on $X$, see \cite{206117} (This also follows from \cite[Proposition 6.1]{Olsson2012Integral} in the case $\cal X$ is tame). Then using Corollaries \ref{corollary-coarsespacetest} and \ref{corollary-powersofnowheretorsiononstack}, we see that $\ecal{L}$ is nowhere torsion on $\mathcal{X}$ if and only if $\ecal{M}$ is nowhere torsion on $X$.
\end{remark}

%
When we have an explicit finite cover of $\cal X$ by a scheme, we can easily check if a line bundle is nowhere torsion.
\begin{example}\label{example-weightedprojectivestack}
    Let $\cal X = \cal P(a_0,\dots,a_n)$ be any weighted projective stack. Then $\operatorname{Pic}(\cal X) = \mathbb{Z} \cdot \ecal{O}_{\mathcal{X}}(1)\cong \mathbb{Z}$ and morphisms $f : T \to \cal X$ are determined by $\mathcal L = f^*\ecal{O}_{\cal{X}}(1)$ and a tuple $(s_0, \dots  , s_n)$ where $s_i \in \Gamma(T, \cal L^{\otimes a_i})$ and $\bigcup T_{s_i} = T$. Thus $\cal X$ admits a finite cover $\pi: \bb P^n \to \cal X$ corresponding to the tuple $(\ecal O_{\bb P^n}(1), x_0^{a_0}, \dots, x_n^{a_n})$. Since $\pi^* \ecal O_\cal X(1) = \ecal O_{\bb P^n}(1)$ by construction, any non-trivial line bundle on $\cal X$ is nowhere torsion. On the other hand, not every non-trivial line bundle on a weighted projective stack is $\otimes$-generating, see Example \ref{example-weightedprojectivestacktensorgen} below.
\end{example} 

\begin{lemma}
\label{lemma-checkafterbasechange}
    Let $\ecal{L}$ be a line bundle on $\mathcal{X}$. Let $\bar{k}$ be an algebraic closure of $k$. The line bundle $\ecal{L}_{\bar{k}}$ on $\mathcal{X}_{\bar{k}}$ is nowhere torsion if and only if $\ecal{L}$ is.
\end{lemma}

\begin{proof}
    If $\ecal{L}$ is somewhere torsion it is easy to see that $\ecal{L}_{\bar{k}}$ is somewhere torsion. Conversely, if $\ecal{L}_{\bar{k}}$ is somewhere torsion, there is a coherent sheaf $\ecal{F}$ on $\mathcal{X}_{\bar{k}}$ whose support is proper of positive dimension and such that $\ecal{F} \otimes \ecal{L}_{\bar{k}} \cong \ecal{F}$. There is a finite extension $k'/k$ and a coherent sheaf $\ecal{G}$ on $\mathcal{X}_{k'}$ with $\ecal{G}_{\bar{k}} = \ecal{F}$. Possibly enlarging $k$ we may assume the support of $\ecal{G}$ is proper and of positive dimension and satisfies $\ecal{G}\otimes \ecal{L}_{k'} \cong \ecal{G}$. Thus $\ecal{L}_{k'}$ is somewhere torsion. Since $X_{k'} \to X$ is finite, we see $\ecal{L}$ is somewhere torsion by Lemma \ref{lemma-properquasifinitepullback}.
\end{proof}

\subsection{Grothendieck duality on a proper tame stack over a field}
\label{subsection-dualityforstacks}

We use the results of \cite[Section 3]{hall2024generalizedbondalorlovfaithfulnesscriterion} to deduce some results about Grothendieck duality for smooth proper tame stacks over a field. In the case of a tame and Deligne--Mumford stack, these results appear in \cite{nironi2009grothendieck}. We begin with some well-known calculations of right adjoints of pushforwards.

\begin{lemma}
\label{lemma-regularclosedimmersioncalc}
    Let $i : \mathcal{Z} \to \mathcal{X}$ be a regular closed immersion of codimension $c$ between Noetherian algebraic stacks with corresponding ideal sheaf $\ecal{I}$ and normal bundle $\ecal{N} = (\ecal{I}/\ecal{I}^2)^\vee$. Then 
    \begin{enumerate}
        \item The functor
    $$
    i^! : D^b_{\operatorname{Coh}}(\mathcal{X}) \to D^b_{\operatorname{Coh}}(\mathcal{Z}), \hspace{3 em} i^!(K) = Li^*(K) \otimes ^{\mathbb{L}} \wedge^c\ecal{N}[-c]
    $$
    is right adjoint to $Ri_* : D^b_{\operatorname{Coh}}(\mathcal{Z}) \to D^b_{\operatorname{Coh}}(\mathcal{X})$.
        \item If $\ecal{F}$ is a vector bundle on $\mathcal{Z}$, then $H^{-p}(Li^*i_*\ecal{F}) \cong \wedge ^p\ecal{I}/\ecal{I}^2 \otimes _{\ecal{O}_{\mathcal{Z}}}\ecal{F}$. \qedhere 
    \end{enumerate}
\end{lemma}

\begin{proof}
    Note that $i$ is a finite and perfect morphism. Therefore, by finite duality \cite[Theorem 4.14]{hall2017perfect}, to prove (i) it is enough to show $\mathcal{E}xt^i(i_*\ecal{O}_\mathcal{Z}, \ecal{O}_{\mathcal{X}}) = 0$ for $i \neq c$ and $\mathcal{E}xt^c(i_*\ecal{O}_{\mathcal{Z}}, \ecal{O}_{\mathcal{X}}) \cong \wedge^c\ecal{N}$. Here the local Ext sheaves are computed on the lisse-\'etale site of $\mathcal{X}$. This follows from the arguments given in  \cite[\href{https://stacks.math.columbia.edu/tag/0BQZ}{Tag 0BQZ}]{stacks-project}. That is, the proof there works with ``ringed space'' replaced by ``ringed site.''

    By Lemma \ref{lemma-backwardprojectionformula} below, the objects $Li^*i_*\ecal{F}$ and $\ecal{F} \otimes ^{\mathbb{L}}_{\ecal{O}_{\mathcal{Z}}}Li^*i_*\ecal{O}_{\mathcal{Z}}$ have isomorphic cohomology sheaves. Thus it is enough to show $H^{-p}(Li^*i_*\ecal{O}_{\mathcal{Z}}) \cong \wedge^p\ecal{I}/\ecal{I}^2$. We have
    $$
    H^{-p}(i^!i_*\ecal{O}_{\mathcal{Z}} ) \cong \mathcal{E}xt^p_{\ecal{O}_{\mathcal{X}}}(\ecal{O}_{\mathcal{Z}}, \ecal{O}_{\mathcal{Z}}) \cong \wedge^p\ecal{N}
    $$
    by the proofs of \cite[\href{https://stacks.math.columbia.edu/tag/0BQY}{Tags 0BQW and 0BQY}]{stacks-project}. Thus using part (i), we have
    \[H^{-p}(Li^*i_*\ecal{O}_{\mathcal{Z}}) = H^{-p}(i^!i_*\ecal{O}_Z \otimes ^{\mathbb{L}} \wedge ^c\ecal{I}/\ecal{I}^2 [c])= H^{c-p}(i^!i_*\ecal{O}_Z)\otimes \wedge^c\ecal{I}/\ecal{I}^2 = \wedge^{p-c} \cal N \otimes \wedge^c\ecal{I}/\ecal{I}^2 \cong \wedge^p\ecal{I}/\ecal{I}^2. \qedhere \]
\end{proof}

\begin{lemma}
\label{lemma-backwardprojectionformula}
    Let $f : \mathcal{Y} \to \mathcal{X}$ be an affine morphism of algebraic stacks. Let $K \in D_{qc}(\mathcal{Y})$. There are isomorphisms
    $$
    Rf_*Lf^*Rf_*K \cong Rf_*(K \otimes^{\mathbb{L}}Lf^*Rf_*\ecal{O}_{\mathcal{Y}})
    $$
    in $D_{qc}(\mathcal{X})$. In particular, for every integer $i$, there are isomorphisms of quasi-coherent $\ecal{O}_{\mathcal{X}}$-modules
    \[
    f_*H^i(Lf^*Rf_*K) \cong f_*H^i(K \otimes^{\mathbb{L}}Lf^*Rf_*\ecal{O}_{\mathcal{Y}}). \qedhere
    \]
\end{lemma}

\begin{proof}
    Under the identification of $D_{qc}(\mathcal{Y})$ with $D_{qc}(\mathcal{X}, f_*\ecal{O}_{\mathcal{Y}})$ of  \cite[Corollary 2.7]{hall2017perfect}, we see that the left hand side is isomorphic to
    $$
    K \otimes ^{\mathbb{L}}_{\ecal{O}_{\mathcal{X}}}f_*\ecal{O}_{\mathcal{Y}}
    $$
    while the right hand side is isomorphic to
    $$
    K \otimes^{\mathbb{L}}_{f_*\ecal{O}_{\mathcal{Y}}}(f_*\ecal{O}_{\mathcal{Y}}\otimes_{\ecal{O}_{\mathcal{X}}}^{\mathbb{L}}f_*\ecal{O}_{\mathcal{Y}}),
    $$
    so we conclude by associativity of derived tensor products.
\end{proof}

\begin{lemma}
\label{lemma-rootstacks}
    Let $\mathcal{X}$ be a Noetherian, tame algebraic stack. Let $\mathcal{D} \subset \mathcal{X}$ be an effective Cartier divisor. Let $r \geq 1$ an integer. Let $\pi : \mathcal{Y} = \sqrt[r]{(\mathcal{X}, \mathcal{D})} \to \mathcal{X}$ be the associated root stack morphism. Let $\mathcal{E} \subset \mathcal{Y}$ be the tautological $r$th root of the effective Cartier divisor $\pi^{-1}(\mathcal{D})$. Then the right adjoint of $R \pi_* : D^b_{\operatorname{Coh}}(\mathcal{Y}) \to D^b_{\operatorname{Coh}}(\mathcal{X})$ is given by
    $
    \pi^!(-) = L\pi^*(-)\otimes^{\mathbb{L}} \ecal{O}_{\mathcal{Y}}((r-1)\mathcal{E}).
    $
\end{lemma}

\begin{proof}
    It suffices by \cite[Theorem 3.1]{hall2024generalizedbondalorlovfaithfulnesscriterion} to give an isomorphism $\pi^!(\ecal{O}_\mathcal{X}) \cong \ecal{O}_{\mathcal{Y}}((r-1)\mathcal{E})$. Let $t \in \Gamma(\mathcal{Y}, \ecal{O}_{\mathcal{Y}}(\mathcal{E}))$ be the section whose zero scheme is $\mathcal{E}$. Then applying the functor $\pi_*$ to $t : \ecal{O}_{\mathcal{Y}}(-\mathcal{E}) \to \ecal{O}_{\mathcal{Y}}$ gives a morphism 
    $$
   \alpha :  \pi_*\ecal{O}_{\mathcal{Y}}(-\mathcal{E}) \to \pi_*\ecal{O}_{\mathcal{Y}} = \ecal{O}_{\mathcal{X}},
    $$
    which corresponds by adjunction to a morphism $\beta: \ecal{O}_{\mathcal{Y}}(-\mathcal{E}) \to \pi^!(\ecal{O}_{\mathcal{X}})$. We will show the target of $\beta$ is a sheaf and $\beta$ is injective with image $\pi^!(\ecal{O}_{\mathcal{X}})(-r \cdot \mathcal{E})$, which will complete the proof. These checks can be done flat locally on $\mathcal{X}$ since the formation of $\pi^!$ commutes with flat base change by \cite[Theorem 3.1]{hall2024generalizedbondalorlovfaithfulnesscriterion}, and we may therefore assume $\mathcal{D} = \operatorname{Spec} A/f \subset \operatorname{Spec} A = \mathcal{X}$ for some ring $A$ and non-zero-divisor $f \in A$. Set $B = A[T]/(T^r-f)$, a $\mathbb{Z}/r$-graded $A$-algebra where $T$ has degree $1$. Then $\mathcal{Y} = [\operatorname{Spec}B/\mu_r]$. The category of coherent sheaves on $\mathcal{Y}$ is equivalent to the category of finitely generated graded $B$-modules, and $\pi_* : \operatorname{Coh}\mathcal{Y} \to \operatorname{Coh}\mathcal{X}$ corresponds to the functor $M \mapsto M_0$ to finitely generated $A$-modules. As in the ungraded case, one checks that the functor
    $$
    \{\text{finite $A$-modules}\} \to \{\text{finite graded $B$-modules}\}, \hspace{3 em} M \mapsto \operatorname{Hom}_A(B, M) = \bigoplus _{i = 0}^{r-1} \operatorname{Hom}_A(B_i, M)
    $$
    is the right adjoint, where the summand $\operatorname{Hom}_A(B_i, M)$ is declared homogeneous of degree $-i$. Since both functors in the adjunction are exact, they also induce an adjunction on the level of bounded derived categories. Under these identifications, the morphism $\alpha$ defined above is given by multiplication by $T$
    $$
    \alpha : B_{-1} \to  B_0, \hspace{3 em} aT^{r-1} \mapsto aT^r = fa,
    $$
    and its adjoint $\beta$ is the map $B(-1) \to \operatorname{Hom}_A(B, A)$ corresponding to the homogeneous element $\alpha \in \operatorname{Hom}_A(B, A)_{1} = \operatorname{Hom}_A(B_{-1}, A)$. It is a routine check that as a $B$-module, $\operatorname{Hom}_A(B, A)$ is freely generated by the element $\lambda_{r-1}$, where $\lambda_0, \dots , \lambda_{r-1}$ is the dual basis to the basis $1, T, \dots , T^{r-1}$ of $B$ as an $A$-module. We see that $\beta = f \cdot \lambda_{r-1}$, so that $\beta : B(-1) \to \operatorname{Hom}_A(B, A) = B\cdot \lambda _{r-1}$ is given by $b \mapsto bf\lambda_{r-1} = bT^r \lambda _{r-1}$, completing the proof.
\end{proof}

For the rest of the section, let $\mathcal{X}$ be a tame algebraic stack proper over a field $k$. Then there is a unique up to isomorphism object $\omega_{\mathcal{X}}^\bullet \in D^b_{\operatorname{Coh}}(\mathcal{X})$ such that there are natural isomorphisms
$$
\operatorname{Hom}_{{\mathcal{X}}}(K, \omega_{\mathcal{X}}^\bullet) = \operatorname{Hom}_k(R \Gamma(\mathcal{X}, K), k)
$$
for $K \in D^b_{\operatorname{Coh}}(\mathcal{X})$. Namely, we can take $\omega_{\mathcal{X}}^\bullet = p^! \ecal{O}_{\operatorname{Spec}k}$ where $p : \mathcal{X} \to \operatorname{Spec} k$ is the structure morphism and $p^! : D^b_{\operatorname{Coh}}(\operatorname{Spec}k) \to D^b_{\operatorname{Coh}}(\mathcal{X})$ is the right adjoint to $Rp_*$, see \cite[Theorem 3.1]{hall2024generalizedbondalorlovfaithfulnesscriterion}. 
If additionally, $\mathcal{X}$ is 
Gorenstein of pure dimension $d$, then $\omega_{\mathcal{X}}^\bullet = \omega_{\mathcal{X}}[d]$ for a line bundle $\omega_{\mathcal{X}}$ on $\mathcal{X}$ \cite[Corollary 3.2]{hall2024generalizedbondalorlovfaithfulnesscriterion}. A regular algebraic stack is Gorenstein. 

\begin{lemma}
\label{lemma-isserrefunctor}
    If $\mathcal{X}$ is Gorenstein, then the functor $K \mapsto K \otimes ^{\mathbb{L}}\omega_{\mathcal{X}}^{\bullet}$ is a Serre functor on $D_{perf}(\mathcal{X})$.
\end{lemma}

\begin{proof}
    We have, for $K, L \in D_{perf}(\mathcal{X})$,
    $$
    \operatorname{Hom}_{{\mathcal{X}}}(K, L \otimes \omega_{\mathcal{X}}^\bullet) = \operatorname{Hom}_{\ecal{O}_{\mathcal{X}}}(R \mathcal{H}om(L, K), \omega_{\mathcal{X}}^\bullet) = \operatorname{Hom}(R\operatorname{Hom}(L, K), k) = \operatorname{Hom}(L, K)^*.
    $$
    The assumption that $\mathcal{X}$ is Gorenstein ensures that the functor $K \mapsto K \otimes ^{\mathbb{L}} \omega_{\mathcal{X}}^\bullet$ preserves $D_{perf}(\mathcal{X})$. 
\end{proof}

\begin{lemma}
\label{lemma-adjunctionformula}
    Let $i: \mathcal{Y} \to \mathcal{X}$ be a regular closed immersion of codimension $c$ with normal bundle $\ecal{N}$. 
    Then $\omega_{\mathcal{Y}}^\bullet = Li^*(\omega_{\mathcal{X}}^\bullet) \otimes ^{\mathbb{L}} \wedge^c \ecal{N}[-c]$
\end{lemma}

\begin{proof}
    This follows formally from Lemmas \ref{lemma-regularclosedimmersioncalc} and \ref{lemma-isserrefunctor}: For $K \in D^b_{\operatorname{Coh}}(\mathcal{Y})$, we have
    \[
    \operatorname{Hom}_{\mathcal{Y}}(K, Li^*(\omega_{\mathcal{X}}^\bullet) \otimes^{\mathbb{L}}\wedge^c\ecal{N}[-c]) = \operatorname{Hom}(Ri_*(K), \omega_{\mathcal{X}}^\bullet) = \operatorname{Hom}( R\Gamma(Ri_*(K)), k) = \operatorname{Hom}(R\Gamma(\mathcal{Y}, K), k). \qedhere 
    \]
    \end{proof}

\begin{example}
If $G$ is a tame group scheme over $k$ and $\mathcal{X} = BG_k$, then $\omega_{\mathcal{X}}^\bullet$ is the trivial representation $k$ in degree zero. Namely, for a finite-dimensional representation $V$ of $G$, we have
    $$
    \operatorname{Hom}_G(V, k) = (V^*)^G = (V^G)^*
    $$
   since $G$ is linearly reductive. Since the category of finite-dimensional representations of $G$ is semi-simple and $D^b_{\operatorname{Coh}}(\mathcal{X})$ is its bounded derived category, it follows that there are natural isomorphisms 
    $$
    R\operatorname{Hom}_{\ecal{O}_{\mathcal{X}}}(K, k) = R\operatorname{Hom}_k(R\Gamma(\mathcal{X}, K),k)
    $$
    for all objects $K \in D^b_{\operatorname{Coh}}(\mathcal{X})$, as needed.
\end{example}

\subsection{Coherent sheaves on a gerbe}

We will need some information on the category of coherent sheaves on a gerbe in order to understand the generalized closed point sheaves on a stack. So, let $k$ be a field and $\mathcal{G} \to \operatorname{Spec}k$ a gerbe with finite diagonal. Then there are a finite field extension $k'/k$ and a finite group scheme $G/k'$ such that $\mathcal{G}_{k'}\cong BG_{k'}$ as stacks over $k'$. The induced morphism $\operatorname{Spec}k' \to \mathcal{G}$ is a finite locally free surjective morphism, and using this, one can see that (1) $\mathcal{G}$ is Noetherian; (2) every coherent sheaf on $\mathcal{G}$ is a vector bundle (hence $\perf \cal G = D^b_{\operatorname{Coh}}(\cal G)$); and (3) the category of coherent sheaves on $\mathcal{G}$ is a finite length category, see \cite[Proposition 7.1]{hall2024generalizedbondalorlovfaithfulnesscriterion}. 

\begin{lemma}
    Let $x : \operatorname{Spec}k' \to \mathcal{G}$ be any morphism over $k$ where $k'/k$ is a finite extension of fields. Then every coherent sheaf on $\mathcal{G}$ is a quotient of a finite direct sum of copies of $x_*\ecal{O}_{\operatorname{Spec}k'}$.
\end{lemma}

\begin{proof}
    Let $x^! : \operatorname{Coh}(\mathcal{G}) \to \operatorname{Coh}(\operatorname{Spec}k')$ be the right adjoint of $x_*$. Recall that since $x$ is (representable and) finite locally free, this is given by the formula $x_*x^!\ecal{F} = \mathcal{H}om_{\mathcal{G}}({x_*\ecal{O}_{\operatorname{Spec}k'}}, \ecal{F})$ as an ${x_*\ecal{O}_{\operatorname{Spec}k'}}$-module. Then given any object $\ecal{F} \in \operatorname{Coh}(\mathcal{G})$, choose a surjection $\ecal{O}_{\operatorname{Spec}k'}^{\oplus N} \to x^!\ecal{F}$, and then apply $x_*$ to get surjective maps
    $$
    x_*\ecal{O}_{\operatorname{Spec}k'}^{\oplus N} \to x_*x^! \ecal{F} \to \ecal{F}.
    $$
    The second arrow is surjective because it is identified with $\mathcal{H}om(\alpha, \ecal{F})$ where $\alpha$ is the locally split injection $\ecal{O}_\mathcal{G} \to x_*\ecal{O}_{\operatorname{Spec}k'}$. 
\end{proof}

\begin{corollary}
\label{corollary-finitelymanysimples}
    The category of coherent sheaves on $\mathcal{G}$ has finitely many simple objects up to isomorphism. The group $\operatorname{Pic}(\mathcal{G})$ is finite. 
\end{corollary}

\begin{proof}
    By the previous lemma and the Jordan--H\"older Theorem, every simple object occurs in a composition series for $x_*\ecal{O}_{\operatorname{Spec}k'}$ and only finitely many objects up to isomorphism occur in that way. The second part is because a line bundle on $\mathcal{G}$ is a simple object of the category of coherent sheaves on $\mathcal{G}$ since every coherent sheaf on $\mathcal{G}$ is a vector bundle. 
\end{proof}

When $\mathcal{G}$ is additionally assumed tame, we can say more.

\begin{lemma}
\label{lemma-tamegerbe}
    Suppose $\mathcal{G}$ is tame. Then the category of coherent sheaves on $\mathcal{G}$ is a semi-simple abelian category with finitely many simple objects, and all simple objects occur as a direct summand of the sheaf $x_*\ecal{O}_{\operatorname{Spec}k'}$ where $x : \operatorname{Spec}k' \to \mathcal{G}$ is any morphism over $k$ with $k'/k$ a finite extension of fields.
\end{lemma}

\begin{proof}
    We have seen that every object of the category of coherent sheaves on $\mathcal{G}$ is a vector bundle, so it suffices to note that every vector bundle on $\mathcal{G}$ is a projective object of the category of (quasi-)coherent sheaves since $\mathcal{G}$ is a tame stack with affine coarse space. 
\end{proof}
 
In particular, higher Ext groups between simple coherent sheaves vanish, and any object of $\perf \cal G = D^b_{\operatorname{Coh}}(\cal G) = D^b(\operatorname{Coh}(\cal G))$ is (quasi-isomorphic to) a finite direct sum of shifts of simple coherent sheaves.

\begin{lemma}
    Assume $\mathcal{G}$ is tame. Let $\ecal{E}$ be a vector bundle on $\mathcal{G}$. Then $\mathbb{V}(\ecal{E}) = \underline{\operatorname{Spec} }_{\mathcal{G}} \operatorname{Sym}(\ecal{E})$ has a dense open which is a scheme if and only if $\ecal{E}$ is a faithful vector bundle on $\mathcal{G}$.
\end{lemma}

\begin{remark}
    The ``if'' direction is not true without the assumption that $\mathcal{G}$ is tame. For example, if $k$ has characteristic $p>0$, $G = \alpha_p \times \alpha_p,$ then the representation
    $$ \alpha_p \times \alpha_p \to GL_{4,k}, \hspace{3 em}
    (s,t) \mapsto
    \begin{pmatrix}
        1 & s & t & 0 \\
        0 & 1 & 0 & 0 \\
        0 & 0 & 1  & 0 \\
        0 & 0 & 0 &1
    \end{pmatrix}$$
    is faithful, but the induced action $G \times \mathbb{A}_k^4 \to \mathbb{A}_k^4$ has non-trivial generic stabilizer. 
\end{remark}

\begin{proof}
    This reduces to the case $k$ is algebraically closed, hence $\mathcal{G} = BG_k$ for some finite linearly reductive group scheme $G/k$; and $\mathbb{V}(\ecal{E}) = \mathbb{A}^n_k/G$ where $G$ acts through the associated representation $G \hookrightarrow GL_{n, k}$. By the classification of finite linearly reductive group schemes over algebraically closed fields of \cite{Abramovich2008}, we have $G = \Delta \rtimes H$ where $\Delta$ is a connected diagonalizable group scheme over $k$ and $H$ is a finite constant group scheme of order invertible in $k$. The composition $H \to GL_{n, k}$ is faithful and $H$ is a constant group scheme so a standard argument shows there is a dense open $U \subset \mathbb{A}^n_k$ which is $H$-invariant and such that $H$ acts freely on $U$. Then for any $u \in U(k)$ we claim that the stabilizer group scheme of $u$ is connected, hence contained in $\Delta$. If not, it has some $k$-point $1 \neq g \in G(k)$, and then, writing $p : G \to H$ for the projection,  $gu = p(g)u \neq u$ in $U(k)$. Thus it suffices to show $\Delta$ acts freely on a dense open $V$ of $\mathbb{A}^n_k$, for then $G$ will act freely on $U \cap V$. For this, since $\Delta$ is a diagonalizable group scheme, there is a basis of $k^n$ in which $\Delta \to GL_{n, k}$ factors as $\Delta \hookrightarrow \mathbb{G}_{m, k}^n \hookrightarrow GL_{n, k}$. But then $\Delta$ acts freely on the dense open $\mathbb{G}_{m, k}^n \subset \mathbb{A}^n_k$, and we are done. 
\end{proof}

\begin{lemma}
\label{lemma-tensorproductsoffaithful}
    Assume $\mathcal{G}$ is tame. Let $\ecal{E}$ be a faithful vector bundle on $\mathcal{G}$. Then every simple object of the category of coherent sheaves on $\mathcal{G}$ occurs as a direct summand of $\ecal{E} ^{\otimes n}$ for some $n \geq 1$.
\end{lemma}

\begin{proof}
    This is an adaptation of the nice argument from the blog post \cite{litt2022tensor}. By the preceding lemma, the algebraic stack $\mathbb{V}(\ecal{E})$ has a dense open which is representable by a scheme, necessarily of finite type over $k$. By the Nullstellensatz, there is a closed immersion $\operatorname{Spec}k' \to \mathbb{V}(\ecal{E})$ over $\mathcal{G}$ where $k'/k$ is a finite extension of fields. Denote $p : \mathbb{V}(\ecal{E}) \to \mathcal{G}$ the structure morphism and $x : \operatorname{Spec}k' \to \mathcal{G}$ the composition of the inclusion and $p$. Then we obtain a surjective map of quasi-coherent $\ecal{O}_{\mathcal{G}}$-modules $\operatorname{Sym}(\ecal{E}) = p_*\ecal{O}_{\mathbb{V}(\ecal{E})} \to x_*\ecal{O}_{\operatorname{Spec}k'}$. Since $x_*\ecal{O}_{\operatorname{Spec}k'}$ is coherent, there is in fact a finite direct sum of objects $\operatorname{Sym}^n(\ecal{E})$ which surject onto $x_*\ecal{O}_{\operatorname{Spec}k'}$, and each is itself a quotient of $\ecal{E}^{\otimes n}$, completing the proof since $\operatorname{Coh}\mathcal{G}$ is semisimple.
\end{proof}
\begin{lemma}
\label{lemma-cyclicstabilizers}
    Suppose $\mathcal{G}$ is tame and there is a line bundle $\ecal{L}$ on $\mathcal{G}$ such that every coherent sheaf on $\mathcal{G}$ is a quotient of a finite direct sum of tensor powers of $\ecal L$. Then there is an integer $n > 0$ such that $\mathcal{G} \cong B\mu_{n, k}$.
\end{lemma}

\begin{proof}
    We claim that $\ecal{L}$ is a faithful vector bundle on $\mathcal{G}$. We know there exists a faithful vector bundle $\ecal{E}$ on $\mathcal{G}$ and by the assumptions we must have $\ecal{E} \cong \ecal{L}^{\otimes i_1} \oplus \cdots \oplus \ecal{L}^{\otimes i_k}$. Let $x : \operatorname{Spec}k' \to \mathcal{G}$ be any point. Then the associated representation
    $$
    G_x = \underline{\operatorname{Aut}}_{k'}(x) \to GL_{x^*\ecal{E}}
    $$
    is faithful, but it factors through the homomorphism
    $$
    \varphi: G_x \to GL_{x^*\ecal{L}^{\otimes i_1}} \times \cdots \times GL_{x^*\ecal{L}^{\otimes i_k}} = \mathbb{G}_{m, k'}^k,
    $$
    which is equal to $(\chi^{i_1}, \dots , \chi^{i_k})$, where $\chi : G_x \to GL_{x^*\ecal{L}}$ is the representation coming from $\ecal{L}$. It follows that $\varphi$ is faithful, and then since $\operatorname{Ker}\chi \subset \operatorname{Ker \varphi}$, it follows that $\chi$ is faithful. This proves the claim.

    Then consider the representation $\psi : I_{\mathcal{G}} \hookrightarrow \mathbb{G}_{m, \mathcal{G}}$ of the inertia stack determined by the line bundle $\ecal{L}$. By the claim, it is a monomorphism, and $I_{\mathcal{G}}$ is a finite flat group scheme over $\mathcal{G}$, so it follows that 
   $I_{\mathcal{G}} \cong \mu_{n, \mathcal{G}}$ for some $n  > 0$. Furthermore, since there is a line bundle $\ecal{L}$ on $\mathcal{G}$ on which $I_{\mathcal{G}}$ acts through $\psi : \mu_{n, \mathcal{G}} \hookrightarrow \mathbb{G}_{m, \mathcal{G}}$, there is an invertible twisted sheaf on the $\mathbb{G}_m$-gerbe associated to $\mathcal{G}$ via $H^2(k, \mu_n) \to H^2(k, \mathbb{G}_m)$. This 
   implies that the class of $\mathcal{G}$ maps to zero in $H^2(k, \mathbb{G}_m)$, but $H^2(k, \mu_n) \hookrightarrow H^2(k, \mathbb{G}_m)$ by the Kummer sequence, since $\operatorname{Pic}(\operatorname{Spec}k) = 0$. Hence $\mathcal{G}$ is the trivial $\mu_n$-gerbe, as needed.
\end{proof}
To summarize, we have the following:

\begin{lemma}\label{lemma-tensor-gen-on-gerbe}
    Assume $\cal G$ is tame and take a line bundle $\ecal L$ on $\cal G$. The following are equivalent. 
    \begin{enumerate}
        \item $\ecal L$ is $\otimes$-generating, i.e., the set $\{\ecal{L}^{\otimes n}\}_{n \in \mathbb{Z}}$ classically generates $\perf \mathcal{G}$.
        \item Every simple coherent sheaf is isomorphic to a tensor power of $\ecal L$. 
        \item Every coherent sheaf is a finite direct sum of tensor powers of $\ecal L$.
        \item Every coherent sheaf is a quotient of a finite direct sum of tensor powers of $\ecal L$.
        \item $\ecal L$ is faithful.
        \item $\cal G \cong B\mu_{N, k}$ for some $N \geq 1$ and $\ecal L$ corresponds to a generator of $\bb X(\mu_{N, k}) = \mathbb{Z}/N\mathbb{Z}$. \qedhere 
    \end{enumerate}
\end{lemma}
\begin{proof}
    We note that any line bundle on $\mathcal{G}$ is indecomposable, and since $\operatorname{Coh}\mathcal{G}$ is semi-simple, this implies that line bundles are simple objects. To see (i) $\implies$ (ii), take a simple coherent sheaf $S$ on $\cal G$. If $S \not \cong \ecal L^{\otimes n}$ for any $n \in \bb Z$, then we have 
    \[
    \hom_{\perf \cal G} (\ecal L^{\otimes n}, S[i]) = 0
    \]
    for any $n, i \in \bb Z$, as $\operatorname{Coh}(\cal G)$ is semisimple and $\perf \cal G = D^b_{\operatorname{Coh}}(\cal G) = D^b(\operatorname{Coh} (\cal G))$, which contradicts (i). 
    (ii) $\implies$ (iii) $\implies$ (iv) is clear. (iv) $\implies$ (v) $\implies$ (vi) follows from the proof of Lemma \ref{lemma-cyclicstabilizers}. Finally, (vi) clearly implies (i). 
\end{proof}

\begin{definition}
\label{defn-pointwisetensorgenerating}
    Let $\mathcal{X}$ be a Noetherian, tame algebraic stack (so by our conventions, $\mathcal{X}$ has finite diagonal). A line bundle $\ecal{L}$ on $\mathcal{X}$ is \emph{point-wise $\otimes$-generating} if for every closed point $x$ of $\mathcal{X}$ with residual gerbe $i_x : \mathcal{G}_x \to \mathcal{X}$, the line bundle $i_x^*\ecal{L}$ on $\mathcal{G}_x$ satisfies the equivalent conditions of Lemma \ref{lemma-tensor-gen-on-gerbe}.
\end{definition}

\begin{lemma}
    Let $\mathcal{X}$ be a Noetherian, tame algebraic stack. Let $\ecal{L}$ be a line bundle on $\mathcal{X}$. If $\ecal{L}$ is $\otimes$-generating, i.e., the set $\{\ecal{L}^{\otimes n}\}$ compactly generates $D_{qc}(\mathcal{X}),$ then $\ecal{L}$ is point-wise $\otimes$-generating.
\end{lemma}

\begin{proof}
    This follows from the fact that each $i_x$ is affine and (quasi-)affine pullbacks preserve compact generation by \cite{hall2017perfect}*{Theorem A, Corollary 2.8}.
\end{proof}

\begin{example}
\label{example-weightedprojectivestacktensorgen}
   Consider the line bundle $\ecal{L} = \ecal O_\cal X(d)$ on the weighted projective stack $\mathcal{X} = \mathcal{P}(a_0, \dots , a_n)$ over a field $k$. We claim it is point-wise $\otimes$-generating if and only if $\gcd(d,a_i) = 1$ for every $i$. In particular, it is not $\otimes$-generating if $\gcd(d,a_i)  \neq 1$ for some $i$.
    First, the pullback of $\ecal{L}$ to the residual gerbe $B\mu_{a_i}$ at the $i$-th coordinate point $[0\!:\! \cdots\!:\!  0\!:\!1\!:\!0\!:\!\cdots 0]$ is the $d^{th}$ power of the standard character, which is $\otimes$-generating if and only if $\gcd(d,a_i) = 1$. Conversely, for any closed point of $\mathcal{X}$ with residue field $\kappa$, the residual gerbe is $B\mu_{a, \kappa}$ where $a| a_i$ for at least one $i$. Hence if $\gcd (d, a_i) = 1$ for all $i$ then $\gcd (d, a) = 1$ so the restriction of $\ecal{L}$ to the residual gerbe is $\otimes$-generating.
\end{example}

\subsection{Generalized points and Nakayama's Lemma}
\label{subsection-generalizedpoints}

We use \cite[Section 7]{hall2024generalizedbondalorlovfaithfulnesscriterion} as a reference. If $\mathcal{X}$ is a quasi-separated algebraic stack, the authors define a \emph{generalized (closed) point} of $\mathcal{X}$ to be a pair $(x, \xi)$ where $x$ is a (closed) point of $\mathcal{X}$ and $\xi$ is a simple object of $\operatorname{Coh}(\mathcal{G}_x)$. For such a pair they write $\kappa(x, \xi) = i_{x,*}\xi$ where $i_x : \mathcal{G}_x \to \mathcal{X}$ is the inclusion. They prove a version of Nakayama's Lemma and several consequences, for example:
\begin{enumerate}
    \item \cite[Corollary 7.10]{hall2024generalizedbondalorlovfaithfulnesscriterion} Let $\ecal{F}$ be a quasi-coherent $\ecal{O}_{\mathcal{X}}$-module of finite type. Let $x$ be a point of $\mathcal{X}$. Suppose that for all generalized points of the form $(x, \xi)$, we have $\operatorname{Hom}(\ecal{F}, \kappa(x, \xi)) = 0$. Then $\ecal{F}$ is zero in a neighborhood of $x$. 
    \item \cite[Corollary 7.10]{hall2024generalizedbondalorlovfaithfulnesscriterion} Assume $\mathcal{X}$ is locally Noetherian, let $x \in \mathcal{X}$ be a closed point, and let $K \in D^b_{\operatorname{Coh}}(\mathcal{X})$. The following are equivalent:
    \begin{enumerate}
    \item for all generalized points of the form $(x, \xi)$ we have
    $$
    \operatorname{Hom}(K, \kappa(x, \xi)[-i]) = 0, \hspace{3 em} i \geq n.
    $$
    \item $\tau_{\geq n}K$ is zero in a neighborhood of $x$.
    \end{enumerate}
    In particular, we can detect the degree of the top non-zero cohomology sheaf of $K$ as 
    \begin{equation}
    \label{equn-topcohomologysheaf}
    \operatorname{sup}\{n: \exists \text{ a generalized point } (x, \xi) \text{ such that } \operatorname{Hom}(K , \kappa(x, \xi)[-n]) \neq 0\}.
    \end{equation}
\end{enumerate}

\begin{lemma}
\label{lemma-generatingset}
    Let $\mathcal{X}$ be a Noetherian algebraic stack and $x$ a closed point of $\mathcal{X}$. Then $D^b_{\operatorname{Coh}, \{x\}}(K)$ is classically generated by the objects $\kappa(x, \xi)$ where $(x,\xi)$ is a generalized point with underlying point $x$.
\end{lemma}

\begin{proof}
    It suffices to show every coherent sheaf $\ecal{F}$ on $\mathcal{X}$ supported at $x$ is in the subcategory generated by the point objects $\kappa(x, \xi)$ as every object of $D^b_{\operatorname{Coh}, \{x\}}(\mathcal{X})$ is an extension of shifts of such objects. If $\ecal{I} \subset \ecal{O}_{\mathcal{X}}$ is the ideal sheaf of $\mathcal{G}_x \subset\mathcal{X}$ corresponding to the closed point $x$, then $\ecal{I}^n \ecal{F} = 0$ for some $n > 0$ and so $\ecal{F}$ has a filtration $\ecal{F} \supset \ecal{I} \ecal{F} \supset \ecal{I}^2 \ecal{F} \supset \cdots \supset \ecal{I}^n\ecal{F} = 0$ whose successive quotients are pushforwards from $\mathcal{G}_x$. Thus we may assume $\ecal{F}$ is a pushforward of a coherent sheaf from $\mathcal{G}_x$. Then by \cite[Proposition 7.1]{hall2024generalizedbondalorlovfaithfulnesscriterion}, every object of $\operatorname{Coh}(\mathcal{G}_x)$ has a filtration by simple objects. Pushing forward to $\mathcal{X}$, we see that $\ecal{F}$ has a filtration whose successive quotients are of the form $\kappa(x, \xi)$, as needed. 
\end{proof}

\begin{lemma}
\label{lemma-characterizingvectorbundles}
    Let $\mathcal{X}$ be a connected Noetherian algebraic stack which is tame. Let $K\in D^-_{\operatorname{Coh}}(\mathcal{X})$. Then the following are equivalent.
    \begin{enumerate}
        \item There exists a vector bundle $\ecal{E}$ on $\mathcal{X}$ and an integer $n$ such that $K = \ecal{E}[n]$.
        \item For every closed point $x$ of $\mathcal{X}$, there exists at most one integer $i$ such that for some generalized closed point
        $(x, \xi)$ at $x$, we have $\operatorname{Ext}^i(K, \kappa(x, \xi)) \neq 0$. Equivalently, for every closed point $x$ of $\cal X$, there are no distinct integers $i\neq j$ and generalized closed points $\kappa(x,\xi_i),\kappa(x,\xi_j)$ such that both $\ext^i(K,\kappa(x,\xi_i))$ and $\ext^j(K,\kappa(x,\xi_j))$ are non-zero.
    \end{enumerate} 
    If the equivalent conditions hold, then $i = n$ for every generalized closed point $(x, \xi)$. 
    
    Furthermore, if there is a closed point $x \in \cal X$ with trivial stabilizer, the rank of $\ecal{E}$ is equal to \[\dfrac{\operatorname{dim}_{\kappa(x)} \operatorname{Ext}^n(K, \kappa(x))}{\operatorname{dim}_{\kappa(x)}\operatorname{Hom}(\kappa(x), \kappa(x))}. \qedhere \] 
\end{lemma}

\begin{proof}
    The object $K$ is a vector bundle in degree $-n$ in a neighborhood of a closed point $x \in \mathcal{X}$ if and only if $H^p(Li_x^*(K)) = 0$ for all $p \neq -n$, where $i_x : \mathcal{G}_x \to \mathcal{X}$ is the residual gerbe at $\mathcal{X}$. This reduces immediately to the case $\mathcal{X}$ is an affine scheme, whereupon it follows from \cite[\href{https://stacks.math.columbia.edu/tag/068V}{Tag 068V}]{stacks-project}. Because the category of coherent sheaves on $\mathcal{G}_x$ is semi-simple, for any generalized point $(x, \xi)$ and integer $p$ we have
    $$
    \operatorname{Ext}^p(K, \kappa(x,\xi)) = \operatorname{Ext}^p(Li_x^*(K), \xi) = \operatorname{Hom}(H^{-p}(Li_x^*(K)), \xi)
    $$
    so $H^{-p}(Li_x^*(K)) = 0$ if and only if $\operatorname{Ext}^p(K, \kappa(x,\xi)) = 0$ for all generalized points $(x, \xi)$ supported at $x$. This proves the first part of the statement, and the second part is clear.
\end{proof}


\subsubsection{Conditions on functors}
\label{subsection-conditions}

Let $\mathcal{X}, \mathcal{Y}$ be Noetherian algebraic stacks and $F : D^b_{\operatorname{Coh}}(\mathcal{X}) \to D^b_{\operatorname{Coh}}(\mathcal{Y})$ an exact functor. Then we will consider the following conditions on $F$:
\begin{enumerate}
    \item For all generalized closed points $(x, \xi)$ of $\mathcal{X}$, the support of the object $F(\kappa(x, \xi))$ of $D^b_{\operatorname{Coh}}(\mathcal{Y})$ is a single closed point, and if $(x, \xi), (x,\xi')$ are generalized closed points with the same underlying closed point, then $F(\kappa(x,\xi))$ and $F(\kappa(x, \xi'))$ have the same support. 
    \item For all generalized closed points $(x, \xi)$ of $\mathcal{X}$, there is a generalized closed point $(y, \eta)$ of $\mathcal{Y}$ such that $F(\kappa(x, \xi)) \cong \kappa(y, \eta)$. \label{item-assumption1}
\end{enumerate}

\begin{lemma}
\label{lemma-supports}
 Let $\mathcal{X}, \mathcal{Y}$ be Noetherian algebraic stacks. Let $F : D^b_{\operatorname{Coh}}(\mathcal{X}) \to D^b_{\operatorname{Coh}}(\mathcal{Y})$ be an exact equivalence.
    Assume that both $F$ and $F^{-1}$ satisfy condition (i) above. Then there is a well-defined bijection
    $$
    \{\text{set of closed subsets of $\mathcal{X}$}\} \to \{\text{set of closed subsets of $\mathcal{Y}$}\}
    $$
    which takes the support of any object $K \in D^b_{\operatorname{Coh}}(\mathcal{X})$ to the support of $F(K)$. In particular, if $K, L \in D^b_{\operatorname{Coh}}(\mathcal{X})$ have the same support, then so have $F(K), F(L) \in D^b_{\operatorname{Coh}}(\mathcal{Y})$. Moreover, the correspondence preserves finite unions, inclusions, and (co)dimension of closed sets.
\end{lemma}

\begin{proof}
    Under the assumptions, for any closed point $x$ of $\mathcal{X}$, there is a closed point $y$ of $\mathcal{Y}$ such that for every generalized point sheaf $\kappa(x, \xi)$, the support of $F(\kappa(x, \xi))$ is $\{y\}$. By Lemma \ref{lemma-generatingset}, we have
    $$
    F(D^b_{\operatorname{Coh}, \{x\}}(\mathcal{X})) \subset D^b_{\operatorname{Coh}, \{y\}}(\mathcal{Y}).
    $$
    Since $F^{-1}$ satisfies the same hypotheses, we see that in fact $F$ and $F^{-1}$ induce quasi-inverse equivalences 
     $$
    D^b_{\operatorname{Coh}, \{x\}}(\mathcal{X}) \leftrightarrow D^b_{\operatorname{Coh}, \{y\}}(\mathcal{Y}).
    $$
    Thus we see: For $K \in D^b_{\operatorname{Coh}}(\mathcal{X})$ and $x$ a closed point of $\mathcal{X}$,  
    \begin{align*}
    x \in \operatorname{Supp}K &\iff \exists n, (x, \xi) \text{ such that }\operatorname{Hom}(K, \kappa(x,\xi)[n]) \neq 0 \\
    &\iff \exists n, (x, \xi) \text{ such that }\operatorname{Hom}(F(K), F(\kappa(x,\xi))[n]) \neq 0 \\
    & \iff y \in \operatorname{Supp} F(K)
    \end{align*}
    where the last $\iff$ is because the objects $F(\kappa(x,\xi)$ classically generate $D^b_{\operatorname{Coh}, \{y\}}(\mathcal{Y})$ by Lemma \ref{lemma-generatingset} and the first paragraph of the proof, so if $\operatorname{Hom}(F(K), F(\kappa(x,\xi))[n])  = 0$ for all $\xi, n$, then also $\operatorname{Hom}(F(K), \kappa(y, \eta)[n]) = 0$ for all $\eta, n$ because $\kappa(y, \eta)[n] \in D^b_{\operatorname{Coh},\{y\}}(\mathcal{Y})$. 
    
    Thus for $K, L \in \operatorname{D}^b_{\operatorname{Coh}}(\mathcal{X}),$ we have $\operatorname{Supp}K = \operatorname{Supp}L$ implies $\operatorname{Supp} F(K) = \operatorname{Supp} F(L)$. Therefore, $\operatorname{Supp}K \mapsto \operatorname{Supp}F(K)$ is a well-defined function from the set of closed subsets of $\mathcal{X}$ to the set of closed subsets of $\mathcal{Y}$. To obtain the inverse, apply the same arguments to $F^{-1}$.

    Now for the last claim, consider closed subsets $Z_1\subset Z_2 \subset |\cal X|$ with $\ecal K_i \in D^b_{\operatorname{Coh}}(\cal X)$ satisfying $\supp \ecal K_i = Z_i$ ($i = 1,2$). Then, under the correspondence, $Z_1 \cup Z_2$ gets sent to $$\supp F(\ecal K_1 \oplus \ecal K_2) = \supp F(\ecal K_1)\oplus F(\ecal K_2) = \supp F(\ecal K_1) \cup \supp F(\ecal K_2)$$ as desired. In particular, irreducible closed sets correspond to irreducible closed sets. The inclusion preservation follows since for closed subsets $Z_1, Z_2 \subset |\cal X|$, $Z_1 \subset Z_2$ if and only if $Z_1 \cup Z_2 = Z_2$. Since the correspondence preserves irreducible closed sets and inclusions, it also preserves (co)dimension of closed sets.
\end{proof}

\begin{corollary}
\label{corollary-passagetoopens}
    Let $\mathcal{X}, \mathcal{Y}$ be Noetherian algebraic stacks. Let $F : D^b_{\operatorname{Coh}}(\mathcal{X}) \to D^b_{\operatorname{Coh}}(\mathcal{Y})$ be an exact equivalence. 
    Assume that both $F$ and $F^{-1}$ satisfy condition (i) above. Then for any open substack $\mathcal{U} \subset \mathcal{X}$, there is an open substack $\mathcal{V} \subset \mathcal{Y}$ and equivalence of categories $D^b_{\operatorname{Coh}}(\mathcal{U}) \to D^b_{\operatorname{Coh}}(\mathcal{V})$ making the diagram 
\[\begin{tikzcd}
	{D^b_{\operatorname{Coh}}(\mathcal{X})} & {D^b_{\operatorname{Coh}}(\mathcal{Y})} \\
	{D^b_{\operatorname{Coh}}(\mathcal{U})} & {D^b_{\operatorname{Coh}}(\mathcal{V})}
	\arrow[from=1-1, to=1-2]
	\arrow[from=1-1, to=2-1]
	\arrow[from=1-2, to=2-2]
	\arrow[from=2-1, to=2-2]
\end{tikzcd}\]
commute, where the vertical arrows are the restriction functors and the top horizontal arrow is the equivalence induced by $F$, see Lemma \ref{lemma-cohtocoh}. The open substack $\mathcal{V} \subset \mathcal{Y}$ is defined as follows: It is the open substack whose underlying open set $|\mathcal{V}| \subset |\mathcal{Y}|$ is equal to the complement of the support of $F(K)$ where $K$ is any object of $D^b_{\operatorname{Coh}}(\mathcal{X})$ whose support is equal to $|\mathcal{X}| \setminus |\mathcal{U}|$. 

Moreover, if $G : D^b_{\operatorname{Coh}}(\mathcal{X}) \to D^b_{\operatorname{Coh}}(\mathcal{Y})$ is another exact functor satisfying the hypotheses and $\alpha : F \to G$ is an isomorphism of functors, $\alpha$ induces an isomorphism between the two equivalences $D^b_{\operatorname{Coh}}(\mathcal{U}) \to D^b_{\operatorname{Coh}}(\mathcal{V})$.
\end{corollary}

\begin{proof}
    Let $S \subset |\mathcal{X}|$ be the closed complement of the open subset corresponding to $\mathcal{U} \subset \mathcal{X}$. Let $T \subset |\mathcal{Y}|$ be the closed subset to which it corresponds under the bijection of Lemma \ref{lemma-supports}. By that lemma (applied to both $F$ and $F^{-1}$), $F$ induces a commutative square
\[\begin{tikzcd}
	{D^b_{\operatorname{Coh}, S}(\mathcal{X})} & {D^b_{\operatorname{Coh}, T}(\mathcal{Y})} \\
	{D^b_{\operatorname{Coh}}(\mathcal{X})} & {D^b_{\operatorname{Coh}}(\mathcal{Y})}
	\arrow[from=1-1, to=1-2]
	\arrow[hook, from=1-1, to=2-1]
	\arrow[hook, from=1-2, to=2-2]
	\arrow["F"', from=2-1, to=2-2]
\end{tikzcd}\]
    where the vertical arrows are the respective inclusions. 
     Since $D^b_{\operatorname{Coh}}(\mathcal{U}), D^b_{\operatorname{Coh}}(\mathcal{V})$ are the quotients of $D^b_{\operatorname{Coh}}(\mathcal{X})$, $D^b_{\operatorname{Coh}}(\mathcal{Y})$ by these subcategories (see \cite[Appendix B]{Hall_Lamarche_Lank_Peng_2025}), the result follows. 
\end{proof}

The following generalizes \cite{peng2024equivalences}*{Lemma 5.7} and ``Step 2'' of the proof of \cite[Proposition 3.2.3]{olander2022resolutions}.
\begin{lemma}
\label{lemma-cohtocoh}
    Let $\mathcal{X}, \mathcal{Y}$ be Noetherian algebraic stacks. Let $F : D^b_{\operatorname{Coh}}(\mathcal{X}) \to D^b_{\operatorname{Coh}}(\mathcal{Y})$ be an exact equivalence.
    Assume that both $F$ and $F^{-1}$ satisfy condition (ii) above. 
    Let $K \in D^b_{\operatorname{Coh}}(\mathcal{X})$. Then $F(K) \in \operatorname{Coh}\mathcal{Y}$ if and only if $K \in \operatorname{Coh}\mathcal{X}$,  and so $F$ induces an equivalence of categories $\operatorname{Coh}\mathcal{X} \to \operatorname{Coh}\mathcal{Y}$. 
\end{lemma}

\begin{proof}
    By the hypotheses and (\ref{equn-topcohomologysheaf}), $F$ and $F^{-1}$ each take objects whose top non-zero cohomology sheaf lives in degree $n$ to objects whose top non-zero cohomology sheaf lives in degree $n$. In particular, if $0 \neq \ecal{F} \in \operatorname{Coh}\mathcal{X}$, then $H^0(F(\ecal{F})) \neq 0$ and $H^i(F(\ecal{F})) = 0$ for $i > 0$. Then there is a truncation triangle
    $$
    \tau_{<0}F(\ecal{F}) \xrightarrow{\alpha} F(\ecal{F}) \to H^0(F(\ecal{F})) \to \tau_{<0}F(\ecal{F})[1].
    $$
    Applying $F^{-1}$ to $\alpha$ gives an arrow $K \to \ecal{F}$ in $D^b_{\operatorname{Coh}}(\mathcal{X})$ with $\ecal{F} \in \operatorname{Coh}\mathcal{X}$ and $H^i(K) = 0$ for $i \geq 0$. Hence $F^{-1}(\alpha) = 0$ and so $\alpha = 0$. But this implies $\tau_{<0}F(\ecal{F}) = 0$ and so $F(\ecal{F}) \in \operatorname{Coh}\mathcal{Y}$. Applying the same argument to $F^{-1}$ gives the rest of the result. 
\end{proof}

For a Noetherian algebraic stack $\cal Z$, let $\operatorname{Ref}(\cal Z)$ denote the category of reflexive coherent sheaves. Then, we have analogous results for reflexive sheaves. 
\begin{corollary}\label{corollary-reflexivetoreflexive}
     Let $\mathcal{X}, \mathcal{Y}$ be normal integral Noetherian algebraic stacks. Let $F : D^b_{\operatorname{Coh}}(\mathcal{X}) \to D^b_{\operatorname{Coh}}(\mathcal{Y})$ be an exact equivalence.
    Assume that both $F$ and $F^{-1}$ satisfy condition (ii) above. Then $F(K) \in \operatorname{Ref}\mathcal{Y}$ if and only if $K \in \operatorname{Ref}\mathcal{X}$,  and so $F$ induces an equivalence of categories $\operatorname{Ref}\mathcal{X} \to \operatorname{Ref}\mathcal{Y}$. 
\end{corollary}
\begin{proof}
    First, recall that by \cite{antieau2016maximal}*{Proposition 2.12}, a coherent sheaf $\ecal F$ on a normal integral Noetherian algebraic stack $\cal X$ is reflexive if and only if it is torsion-free and $\cal H^1_\cal Z(\cal F) = 0$ for all closed substacks $\cal Z \subset \cal X$ with $\operatorname{codim}_\cal X \cal Z \geq 2$. Take a reflexive coherent sheaf $\ecal F$ on $\cal X$. By Lemma \ref{lemma-cohtocoh}, $F(\ecal F)$ is a coherent sheaf. Suppose $F(\ecal F)$ has a non-zero torsion subsheaf $\ecal G \subset F(\ecal F)$.
    Then, since $F$ induces an exact equivalence on $\operatorname{Coh}$ and preserves codimension of the support, $\ecal F$ has a torsion subsheaf, which is absurd. Hence, $F(\ecal F)$ is torsion-free. The rest mimics the proof of \cite{antieau2016maximal}*{Proposition 2.12}. First, we have a natural short exact sequence
    \[
    0 \to F(\ecal F) \to F(\ecal F)^{\vee \vee} \to \ecal G \to 0 
    \]
    where $\supp \ecal G$ has codimension $\geq 2$. Then by applying the exact equivalence $F\inv$ (which preserves torsion-free sheaves as above) to the short exact sequence, we obtain the long exact sequence
    \[
    0 \to \cal H_{\supp F\inv(\ecal G)}^0(F\inv (\ecal G))  \to \cal H_{\supp F\inv(\ecal G)}^1(\ecal F)  \to  \cal H_{\supp F\inv(\ecal G)}^1(F\inv(F(\ecal F)^{\vee \vee})).
    \]
    Since $\cal H_{\supp F\inv(\ecal G)}^0(F\inv (\ecal G)) \cong F\inv (\ecal G)$ and $\cal H_{\supp F\inv(\ecal G)}^1(\ecal F) = 0$ as $\ecal F$ is reflexive and $\supp F\inv(\ecal G)$ has codimension $\geq 2$, we see $F\inv (\ecal G) = 0$ and therefore $\ecal G = 0$ as desired. Now applying the same argument to $F^{-1}$ completes the proof.
\end{proof}
We can prove the analogous results for vector bundles under the tameness assumption. 
\begin{lemma}
\label{lemma-vecttovect}
 Let $\mathcal{X}, \mathcal{Y}$ be tame Noetherian algebraic stacks of finite type over a field $k$.  Let $F : D^b_{\operatorname{Coh}}(\mathcal{X}) \to D^b_{\operatorname{Coh}}(\mathcal{Y})$ be a $k$-linear, exact equivalence such that both $F$ and $F^{-1}$ satisfy condition (ii) above. Then if $\ecal{F}$ is a finite locally free $\ecal{O}_{\mathcal{X}}$-module, then so is $F(\ecal{F})$. If $\cal X$ contains a closed point with trivial stabilizer, then the rank is also preserved.
\end{lemma}


\begin{proof}
The first claim follows immediately from Lemma \ref{lemma-characterizingvectorbundles} (by applying it to each connected component). For the latter, it suffices by Lemma \ref{lemma-characterizingvectorbundles} again to show that the closed point $y$ corresponding to $x$ under Lemma \ref{lemma-supports} has trivial stabilizer and the same residue field as $x$. 
For this, note that $F$ and $F\inv$ induce a bijection between the simple objects of $\operatorname{Coh}(\cal G_x)$ and $\operatorname{Coh}(\cal G_y)$ and the former is a singleton, which implies $\ecal O_{\cal G_y}$ is $\otimes$-generating and hence $\cal G_y = B\mu_{1,\kappa(y)}  = \spec \kappa(y)$ by Lemma \ref{lemma-cyclicstabilizers}. We deduce that $\kappa(x) = \kappa(y)$ since they are the endomorphism rings of simple objects which correspond under $F$ and $F^{-1}$.
\end{proof}

\subsection{Definition of point-like objects}
\label{subsection-pointlike}

Let $k$ be a field, let $\cal X$ be a regular, proper, connected tame algebraic stack over a field $k$ of dimension $d$, and let $S_\cal X$ denote the Serre functor on $D^b_{\operatorname{Coh}}(\cal X) = \perf(\cal X)$, which exists and is given by $- \otimes^{\bb L} \omega_\cal X [d]$ by Lemma \ref{lemma-isserrefunctor}.  

\begin{definition}
We say an object $\ecal P \in \perf(\cal X)$ is \textit{point-like} if:
    \begin{enumerate}
        \item There exists $N_0 > 0$ such that $S_\cal X^{N_0} \ecal P \cong \ecal P [N_0 d]$.
        \item For every $N \in \bb Z$, every non-zero morphism $\ecal P \to S_\cal X^N  \ecal P  [-Nd] =\ecal P \otimes^{\bb L} \omega_\cal X^{\otimes N}$ is an isomorphism. 
        In particular, $\hom(\ecal P, \ecal P)$ is a division $k$-algebra.  
        \item $\hom_{\cal X}(\ecal P, S_\cal X^N \ecal P[-Nd  + i]) = \hom_{\cal  X}(\ecal P, \ecal P\otimes \omega_\cal X^{\otimes N}[i]) =  0$ for every $N \in \bb Z$ and every $i < 0$.
    \end{enumerate}
    Since the Serre functor commutes with equivalences, the property of an object being a point-like is preserved by equivalences. 
\end{definition}
\begin{remark}\label{remark-point-like}
    We generally cannot take $N_0 = 1$ in condition (i) since $\omega_\cal X$ may restrict to a non-trivial line bundle on some residual gerbe $\mathcal{G}_x$ and we want $\kappa(x, \xi)$ to be point-like. For condition (ii), recall that if $\cal X$ is a smooth projective variety and $\cal P \in \perf \cal X$ has finite support and satisfies conditions (i) (with $N_0 = 1$) and (iii) together with $\End(\cal P)$ being a division $k$-algebra, then $\cal P \cong \kappa(x)$ for some closed point $x \in \cal X$  (e.g. \cite{HuyBook}*{Lemma 4.5}). However, when $\cal X$ is stacky, $\End(\cal P)$ being a division $k$-algebra is not enough to get $\cal P \cong \kappa(x, \xi)$. For example, let $\cal X = [\bb P^1/\mu_2]$, where $\mu_2$
    acts by $(\zeta,  [u:v] )= [\zeta u: v]$, and let $x = [0:1]$. 
    Take the affine open chart $x\in  [\bb A^1_t/\mu_2] \subset \cal X$ given by $v \neq 0$, with the standard action of $\mu_2$ given by $(\zeta, t)\mapsto \zeta t$. Set $\cal P$ to be the pushforward of the structure sheaf of $\mathcal{Z} = [(\spec k[t]/t^2)/\mu_2]$ under the inclusion $\mathcal{Z} \to \mathcal{X}$. Now note that
    \[
    \End_\cal X(\cal P) = H^0(\mathcal{Z}, \ecal{O}_{\mathcal{Z}}) = (k[t]/(t^2))^{\mu_2} = k
    \]
   It is easy to see $\cal P$ also satisfies the conditions (i) and (iii). However, since $(t)/(t^2) \subsetneq k[t]/(t^2)$ is a $\mu_2$-equivariant embedding of $k[t]/(t^2)$-modules, $\cal P$ is not a simple coherent sheaf on $\cal X$ and therefore not of the form $\kappa(x, \xi)$. 
\end{remark}

\begin{lemma}
\label{lemma-pointsarepointlike}
    Let $(x, \xi)$ be a generalized closed point of $\mathcal{X}$ and $n \in \mathbb{Z}$. Then $\kappa(x, \xi)[n]$ is a  point-like object of $\perf (\mathcal{X}).$
\end{lemma}
\begin{proof}
    Condition $\rm{(iii)}$ holds because $\kappa(x,\xi)$ and its twists by tensor-powers of the canonical bundle are coherent sheaves on $\mathcal{X}$. For condition (ii), note that if $i : \mathcal{G}_x \to \mathcal{X}$ is the inclusion of the residual gerbe at $x$, then $i$ is a closed immersion, $i_*\xi = \kappa(x, \xi)$, and $i_*(\xi \otimes i^* \omega_\cal X^{\otimes a} ) = \kappa(x, \xi)\otimes \omega_\cal X^{\otimes a}$. Thus for $a, b \in \mathbf{Z},$
    $$
    \operatorname{Hom}_{\mathcal{X}}(\kappa(x, \xi)\otimes \omega_{\mathcal{X}}^{\otimes a}, \kappa(x, \xi)\otimes \omega_\cal X^{\otimes b}) = \operatorname{Hom}_{\mathcal{G}_x}(\xi\otimes i^* \omega_\cal X^{\otimes a}, \xi \otimes i^* \omega_\cal X^{\otimes b})
    $$
    and any non-zero element in the latter, and hence the former, is invertible since $\xi \otimes i^*\omega_\cal X^{\otimes a}, \xi \otimes i^*\omega_\cal X^{\otimes b}$ are simple objects of the category of coherent sheaves on $\mathcal{G}_x$. Finally, (i) holds since the object $i^*\omega_{\mathcal{X}} \in \operatorname{Pic}(\mathcal{G}_x)$ has finite order, say $N$, by Corollary \ref{corollary-finitelymanysimples}. Hence 
    \[
    S_{\mathcal{X}}^N(\kappa(x,\xi)) = \kappa(x, \xi) \otimes \omega_{\mathcal{X}}^{\otimes N}[Nd] = i_*(\xi \otimes i^*\omega_{\mathcal{X}}^{\otimes N})[Nd] = i_*(\xi)[Nd] = \kappa(x, \xi)[Nd]. \qedhere
    \]
\end{proof}

When $\omega_{\mathcal{X}}$ is nowhere torsion, we have the following partial converse. We will obtain a stronger partial converse under stronger conditions in Proposition \ref{prop-classifypointlikecycliccase}.

\begin{lemma}
\label{lemma-pointlikeandnowheretorsion}
Assume the line bundle $\omega_{\mathcal{X}}$ is nowhere torsion. Then any point-like object in $D^b_{\operatorname{Coh}}(\mathcal{X})$ has support equal to a single closed point of $\mathcal{X}$.
\end{lemma}

\begin{proof}
    Let $K \in D^b_{\operatorname{Coh}}(\mathcal{X})$ be point-like. Condition (i) in the definition of point-like object implies there is $N >0$ so that $\omega_X ^{\otimes N}\otimes H^i(K) \cong H^i(K)$ for every integer $i$. By Corollary \ref{corollary-powersofnowheretorsiononstack}, the line bundle $\omega_{\mathcal{X}}^{\otimes N}$ is nowhere torsion, so $H^i(K)$ has finite support for every $i$, and hence $K$ has finite support. If the support consists of more than one closed point, then the (not necessarily commutative) $k$-algebra $\operatorname{Hom}(K,K)$ contains more than two idempotents, so cannot be a division algebra, contradicting (ii) in the definition of point-like object. This completes the proof. 
\end{proof}

\begin{lemma}\label{lemma-ext of skyscrapers}
    Let $(x, \xi), (y, \eta)$ be generalized closed points of $\mathcal{X}$. Write $\ecal{N}_x, \ecal{N}_y$ for the normal bundles of the residual gerbes $i_x : \mathcal{G}_x \to \mathcal{X}, i_y : \mathcal{G}_y \to \mathcal{X}$. Then we have
    \[
    \ext^i_\cal X(\kappa(x, \chi), \kappa(y, \eta)) = \begin{cases}
        0 & \text{if $x \neq y$} \\
        \hom_{\mathcal{G}_x}(\chi, \eta \otimes \wedge^i\ecal{N}_x ) & \text{if $x = y$. }
    \end{cases}
    \]
    \qedhere 
\end{lemma}
\begin{proof}
First, note that we have
\begin{align*}
    \ext^i_\cal X(\kappa(x, \chi), \kappa(y, \eta)) & = \hom_{\cal X}(  {i_x}_*  \chi,  {i_y}_* \eta) \\
    & = \hom_{\mathcal{G}_y}(\bb L i^*_y {i_x}_* \chi , \eta[i]).
\end{align*}
Thus, if $x \neq y$, then $\ext^i_\cal X(\kappa(x, \chi), \kappa(y, \chi)) = 0$ as $\bb L i^*_y {i_x}_* \chi = 0$. From now on, we assume $x = y$ and write $i$ for $i_x = i_y$. Since the category of coherent sheaves on $\mathcal{G}_x$ is semisimple, we have
$$
\ext^i_\cal X(\kappa(x, \chi), \kappa(x, \eta)) = \hom_{\mathcal{G}_x}(\bb L i^* {i}_* \chi , \eta[i]) = \hom_{\mathcal{G}_x}(H^{-i}(\bb L i^* {i}_* \chi ),  \eta) = \hom_{\mathcal{G}_x}(\wedge^i \ecal{N}_x^\vee \otimes \chi ,  \eta)
$$
by Lemma \ref{lemma-regularclosedimmersioncalc}, as needed.
\end{proof}

\begin{prop}
\label{prop-classifypointlikecycliccase}
 Suppose $x$ is a closed point of $\mathcal{X}$ with residual gerbe $i_x: \cal G_x \to \cal X$. Suppose  $i_x^* \omega_\cal X$ is $\otimes$-generating. Then any  point-like object $K \in D^b_{\operatorname{Coh}}(\mathcal{X})$ with support equal to $\{x\}$ is of the form $\kappa(x, \xi)[n]$ for some generalized closed point $(x,\xi)$ of $\cal X$ and for some integer $n \in \bb Z$.
\end{prop}
\begin{proof}
    Note by Lemma \ref{lemma-tensor-gen-on-gerbe} that $\cal G_x \cong B\mu_N$ for some $N > 0$. The proof is parallel to the classical case (e.g. \cite{HuyBook}*{Lemma 4.5}). 
    Set
    \[
    m_0: = \max\{i \mid  H^i(K) \neq 0\} \quad \text{and} \quad m_1 := \min \{i \mid  H^i(K) \neq 0\}. 
    \]
    Since $H^{m_0}(K)$ and $H^{m_1}(K)$ are supported at $x$, we can take a simple subobject and quotient $\kappa(x, \eta) \inj H^{m_1}(K)$ and $H^{m_0}(K) \surj \kappa(x, \xi)$.  Since $i_x^* \omega_\cal X$ is $\otimes$-generating, we have $\kappa(x, \eta) \otimes \omega^{\otimes n}_\cal X \cong \kappa(x, \xi)$ for some $n$. Hence, there is a non-trivial composition
    \[
    K [m_0] \to H^{m_0}(K) \surj \kappa(x, \xi) \cong \kappa(x, \eta) \otimes \omega_\cal X^{\otimes n} \inj H^{m_1}(K) \otimes \omega_\cal X^{\otimes n}\to K \otimes \omega_\cal X^{\otimes n}[m_1]. 
    \]
    Hence, $m_0 = m_1$ by the condition $\rm{(iii)}$. Hence, we may assume $K = \ecal F$ for a coherent sheaf $\ecal F$ on $\cal X$ (supported at $\{x\}$). Again, we take a simple subobject and quotient 
    \[
    \kappa(x, \eta) \inj \ecal F \surj \kappa(x, \xi).
    \]
    Taking $n$ so that $\kappa(x, \eta) \otimes \omega^{\otimes n} \cong \kappa(x, \xi)$, we get a non-trivial composition 
    \[
    \ecal F \surj \kappa (x, \xi) \cong \kappa (x,\eta) \otimes \omega^{\otimes n} \inj \ecal F \otimes \omega^{\otimes n},
    \]
    which is an isomorphism by the condition (ii). Therefore, $\ecal F \cong \kappa (x, \xi)$. 
\end{proof}
By the example in Remark \ref{remark-point-like}, the condition that $\End(K)$ is a division algebra over $k$, instead of (ii), is not sufficient to prove the above result, as $i_x^*\omega_\cal X$ is indeed $\otimes$-generating in that example. We also need the $\otimes$-generation condition in the result above by the following example.
  \begin{example}\label{example-spherical-twist-on-p112}
    Let $\cal X = \cal P(1,1,2)$ be the weighted projective plane over $k$ and let $x \in \cal X$ be the unique stacky point. Then $x$ has an open neighborhood isomorphic to $[\bb A^2/\mu_2]$ where $\mu_2$ acts with weights $(1,1)$. Note that $\omega_\cal X$ is nowhere torsion by Example \ref{example-weightedprojectivestackcanonical}. Let $\chi_0, \chi_1$ denote the trivial character and the nontrivial character of $G_x = \mu_2$. Then the normal bundle $\ecal{N}_x$ is isomorphic to $\chi_1 \oplus \chi_1$. Using Lemma \ref{lemma-ext of skyscrapers}, we compute that for $\ecal{F} = \kappa(x, \chi_0)$ or $\kappa(x, \chi_1)$, we have
    $$
    \operatorname{Ext}^i_{\cal X}(\ecal{F}, \ecal{F}) = \begin{cases}
        k & \text{if $i = 0,2$}\\
        0 & \text{otherwise}.
    \end{cases}
    $$
    Furthermore, by Lemma \ref{lemma-adjunctionformula} applied to the inclusion of the residual gerbe, the restriction of $\omega_{\mathcal{X}}$ to the residual gerbe at $x$ is $\wedge^2 \ecal{N}_x \cong \chi_0$. In particular, $(\mathcal{X}, x)$ do not satisfy the conditions of Proposition \ref{prop-classifypointlikecycliccase} (but $(\mathcal{X}, y)$ does for any closed point $y \neq x$). It also follows that $\ecal{F} \otimes \omega_{\mathcal{X}} \cong \ecal{F}$ for either $\ecal{F}$ above. We conclude that $\kappa(x, \chi_0), \kappa(x, \chi_1) \in \perf(\cal X)$ are $2$-spherical objects. We can additionally compute $\bb R \hom_\cal X(\kappa(x, \chi_0), \kappa(x, \chi_1)) = k^{\oplus 2}[-1]$. 
    We now observe that the spherical twist $T_{\kappa(x,\chi_0)},$ an autoequivalence of $\perf \cal X,$ cannot satisfy condition (ii) of \ref{subsection-conditions}.
    \begin{enumerate}
        \item We first claim $T_{\kappa(x, \chi_0)} \kappa(x,\chi_1)$ is not of the form $\kappa(x, \chi)[m]$. Indeed, we have
        \begin{align*}
            T_{\kappa(x, \chi_0)} \kappa(x,\chi_1)& = \operatorname{cone} (\bb R\hom_\cal X(\kappa(x, \chi_0), \kappa(x, \chi_1)) \otimes_k \kappa(x, \chi_0) \overset{\sf{ev}}{\to} \kappa(x, \chi_1)) \\
            & = \operatorname{cone}(\kappa(x,\chi_0)^{\oplus 2}[-1]\overset{\sf{ev}}{\to} \kappa(x, \chi_1)).
        \end{align*}
        In particular, we have the corresponding short exact sequence
        \[
        0 \to \kappa(x, \chi_1) \to T_{\kappa(x, \chi_0)} \kappa(x,\chi_1) \to \kappa(x,\chi_0)^{\oplus 2} \to 0. 
        \]
        Therefore, the point-like object $T_{\kappa(x, \chi_0)} \kappa(x,\chi_1)$ is quasi-isomorphic to a coherent sheaf supported at $x$, but not of the form $\kappa(x, \chi)$ (for example, as it has length $3$ in $\operatorname{Coh}_x \cal X$). 
        \item Similarly, $T_{\kappa(x, \chi_0)} \kappa(x, \chi_0) = \kappa(x, \chi_0)[-1]$ (cf. \cite{HuyBook}*{Exercise 8.5 (ii)}). On the other hand, we have that for any generalized closed point $(y,\eta)$ with $x \neq y \in |\cal X|$, 
        \[
        T_{\kappa(x, \chi_0)} \kappa(y, \eta) = \kappa(y, \eta).
        \]
        In particular, $T_{\kappa(x, \chi_0)}$ is not a standard autoequivalence. 
        \item As another example, $T_{\kappa(x,\chi_0)} \ecal O_\cal X$ is not quasi-isomorphic to a coherent sheaf even up to shift. Indeed, by \ref{lemma-regularclosedimmersioncalc} applied to the inclusion of the residual gerbe of $\mathcal{X}$ at $x$, we have $\bb R\hom_\cal X(\kappa(x, \chi_0), \ecal O_\cal X) = k[-2]$, so 
        \begin{align*}
            T_{\kappa(x, \chi_0)} \ecal O_\cal X & = \operatorname{cone} (\bb R\hom_\cal X(\kappa(x, \chi_0), \ecal O_\cal X) \otimes_k k(x, \chi_0) \overset{\sf{ev}}{\to} \ecal O_\cal X) \\
            & = \operatorname{cone}(\kappa(x,\chi_0)[-2]\overset{\sf{ev}}{\to} \ecal O_\cal X).
        \end{align*}
        Therefore, we have $ H^0(T_{\kappa(x, \chi_0)} \ecal O_\cal X) = \ecal O_\cal X$ and $ H^1(T_{\kappa(x, \chi_0)} \ecal O_\cal X) = \kappa(x,\chi_0)$. \qedhere 
    \end{enumerate}
\end{example}

\begin{definition}
    Let $\mathcal{T}$ be a triangulated category and $S \subset \mathcal{T}$ a collection of objects. We say $S$ is \textit{$\operatorname{Ext}^{1}$-connected} if for any objects $P, P' \in S$, there exist 
    objects $P = P_0, P_1, \dots, P_k = \ecal P'$ such that $\ext^1( P_i,  P_{i+1}) \neq 0$ for all $i$. 
\end{definition}

Note that the implicit relation between $P$ and $P'$ in the definition above is not symmetric in general. This asymmetry is useful for the following reason.

\begin{lemma}\label{lemma-uniform-shift-ext1}
    Let $\{P_i \}_{i \in I}$ be a collection of objects of a triangulated category $\mathcal{T}$. Assume that
    $$
    \operatorname{Ext}^n(P_i, P_j) = 0
    $$
    whenever $i \neq j \in I$ and $n \leq 0$. Let $n_i, i \in I$ be integers. If the collection 
    $$
    S = \{P_i[n_i]\}_{i \in I}
    $$
    is $\operatorname{Ext}^{1}$-connected, then all the $n_i$ are equal. 
\end{lemma}
\begin{proof}
    Let $i, j \in I$ be distinct. If
    $$
    \operatorname{Ext}^1(P_i[n_i], P_j[n_j]) = \operatorname{Ext}^{1+n_j-n_i}(P_i, P_j) \neq 0,
    $$
    then we must have $1+n_j-n_i >0$, or $n_i \leq n_j$. Thus by induction, given any collection $Q_0 = P_i[n_i], Q_1, \dots , Q_r = P_j[n_j]$ of elements of $S$ such that $\operatorname{Ext}^1(Q_i, Q_{i+1}) \neq 0$ for all $i$, we also have $n_i \leq n_j$ (to prove this we may assume $Q_i \neq Q_{i+1}$ for all $i$ by possibly removing elements of the sequence). If $S$ is $\operatorname{Ext}^1$-connected, we see that $n_i \leq n_j$ for any pair of objects $P_i[n_i], P_j[n_j]$ of $S$, in particular also for the pair $P_j[n_j], P_i[n_i],$ as needed.
\end{proof}

\begin{lemma}
\label{lemma-genericallyaschemeextconnected}
    Let $x \in \mathcal{X}$ be a closed point. Let $\ecal{N}_x$ be the normal bundle to the residual gerbe $\mathcal{G}_x \to \mathcal{X}$. Assume $\ecal{N}_x$ is a faithful vector bundle on $\mathcal{G}_x$. Let $S \subset \perf (\mathcal{X})$ be the set of generalized closed point sheaves supported at $x$. Then $S$ is $\operatorname{Ext}^1$-connected.
\end{lemma}

\begin{proof}
    By Lemma \ref{lemma-ext of skyscrapers}, for two objects $\kappa(x, \xi), \kappa(x, \eta) \in S$, we have 
    $$
    \operatorname{Ext}^1_{\mathcal{X}}(\kappa(x, \xi), \kappa(x, \eta)) \neq 0 \iff \xi \subset \eta \otimes \ecal{N}_x
    $$
    is a direct summand. So, all $\kappa(x, \xi) \in S$ with $\xi$ a simple object occuring as a direct summand of $\eta \otimes \ecal{N}_x$ are $\operatorname{Ext}$-connected to $\kappa(x, \eta)$, and then all $\kappa(x, \chi) \in S$ with $\chi$ occurring as a simple direct summand of $\eta \otimes \ecal{N}$ are $\operatorname{Ext}$-connected to $\xi$, etc. Thus we see that any $\kappa(x, \xi) \in S$ such that $\xi$ occurs as a simple direct summand of $\eta \otimes \ecal{N}_x^{\otimes i}$ for some $i \geq 0$ is $\operatorname{Ext}$-connected to $\eta$. But every simple object $\xi$ occurs in this way: Let $\chi$ be any simple direct summand of $\xi \otimes \eta^\vee$. Then by Lemma \ref{lemma-tensorproductsoffaithful}, $\chi$ occurs as a direct summand of $\ecal{N}^{\otimes i}$ for some $i \geq 0$. Hence:
    $$
    \operatorname{Hom}_{\mathcal{X}}(\xi, \eta \otimes \ecal{N}^{\otimes i}) = \operatorname{Hom}_{\mathcal{X}}(\xi \otimes \eta^\vee, \ecal{N}^{\otimes i}) \neq 0
    $$
    which shows $\xi$ occurs as a simple direct summand of $\eta \otimes \ecal{N}^{\otimes i}$.
\end{proof}

\section{Bondal--Orlov reconstruction for regular tame stacks}

In this chapter, we prove the main results of the paper, Theorems \ref{thm-bo-tamestack-introversion}, \ref{theorem-autoequivsintroversion}, and \ref{thm-canonicalgeneratesintroversion}. The first two of these are proven in Section \ref{subsection-weakreconstructionforstacks} and the third is proven in Section \ref{subsection-reconstruction-strong-form}. Before this, we review some standard material about the structure of regular tame stacks with trivial generic stabilizer in Section \ref{subsection-structureoftamestacks}. Then in Section \ref{subsection-monoidalfunctors}, we give a criterion for upgrading an equivalence of abelian categories $F: \operatorname{Coh}(\mathcal{X}) \to \operatorname{Coh}(\mathcal{Y})$ for stacks $\mathcal{X}, \mathcal{Y}$ to a symmetric monoidal equivalence, when $F$ commutes with tensoring with a sufficiently negative line bundle. This is one of the key technical inputs for the main results, another being a lemma on Fourier--Mukai transforms in Section \ref{subsection-lemmaonFM}. In the following two sections, we combine the ingredients to prove the main results. Finally, in Section \ref{subsection-stackyexamples}, we provide examples of algebraic stacks to which the main results apply, and in Section \ref{subsection-stackysurfaceexamples}, we analyze the group of auto-equivalences of the derived category of some stacky surfaces in detail.

\label{section-BOforstacks}
\subsection{Structure of regular tame stacks}
\label{subsection-structureoftamestacks}
The results in this section are all well-known, but we state them here for ease of reference. 
The most important for us is the following, which is mostly due to Abramovich, Olsson, and Vistoli, with an added observation due to Bergh.

\begin{theorem}[{\cite[Theorem 3.2]{Abramovich2008}, \cite[Appendix A]{Bergh_2017}}]
\label{thm-localstroftame}
    Let $\pi : \mathcal{X} \to X$ be the coarse space map of a tame algebraic stack locally of finite type over a locally Noetherian scheme. Assume $X$ is the spectrum of a (necessarily Noetherian) strictly Henselian local ring $A$ with residue field $\kappa$. Then
    \begin{enumerate}
        \item There is a tame group scheme $G/A$ acting $A$-linearly on $\operatorname{Spec}B$ for some finite local $A$-algebra $B$ with $A/\mathfrak{m}_A = B/\mathfrak{m}_B$, where $G$ acts  trivially on the closed subscheme $\operatorname{Spec} B/\mathfrak{m}_B$, and such that  $\mathcal{X} \cong [\operatorname{Spec}B/G]$ over $A$. 
        \item If $\mathcal{X}$ is regular, we may take $B$ to be regular. \qedhere
    \end{enumerate}
\end{theorem}

Recall that we call a group scheme $G$ over a base scheme $S$ \emph{tame} if $G/S$ is finite locally free and the $G$-invariants functor from $G$-equivariant quasi-coherent sheaves on $S$ to quasi-coherent sheaves on $S$ is exact. By \cite{Abramovich2008}, this is equivalent to $G/S$ being \'etale locally on $S$ an extension of a finite constant group scheme of order invertible on $S$ by a finite diagonalizable group scheme. Furthermore, such an extension is shown in loc. cit. to split fppf locally on $S$. 

\begin{proof}
%
  %
   Part (i) is essentially \cite[Theorem 3.2]{Abramovich2008}, and it follows immediately from the proof given there. Namely, if $\kappa$ denotes the (separably closed) residue field of $A$, then by \cite[Proposition 3.7]{Abramovich2008}, the morphism $\operatorname{Spec} \kappa \to X$ lifts to a morphism $x : \operatorname{Spec}\kappa \to \mathcal{X}$. Then by \cite[Proposition 2.23]{Abramovich2008}, there is a tame group scheme $G/A$ with $G_\kappa = G_x$ equal to the stabilizer group scheme of the point $x$. Furthermore, it is shown in the first part of the proof of \cite[Proposition 3.6]{Abramovich2008} that there is a $G$-torsor $Y \to \mathcal{X}$ whose pullback along $BG_x \to \mathcal{X}$ is the tautological $G_x$-torsor $\operatorname{Spec} \kappa$. Then $Y \to \operatorname{Spec}A$ is finite and its fiber over $\operatorname{Spec}\kappa$ has exactly one point. Since $A$ is strictly Henselian, it follows that $Y = \operatorname{Spec}B$ is the spectrum of a local ring. Since the fiber of the $G$-torsor $Y \to \mathcal{X}$ over $BG_\kappa$ is the tautological $G_\kappa$-torsor $\operatorname{Spec} \kappa$, it follows that $B$ has residue field $\kappa$ and $G$ acts trivially on $\operatorname{Spec}\kappa$.


   Finally, part (ii) is the part due to Bergh, and the proof is as follows. Choose $G, B$ as in (ii) so that $p: \operatorname{Spec}B/\mathfrak{m}_B \to \operatorname{Spec}B$ is stabilized by $G$. That is, the morphism $p$ is $G$-equivariant, where $G$ acts trivially on $\operatorname{Spec}B/\mathfrak{m}_B$. Thus, there is a Cartesian square
\[\begin{tikzcd}
	{\operatorname{Spec}B/\mathfrak{m}_B} & {\operatorname{Spec}B} \\
	{BG_{B/\mathfrak{m}_B}} & {\mathcal{X}}.
	\arrow[from=1-1, to=1-2]
	\arrow[from=1-1, to=2-1]
	\arrow[from=1-2, to=2-2]
	\arrow[from=2-1, to=2-2]
\end{tikzcd}\]
 Since $BG_{B/\mathfrak{m}_B} \to \mathcal{X}$ is a closed immersion between regular stacks, it is a regular closed immersion. Since $\operatorname{Spec}B \to \mathcal{X} = [\operatorname{Spec}B/G]$ is flat, it follows that $\operatorname{Spec}B/\mathfrak{m}_B \to \operatorname{Spec}B$ is a regular immersion as well, which implies that the local ring $B$ is regular.
\end{proof}

\begin{lemma}
\label{lemma-coarsespaceisnormal}
    Let $\mathcal{X}$ be a tame, regular algebraic stack locally of finite type over a locally Noetherian scheme, with coarse space $X$. Then $X$ is a normal algebraic space.
\end{lemma}

\begin{proof}
    We may assume $X$ is the spectrum of a strictly Henselian local ring $A$ and therefore that $\mathcal{X} = [\operatorname{Spec}B/G]$ where $B$ is a finite local $A$-algebra which is regular and $G$ is a tame group scheme over $A$. Then $X = \operatorname{Spec}A = \operatorname{Spec}(B^G)$ but it is classical that $B^G$ is normal since $B$ is.
\end{proof}

We believe the following result is folklore, but have not found a proof in the literature which works in this generality.

\begin{prop}
\label{prop-faithfulnormalbundle}
    Let $\mathcal{X}$ be a tame algebraic stack locally of finite type over a locally Noetherian scheme. Assume $\mathcal{X}$ has a scheme-theoretically dense open which is a scheme. Let $i : \mathcal{G}_x \to \mathcal{X}$ be the residual gerbe at a closed point. Write $\ecal{N} = i^*\ecal{I}$ where $\ecal{I}$ is the ideal sheaf of the closed immersion $i$. Then $\ecal{N}$ is a faithful vector bundle on $\mathcal{G}$. 
\end{prop}

\begin{lemma}
    Let $A$ be a strictly Henselian local ring with residue field $\kappa$ and $G/A$ a group scheme of the form $\Delta \rtimes C$ where $\Delta$ is a finite diagonalizable group scheme over $A$ and $C$ is a finite constant group scheme. Suppose $H_0 \subset G_\kappa$ is a closed subgroup scheme isomorphic to $\mu_{\ell, \kappa}$ for some prime number $\ell > 0$. Then there is a closed subgroup scheme $H \subset G$ such that $H \cong \mu_{\ell, A}$ and $H_\kappa = H_0$ as closed subgroup schemes of $G_\kappa$. 
\end{lemma}

\begin{proof}
    Case 1: $H_0 \subset \Delta_\kappa$. Let $M$ be the abstract group which is Cartier dual to $\Delta$. Then $H_0 \subset \Delta_{\kappa}$ corresponds to a quotient group $M \twoheadrightarrow \mathbb{Z}/\ell$, which defines a subgroup scheme $H \subset \Delta$ over all of $\operatorname{Spec}A$.

    Case 2: $H_0 \not \subset \Delta_\kappa$. Then the composition $H_0 \to G_\kappa/\Delta_\kappa = C_\kappa$ is not zero, hence monomorphic as $\mu_{\ell, \kappa}$ has no non-trivial closed subgroup schemes. It follows that $H_0$ is an \'etale group scheme  (i.e., $\ell$ is invertible in $\kappa$), hence constant since $\kappa$ is separably closed. Then $H_0 \subset C_\kappa$ corresponds to an inclusion of abstract groups and hence is defined over all of $\operatorname{Spec}A$. Since $C \subset G$ is a closed subgroup scheme, this suffices.
\end{proof}

\begin{proof}[Proof of Proposition \ref{prop-faithfulnormalbundle}]
The question is fppf local on the coarse space $X$ which allows us to replace $\mathcal{X} \to X$ by the base change along $\operatorname{Spec} \ecal{O}_{X,x}^{s.h.} \to X$ and therefore assume: $X = \operatorname{Spec}A$ is strictly Henselian local, and
$\mathcal{X} = [\operatorname{Spec}B/G]$ where $B, G$ are as in part (ii) of \ref{thm-localstroftame}. Replacing $A$ further by a finite flat covering, we may assume $G = \Delta \rtimes C$ where $\Delta$ is a finite diagonalizable group scheme over $A$ and $C$ is a finite constant group scheme of order invertible in $A$. 

Write $\kappa = B/\mathfrak{m}_B = A/\mathfrak{m}_A$. Then the residual gerbe of $\mathcal{X}$ at the closed point $x$ is 
$$
BG_{\kappa} = [(\operatorname{Spec}B/\mathfrak{m}_B)/G] \to [\operatorname{Spec}B/G] = \mathcal{X},
$$
and $\ecal{N}$ corresponds to the $G_\kappa$-equivariant $\kappa$-vector space $\mathfrak{m}_B /\mathfrak{m}_B^2$. Suppose this is not a faithful representation of $G_\kappa$. Then there is a non-trivial closed subgroup scheme $H_0 \subset G_{\kappa}$ which acts trivially on $\mathfrak{m}_B /\mathfrak{m}_B^2$. Possibly replacing $H_0$ with a smaller closed subgroup scheme, by the structure of finite linearly reductive group schemes over separably closed fields, we may assume  $H_0 \cong \mu_{\ell, \kappa}$ where $\ell>0$ is a prime number possibly equal to the residue characteristic. 
By the preceding lemma, there is a closed subgroup scheme $\mu_{\ell, A} \cong H \subset G$ with $H_\kappa = H_0$ as closed subgroup schemes of $G_\kappa$. Then we see that $H$ acts on both $B/\mathfrak{m}_B$ and $\mathfrak{m}_B/\mathfrak{m}_B^2$ trivially. We claim $H$ acts trivially on $B$, i.e., $B = B^H$. First, we claim that $(B/\mathfrak{m}_B^n)^H = B/\mathfrak{m}_B^n$ for all $n >0$. For $n = 1$ this is true. If it holds for some $n >0$, then consider the commutative diagram with exact rows
\[\begin{tikzcd}
	0 & {(\mathfrak{m}_B^n/\mathfrak{m}_B^{n+1})^H} & {(B/\mathfrak{m}_B^{n+1})^H} & {(B/\mathfrak{m}_B^{n})^H} & 0 \\
	0 & {\mathfrak{m}_B^n/\mathfrak{m}_B^{n+1}} & {B/\mathfrak{m}_B^{n+1}} & {B/\mathfrak{m}_B^{n}} & 0,
	\arrow[from=1-1, to=1-2]
	\arrow[from=1-2, to=1-3]
	\arrow[hook, from=1-2, to=2-2]
	\arrow[from=1-3, to=1-4]
	\arrow[hook, from=1-3, to=2-3]
	\arrow[from=1-4, to=1-5]
	\arrow["{=}", hook, from=1-4, to=2-4]
	\arrow[from=2-1, to=2-2]
	\arrow[from=2-2, to=2-3]
	\arrow[from=2-3, to=2-4]
	\arrow[from=2-4, to=2-5]
\end{tikzcd}\]
where we use that $H$ is a tame group scheme to get exactness of the top row. The $B$-module $\mathfrak{m}_B^n/\mathfrak{m}_B^{n+1}$ is a quotient of $(\mathfrak{m}_B/\mathfrak{m}_B^{2})^{\otimes n}$ as an $H$-equivariant $B$-module, and $H$ acts trivially on the latter, hence also on the former. It follows that the left vertical inclusion in the commutative diagram is an equality, hence so is the middle. Since taking $H$-invariants commutes with limits, it follows that $H$ acts trivially on $\widehat{B} = \operatorname{lim}_n B/\mathfrak{m}_B^n$, and since $B \subset \widehat{B}$ as an $H$-equivariant $B$-module (and $H$ is flat and affine over $A$), it follows that $H$ acts trivially on $B$, as desired. But then $\mathcal{X} = [\operatorname{Spec}B/G]$ does not have a scheme-theoretically dense open subscheme as every point contains a $\mu_\ell$ in its stabilizer group. 
\end{proof}

The following is well-known. 

\begin{lemma}
\label{lemma-schemeawayfromcodim2}
    Let $X$ be a separated, normal algebraic space of finite type over a field. Then there is an open subspace $U \subset X$ such that $\operatorname{codim}(X\setminus U \subset X) \geq 2$ and such that $U$ is representable by a scheme.
\end{lemma}

\begin{proof}
    By Chow's Lemma, there is a proper birational surjective morphism $Y \to X$ such that $Y$ is a quasi-projective scheme over $k$. Since $X$ is normal, by the valuative criterion of properness, there is an open subspace $U \subset X$ whose complement has codimension $\geq 2$ such that $Y \to X$ has a section over $U$. In particular, there is an immersion $U \to Y$, and so $U$ is a scheme.
\end{proof}

\begin{lemma}
\label{lemma-structureawayfromcodim2}
    Let $\mathcal{X}\to X$ be the coarse space morphism of a regular tame stack locally of finite type over a locally Noetherian scheme. Assume $\mathcal{X}$ has trivial generic stabilizer. Then there is an open subspace $U \subset X$ which is representable by a regular scheme, and whose complement has codimension $\geq 2$, such that locally on $U$, $\mathcal{X}_U \to U$ is a root stack with respect to a regular divisor. 
\end{lemma}

\begin{proof}
We know that $X$ is normal by Lemma \ref{lemma-coarsespaceisnormal}. Then by Lemma \ref{lemma-schemeawayfromcodim2}, after replacing $X$ with an open whose complement has codimension $\geq 2$, we may assume $X$ is a (normal) scheme. Then for any $\eta \in X$ such that $\operatorname{dim}(\ecal{O}_{X,\eta}) = 1$, the local ring $\ecal{O}_{X,\eta}$ is a DVR. 
The fiber product $\mathcal{X} \times _X \operatorname{Spec}(\ecal{O}_{X,\eta}) \to \operatorname{Spec}(\ecal{O}_{X,\eta})$ is itself a coarse space map. By \cite[Proposition 3.12]{BrescianiVistoli2024}, we see that $\mathcal{X} \times _X \operatorname{Spec}(\ecal{O}_{X,\eta}) \to \operatorname{Spec}(\ecal{O}_{X,\eta})$ is a root stack map $\sqrt[n]{\operatorname{Spec}R} \to \operatorname{Spec}R$ for some $n$. By a standard limit argument, there is an open neighborhood $V$ of $\eta$ such that $\mathcal{X} \times _X V \to V$ is the $n^{th}$ root stack with respect to the closure $\overline{\eta}$ of $\eta$ in $V$. Take for $U$ the union of all such $V$ and the proof is complete.
\end{proof}

\begin{prop}
\label{prop-structureawayfromcodim2}
    Let $\mathcal{X}$ be a tame algebraic stack which is regular and proper over a field. Assume $\mathcal{X}$ contains a dense open which is representable by a scheme. Then there is an open substack $\mathcal{U} \subset \mathcal{X}$ with the following properties:
    \begin{enumerate}
        \item $\operatorname{codim}(\mathcal{X} \setminus \mathcal{U} \subset \mathcal{X}) \geq 2$.
        \item The coarse space of $\mathcal{U}$ is a scheme.
        \item The line bundle $\omega_{\mathcal{X}}|_{\mathcal{U}}$ is point-wise $\otimes$-generating (see Definition \ref{defn-pointwisetensorgenerating}).
    \end{enumerate}
\end{prop}

\begin{proof}
    By Lemma \ref{lemma-structureawayfromcodim2}, there is an open $\mathcal{U} \subset \mathcal{X}$ whose complement has codimension $\geq 2$ such that its coarse space $U$ is a regular scheme, and locally on $U$, $\mathcal{U} \to U$ is a root stack with respect to a regular divisor. If $Y$ is a scheme, $D \subset Y$ is an effective Cartier divisor, and $n \geq 1$ is an integer, then the stack $\mathcal{Y} = \sqrt[n]{(Y, D)}$ has a tautological effective Cartier divisor $\mathcal{D} \subset \mathcal{Y}$ which is a Zariski locally trivial $\mu_n$-gerbe over $D$, and $\mathcal{Y} \setminus \mathcal{D} \cong Y \setminus D$. Thus the residual gerbes of $\mathcal{U}$ are all of the form $B\mu_{n, \kappa}$ for some $n$ (possibly $1$). Also, for any line bundle $\ecal{M}$ on $Y$, we have 
    $$(\mathcal{Y} \to Y)^!(\ecal{M}) = (\mathcal{Y} \to Y)^*\ecal{M} \otimes \ecal{O}_{\mathcal{Y}}((n-1)\mathcal{D}),
    $$ 
    see Lemma \ref{lemma-rootstacks}. For any $y \in \mathcal{D}$, the restriction of this line bundle to the residual gerbe $\mathcal{G}_y \cong B\mu_{n, \kappa(y)}$ has weight $-1$, so its tensor powers generate the category of coherent sheaves on $\mathcal{G}_y$. We can now conclude since $\omega_{\mathcal{X}}|_{\mathcal{U}} = (\mathcal{U} \to U)^!(\omega_X^\bullet|_U) = (\mathcal{U} \to U)^!(\omega_U)$ by compatibility of the functors $(-)^!$ with flat base change, see Section \ref{subsection-dualityforstacks}. 
\end{proof}

The open $\mathcal{U}$ produced above is useful because Zariski locally on $\mathcal{U}$, $\omega_{\mathcal{X}}|_{\mathcal{U}}$ is a tensor generator of the category of coherent sheaves.

\begin{lemma}
\label{lemma-pointwiseonditionfortensorgeneration}
    Let $\mathcal{X}$ be a tame Noetherian algebraic stack 
    and $\pi: \mathcal{X} \to X$ its coarse moduli space. Let $\ecal{L}$ be an invertible $\ecal{O}_{\mathcal{X}}$-module. Assume:
    \begin{enumerate}
        \item $X$ is an affine scheme.
        \item $\ecal{L}$ is point-wise $\otimes$-generating (see Definition \ref{defn-pointwisetensorgenerating}).
    \end{enumerate}
Then every coherent $\ecal{O}_{\mathcal{X}}$-module is a quotient of a direct sum of positive tensor powers of $\ecal{L}$.
\end{lemma}


\begin{proof}
Let $\ecal{F}$ be a coherent $\ecal{O}_{\mathcal{X}}$-module and $x \in \mathcal{X}$ closed. Write $i_x : \mathcal{G}_x \to \mathcal{X}$ for the residual gerbe at $x$. Then by assumption, there is a surjection $\bigoplus_j i_x^*\ecal{L}^{\otimes n_j} \twoheadrightarrow i_x^*\ecal{F}$ where the sum is finite and each $n_j$ is positive. Then, since $i_x$ is a closed immersion, the composition
$$
\bigoplus_j \ecal{L}^{\otimes n_j} \to i_{x,*}i_x^*(\bigoplus_j \ecal{L}^{\otimes n_j}) \to i_{x,*}i_x^*\ecal{F}
$$
is surjective also. Since $\mathcal{X}$ is tame with affine coarse space, any finite locally free $\ecal{O}_{\mathcal{X}}$-module is a projective object of the category of quasi-coherent sheaves. Therefore the surjection lifts to a morphism
$$
\bigoplus_j \ecal{L}^{\otimes n_j} \to \ecal{F}
$$
and we see by Nakayama's Lemma (see Section \ref{subsection-generalizedpoints}) that this morphism is surjective in a neighborhood of $x$. Doing this for every closed point $x$ and then applying quasi-compactness proves the result. 
\end{proof}

Here are two standard results about extension away from codimension $2$.

\begin{lemma}\label{lemma-hartogs}
    Let $\mathcal{X}$ be an algebraic stack which is regular. Let $\mathcal{U} \subset \mathcal{X}$ be an open whose complement has codimension $\geq 2$. Then $\operatorname{Pic}(\mathcal{X}) \to \operatorname{Pic}(\mathcal{U})$ is an isomorphism.
\end{lemma}

\begin{proof}
    As in the case of schemes, the inverse is given by $\ecal{L} \mapsto j_*(\ecal{L})$ where $j : \mathcal{U} \to \mathcal{X}$ is the inclusion. The checks that $j_*(\ecal{L})$ is a line bundle and that the two maps are inverses are local and therefore follow from the case of schemes.
\end{proof}

The next result is  \cite[Theorem 4.6]{FantechiMannNironi+2010+201+244} in the Deligne-Mumford case. The exact same proof works in the following setting, except that Grothendieck's purity of the branch locus is replaced by the purity theorem \cite[Theorem 3.1]{marrama2016purity}.

\begin{prop}
\label{prop-extendingmorphisms}
    Let $\mathcal{X}, \mathcal{Y}$ be algebraic stacks. Assume $\mathcal{X}$ is regular and $\mathcal{Y}$ is tame and locally of finite type over  locally Noetherian scheme, and $\mathcal{Y} \to Y$ is its coarse space morphism. Suppose $\mathcal{U} \subset \mathcal{X}$ is an open whose complement has codimension $\geq 2$. Then for any solid $2$-commutative square
\[\begin{tikzcd}
	{\mathcal{U}} & {\mathcal{Y}} \\
	{\mathcal{X}} & Y,
	\arrow[from=1-1, to=1-2]
	\arrow[from=1-1, to=2-1]
	\arrow[from=1-2, to=2-2]
	\arrow[dashed, from=2-1, to=1-2]
	\arrow[from=2-1, to=2-2]
\end{tikzcd}\]
there is a unique (up to unique isomorphism) dashed arrow making the diagram $2$-commute
\end{prop} 

The uniqueness will follow immediately from the following lemma, which is almost identical to {\cite[Proposition A.1]{FantechiMannNironi+2010+201+244}}.

\begin{lemma}
\label{lemma-separationforstacks}
    Let $\mathcal{Y}$ be an algebraic stack with finite diagonal. 
    Let $\mathcal{X}$ be a normal algebraic stack and $\mathcal{U} \subset \mathcal{X}$ a dense open substack. Given $1$-morphisms $f, g : \mathcal{X} \to \mathcal{Y}$ and a $2$-isomorphism $\alpha : f|_{\mathcal{U}} \to g|_{\mathcal{U}}$, there is a unique isomorphism $\beta : f \to g$ whose restriction to $\mathcal{U}$ is $\alpha$.
\end{lemma}

\begin{proof}
Consider the $2$-commutative diagram of algebraic stacks in which the square is $2$-Cartesian.
\[\begin{tikzcd}
	& {\mathcal{I}} & {\mathcal{X}} \\
	{\mathcal{U}} & {\mathcal{X}} & {\mathcal{X} \times \mathcal{X}}
	\arrow[from=1-2, to=1-3]
	\arrow[from=1-2, to=2-2]
	\arrow["{\Delta_\mathcal{X}}", from=1-3, to=2-3]
	\arrow["\alpha", from=2-1, to=1-2]
	\arrow[from=2-1, to=2-2]
	\arrow["{(f,g)}",from=2-2, to=2-3]
\end{tikzcd}\]
Giving a morphism from an algebraic stack $\mathcal{T}$ into $\mathcal{I}$ is equivalent to giving a morphism $h: \mathcal{T} \to \mathcal{X}$ together with a $2$-isomorphism $f \circ h \cong g \circ h$, hence $\alpha$ induces an arrow as in the diagram. We must show there is a section of $\mathcal{I} \to \mathcal{X}$ which agrees with $\alpha$ over $\mathcal{U}$. As usual, consider the reduced closed substack $\mathcal{Y} \subset \mathcal{I}$ corresponding to the closure of the image of $|\mathcal{U}| \to |\mathcal{I}|.$  Equivalently, $\mathcal{Y}$ is the scheme theoretic image of $\mathcal{U} \to \mathcal{I}$. Then $\mathcal{Y} \to \mathcal{X}$ is finite and birational over the normal algebraic stack $\mathcal{X}$, so $\mathcal{Y} \to \mathcal{X}$ is an isomorphism. Its inverse gives the desired section. 
\end{proof}

\begin{proof}[Proof of Proposition \ref{prop-extendingmorphisms}]
    Uniqueness is Lemma \ref{lemma-separationforstacks} above. 
    
    Existence: By uniqueness, existence is smooth local on the 
    source $\mathcal{X}$ and on the base $Y$, so we may assume $\mathcal{X} = X$ is a regular 
    scheme (and we write $U$ for $\mathcal{U}$). Working \'etale 
    locally on $Y$, we may assume $Y$ is affine, and by 
    \cite[Theorem 3.2]{Abramovich2008}, we may assume 
    $\mathcal{Y} = [V/G]$ where $G/Y$ is a finite locally free 
    tame group scheme acting $Y$-linearly on a finite $Y$-scheme 
    $V$. Then by definition of the quotient $[V/G]$, we are given 
    a $G$-torsor $P \to U$ and a $G$-equivariant morphism $P \to 
    V$ over $Y$. There is a $G$-torsor $\overline{P} \to X$ 
    extending $P \to V$ by \cite[Theorem 3.1]{marrama2016purity} 
    or \cite[Theorem 7.1.3]{CesnaviciusScholze2024}. Since 
    $\overline{P} \to X$ is flat, it follows that $\overline{P}$ 
    is Cohen--Macaulay, and in particular satisfies Serre's 
    condition $S_2$. Then $P \subset \overline{P}$ is the complement
    of a closed set of codimension $\geq 2$ and $V$ is affine, so 
    the morphism $P \to V$ extends uniquely to a morphism 
    $\overline{P} \to V$ by Hartogs' Lemma. The morphism 
    $\overline{P} \to V$ is $G$-equivariant since $P \subset 
    \overline{P}$ is scheme-theoretically dense and $P \to V$ is 
    $G$-equivariant. This completes the proof.
\end{proof}

The following weaker version of the resolution property will be useful later.

\begin{lemma}
\label{lemma-resolutionbyreflexives}
    Let $\cal X$ be a regular, integral, finite type algebraic stack over a field $k$ with quasi-finite, separated diagonal. 
    Then, any coherent sheaf on $\cal X$ is a quotient of a reflexive coherent sheaf. 
\end{lemma}

\begin{proof}
First of all, the claim follows for a normal Noetherian scheme by the first paragraph of the proof of \cite{mathur2023resolution}*{Proposition 6.5}. There exists a finite, surjective morphism $f: Y \to \cal X$ from a scheme $Y$ by \cite{rydh_approx}*{Theorem B}. By taking the normalization of a dominating irreducible component, we may assume $Y$ is a normal and integral scheme of finite type, 
and hence any coherent sheaf on $Y$ is a quotient of a reflexive coherent sheaf. 
Recall that the functor $f_*:\operatorname{Coh}(Y) \to \operatorname{Coh}(\cal X)$ has a right adjoint $f^!: \operatorname{Coh}(\cal X) \to \operatorname{Coh}(Y)$ given by 
\[f_*f^!(-) = \mathcal{H}om_{\ecal O_\cal X}(f_*\ecal O_Y, - )\]
as an $f_*\ecal{O}_Y$-module. 
We claim that for a coherent sheaf $\ecal{F}$ on $\mathcal{X}$, the co-unit $f_*f^!(\ecal{F}) \to \ecal{F}$ is surjective.
This morphism is the result of applying the functor $\mathcal{H}om_{\ecal{O}_{\mathcal{X}}}(-, \ecal{F})$ to $\ecal{O}_{\mathcal{X}} \to f_*\ecal{O}_Y$, so it suffices to show that $\ecal{O}_{\mathcal{X}} \to f_*\ecal{O}_Y$ is an fppf locally split injection. Thus replacing $\mathcal{X}$ with a smooth cover, we may assume $f: \spec A \to \spec B$ with $B$ a regular $k$-algebra, and then we can apply the (equal characteristic case of the) direct summand theorem \cite{hochster1973contracted}.  

Now, take a coherent sheaf $\ecal F$ on $\cal X$. Since $f^!\ecal F$ is a coherent sheaf on $Y$, there is a reflexive coherent sheaf $\ecal G$ and a surjective map $\ecal G \surj f^! \ecal F$. Therefore, since $f_*$ is exact, we have the composition 
\[
f_* \ecal G \surj f_*f^! \ecal F \surj \ecal F
\]
of surjective maps. Hence, it suffices to show the coherent sheaf $f_*\ecal G$ is reflexive, but this can be checked smooth locally, so we may assume $\cal X$ is a scheme, in which case 
it is well-known, see 
\cite{schwede2010behavior}*{\S2.3 (1)}. 
\end{proof}

\subsection{Monoidal functors on abelian categories}
\label{subsection-monoidalfunctors}

Let $(\mathcal{A}, \otimes)$ be a monoidal category. For $A \in \mathcal{A}$, write $\tau_A$ for the functor $A' \mapsto A' \otimes A$. There are maps $f: \operatorname{Hom}(\tau_{A_1}, \tau_{A_2}) \to \operatorname{Hom}_{\mathcal{A}}(A_1, A_2)$ and $g : \operatorname{Hom}_{\mathcal{A}}(A_1, A_2) \to \operatorname{Mor}(\tau_{A_1}, \tau_{A_2})$ where $f$ is defined by sending $\alpha: \tau_{A_1} \to \tau_{A_2}$ to $\alpha_{1_{\mathcal{A}}} : A_1 \to A_2$ and $g$ is defined by sending $\varphi: A_1 \to A_2$ to the natural transformation which on an object $A$ is given by $A \otimes A_1 \xrightarrow{1 \otimes \varphi}A \otimes A_2$. It is clear that the composition $f \circ g$ is the identity. Still, $f$ and $g$ need not be inverse bijections.

\begin{example}
    Let $k = \bar{k}$ be an algebraically closed field and $G \neq 1$ a finite 
    group of order prime to the characteristic of $k$. Consider the category $\operatorname{Coh}B_kG$ of finite-dimensional $k$-representations of $G$. There is a natural transformation $\alpha: \operatorname{id} \to \operatorname{id}$ which is not induced by any morphism $k \to k$ from the trivial representation to itself (note that $\operatorname{id} \cong \tau_k$): To define such a natural transformation $\alpha$, it suffices to define $\alpha_V :V \to V$ for every irreducible representation $V$ of $G$. Each map $\alpha_V$ must be multiplication by some scalar $a_V \in k$ by Schur's Lemma. Simply choose these scalars so they are not all the same.
\end{example}

\begin{lemma}
\label{lemma-hombetweentensorfunctors}
    Suppose $\mathcal{X}$ is a Noetherian algebraic stack which has a scheme theoretically dense open which is a scheme, and which has the resolution property. Let $\ecal{F}, \ecal{G}$ be finite locally free $\ecal{O}_{\mathcal{X}}$-modules, and let $\tau_{\ecal{F}} , \tau_{\ecal{G}} : \operatorname{Coh}\mathcal{X} \to \operatorname{Coh}\mathcal{X}$ be as above. Then the natural map
    $$
    f: \operatorname{Hom}(\tau_{\ecal{F}}, \tau_{\ecal{G}}) \to \operatorname{Hom}_{\ecal{O}_{\mathcal{X}}}(\ecal{F}, \ecal{G})
    $$
    is a bijection.
\end{lemma}

\begin{remark}
    One can check that in the setting of the lemma, if $\Delta : \mathcal{X} \to \mathcal{X} \times _k \mathcal{X}$ is the diagonal, then 
    $$\operatorname{Hom}_{\mathcal{X} \times _k \mathcal{X}}(\Delta_*\ecal{F}, \Delta_*\ecal{G}) = \operatorname{Hom}_{\mathcal{X}}(\ecal{F}, \ecal{G}).$$
    This is related because $\Delta_*\ecal{F}$ and $\Delta_*\ecal{G}$ are kernels giving rise to the functors $\tau_{\ecal{F}}, \tau_{\ecal{G}}$. Note however that $\Delta_*$ is not in general fully faithful under the assumptions since $\mathcal{X}$ is a stack.
\end{remark}

\begin{proof}
    Let $\alpha : \tau_{\ecal{F}}  \to \tau_{\ecal{G}}$ be a natural transformation and let $\beta  = g(f(\alpha))$. By the discussion above, we have to show $\beta   = \alpha$. 

    Claim: For any finite locally free $\ecal{O}_{\mathcal{X}}$-module $\ecal{E}$, we have $\alpha_{\ecal{E}} = \beta_{\ecal{E}} : \ecal{E} \otimes \ecal{F} \to \ecal{E} \otimes \ecal{G}$.

    By assumption, there is a scheme-theoretically dense open immersion $j: U \to \mathcal{X}$ with $U$ a scheme, and possibly shrinking $U$, we may assume $\ecal{E}|_{U} \cong \ecal{O}_U^{\oplus r}$. Choose a coherent $\ecal{O}_{\mathcal{X}}$-module $\ecal{H}$ and morphisms $\ecal{H} \to \ecal{E}$ and $\ecal{H} \to \ecal{O}_{\mathcal{X}}^{\oplus r}$ which become isomorphisms when restricted to $U$. Then we have two commutative diagrams
\[\begin{tikzcd}
	{\ecal{E} \otimes \ecal{G}} & {\ecal{H} \otimes \ecal{G}} & {\ecal{O}_{\mathcal{X}}^{\oplus r} \otimes \ecal{G}} \\
	{\ecal{E} \otimes \ecal{F}} & {\ecal{H}\otimes \ecal{F}} & {\ecal{O}_{\mathcal{X}}^{\oplus r} \otimes \ecal{F}}
	\arrow[from=1-2, to=1-1]
	\arrow[from=1-2, to=1-3]
	\arrow["{\alpha_{\ecal{E}}}"', from=2-1, to=1-1]
	\arrow["{\beta_{\ecal{E}}}", shift left=3, from=2-1, to=1-1]
	\arrow["{\alpha_{\ecal{H}}}"', from=2-2, to=1-2]
	\arrow["{\beta_{\ecal{H}}}", shift left=3, from=2-2, to=1-2]
	\arrow[from=2-2, to=2-1]
	\arrow[from=2-2, to=2-3]
	\arrow["{\alpha_{\ecal{O}_{\mathcal{X}}^{\oplus r}} = \beta_{\ecal{O}_{\mathcal{X}}^{\oplus r}}}"', from=2-3, to=1-3]
\end{tikzcd}\]
in which all horizontal arrows become isomorphisms when restricted to $U$. Thus $\alpha_{\ecal{E}}|_{U}$ agrees with $\beta_{\ecal{E}}|_{U}$. But since $\alpha_{\ecal{E}}$ and $\beta_{\ecal{E}}$ are morphisms of vector bundles which agree when restricted to the scheme-theoretically dense open $U$, we see that $\alpha_{\ecal{E}} = \beta_{\ecal{E}}$, as claimed. 

Now if $\ecal{H}\in \operatorname{Coh}\mathcal{X}$ is arbitrary, then we can choose a surjection $\ecal{E} \twoheadrightarrow \ecal{H}$ with $\ecal{E}$ finite locally free and then we have two commutative squares,
\[\begin{tikzcd}
	{\ecal{E} \otimes \ecal{F}} & {\ecal{H} \otimes \ecal{F}} \\
	{\ecal{E} \otimes \ecal{G}} & {\ecal{H} \otimes \ecal{G}},
	\arrow[two heads, from=1-1, to=1-2]
	\arrow["{\alpha_{\ecal{E}}= \beta_{\ecal{E}}}", from=1-1, to=2-1]
	\arrow["{\alpha_{\ecal{H}}}"', from=1-2, to=2-2]
	\arrow["{\beta_{\ecal{H}}}", shift left=3, from=1-2, to=2-2]
	\arrow[two heads, from=2-1, to=2-2]
\end{tikzcd}\]
which implies $\alpha_{\ecal{E}} = \beta_{\ecal{E}}$.
\end{proof}

\begin{prop}
\label{prop-itsmonoidal}
Let $\mathcal{X}, \mathcal{Y}$ be Noetherian algebraic stacks which have scheme-theoretically dense opens which are schemes. Suppose $\mathcal{X}$ has an invertible $\ecal{O}_{\mathcal{X}}$-module $\ecal{L}$ with the property that every coherent $\ecal{O}_{\mathcal{X}}$-module $\ecal{F}$ is a quotient of a finite direct sum $\bigoplus _i \ecal{L}^{\otimes n_i}$ for some integers $n_i>0$. 
Suppose 
   $F: \operatorname{Coh}\mathcal{X} \to \operatorname{Coh}\mathcal{Y}$ is an equivalence of categories such that there are an isomorphism $F(\ecal{O}_{\mathcal{X}}) \cong \ecal{O}_{\mathcal{Y}}$, and an isomorphism of functors $\beta : F \circ \tau_{\ecal{L}} \cong \tau_{F(\ecal{L})} \circ F$. Then there is a unique structure $\alpha = \alpha_{\ecal{F}, \ecal{G}}$ of 
 monoidal functor on $F$ such that for all $\ecal{F} \in \operatorname{Coh}\mathcal{X}$, 
   $$
  \alpha_{\ecal{F}, \ecal{L}} =  \beta_{\ecal{F}} : F(\ecal{F} \otimes \ecal{L}) \to F(\ecal{F}) \otimes F(\ecal{L}).
   $$
   Moreover, $\alpha$ is a symmetric monoidal structure on $F$. 
\end{prop}

\begin{remark}
\label{remark-lgoestom}
    Suppose we assume instead that there exists some line bundle $\ecal{M}$ on $\mathcal{Y}$ such that $F \circ \tau_{\ecal{L}} \cong \tau_{\ecal{M}} \circ F$ and $F(\ecal{O}_{\mathcal{X}}) \cong \ecal{O}_{\mathcal{Y}}$. It follows that $\ecal{M} \cong F(\ecal{L})$ since $F(\ecal{L}) = F\circ \tau_{\ecal{L}}(\ecal{O}_{\mathcal{X}}) \cong \tau_{\ecal{M}} \circ F(\ecal{O}_{\mathcal{X}}) \cong \ecal{M}$. 
\end{remark}

Much of the following argument is formal in the sense that it could be generalized to symmetric monoidal abelian categories with right exact tensor products. The subtlety seems to be checking that the morphisms $\alpha_{\ecal{F}, \ecal{L}^{\otimes n}}$ defined below are natural in the second variable, which uses Lemma \ref{lemma-hombetweentensorfunctors} above.

\begin{proof}

We leave uniqueness to the reader. For existence, 
define inductively for $n \geq 1$ and all coherent sheaves $\ecal{F}$ an isomorphism $\alpha_{\ecal{F}, \ecal{L}^{\otimes n}} : F(\ecal{F} \otimes \ecal{L}^{\otimes n}) \to F(\ecal{F}) \otimes F(\ecal{L}^{\otimes n})$ by $\alpha_{\ecal{F}, \ecal{L}} = \beta_{\ecal{F}}$, and for $n \geq 1$, $\alpha_{\ecal{F}, \ecal{L}^{\otimes {n+1}}}$ is the unique isomorphism making the diagram
\[\begin{tikzcd}
	{F(\ecal{F} \otimes \ecal{L}^{\otimes n+1})} && {F(\ecal{F}) \otimes F(\ecal{L}^{\otimes n+1})} \\
	{F(\ecal{F} \otimes \ecal{L}^{\otimes n})\otimes F(\ecal{L})} && {F(\ecal{F} )\otimes F(\ecal{L}^{\otimes n})\otimes F(\ecal{L})}
	\arrow["{\alpha_{\ecal{F}, \ecal{L}^{\otimes n+1}}}", from=1-1, to=1-3]
	\arrow["{\alpha_{\ecal{F} \otimes \ecal{L}^{\otimes n}, \ecal{L}}}"', from=1-1, to=2-1]
	\arrow["{1 \otimes \alpha_{\ecal{L}^{\otimes n}, \ecal{L}}}", from=1-3, to=2-3]
	\arrow["{\alpha_{\ecal{F}, \ecal{L}^{\otimes n}}\otimes 1}"', from=2-1, to=2-3]
\end{tikzcd}\]
commute. One proves by induction that for each fixed $n$, $\alpha_{\ecal{F}, \ecal{L}^{\otimes n}}$ is natural in $\ecal{F}$. One then proves by induction on $j>0$ that for all $i>0$, the diagram
\[\begin{tikzcd}
	{F(\ecal{F} \otimes \ecal{L}^{\otimes i + j})} && {F(\ecal{F}) \otimes F(\ecal{L}^{\otimes i+j})} \\
	{F(\ecal{F} \otimes \ecal{L}^{\otimes i})\otimes F(\ecal{L}^{\otimes j})} && {F(\ecal{F}) \otimes F(\ecal{L}^{\otimes i}) \otimes F(\ecal{L}^{\otimes j})}
	\arrow["{\alpha_{\ecal{F}, \ecal{L}^{\otimes i+j}}}", from=1-1, to=1-3]
	\arrow["{\alpha_{\ecal{F} \otimes \ecal{L}^{\otimes i}, \ecal{L}^{\otimes j}}}"', from=1-1, to=2-1]
	\arrow["{1 \otimes \alpha_{\ecal{L}^{\otimes i}, \ecal{L}^{\otimes j}}}", from=1-3, to=2-3]
	\arrow["{\alpha_{\ecal{F}, \ecal{L}^{\otimes i}} \otimes 1}"', from=2-1, to=2-3]
\end{tikzcd}\]
commutes. Namely, for $j = 1$ this is by definition, and for $j\geq 1$, in the diagram
\[\begin{tikzcd}
	{F(\ecal{F} \otimes \ecal{L}^{\otimes i+j+1}) } & {\ecal{F}(F) \otimes F(\ecal{L}^{\otimes i + j + 1})} \\
	{F(\ecal{F} \otimes \ecal{L}^{\otimes i + j})\otimes F(\ecal{L})} & {F(\ecal{F}) \otimes F(\ecal{L}^{\otimes i + j})\otimes F(\ecal{L})} \\
	{F(\ecal{F} \otimes \ecal{L}^{\otimes i}) \otimes F(\ecal{L}^{\otimes j}) \otimes F(\ecal{L})} & {F(\ecal{F}) \otimes F(\ecal{L}^{\otimes i}) \otimes F(\ecal{L}^{\otimes j}) \otimes F(\ecal{L})} \\
	{F(F\otimes \ecal{L}^{\otimes i})\otimes F(\ecal{L}^{\otimes j+1})} & {F(\ecal{F}) \otimes F(\ecal{L}^{\otimes i}) \otimes F(\ecal{L}^{\otimes j+1}),}
	\arrow[from=1-1, to=1-2]
	\arrow[from=1-1, to=2-1]
	\arrow[from=1-2, to=2-2]
	\arrow[from=2-1, to=2-2]
	\arrow[from=2-1, to=3-1]
	\arrow[from=2-2, to=3-2]
	\arrow[from=3-1, to=3-2]
	\arrow[from=3-1, to=4-1]
	\arrow[from=3-2, to=4-2]
	\arrow[from=4-1, to=4-2]
\end{tikzcd}\]
the top square commutes by definition of the top horizontal arrow, the middle square commutes by induction, and the bottom square (where the vertical arrows are each the inverses of $1 \otimes \alpha_{\ecal{L}^{\otimes j}, \ecal{L}}$) commutes because both compositions are $\alpha_{\ecal{F}, \ecal{L}^{\otimes i}} \otimes \alpha_{\ecal{L}^{\otimes j}, \ecal{L}}^{-1}$.

Finally, $\alpha_{\ecal{F}, \ecal{L}^{\otimes n}}$ is natural in the second variable $\ecal{L}^{\otimes n}$: Suppose given a morphism $\varphi: \ecal{L}^{\otimes m} \to \ecal{L}^{\otimes n}$. We must show for all $\ecal{F} \in \operatorname{Coh}\mathcal{X}$ the diagram 
\[\begin{tikzcd}
	{F(\ecal{F} \otimes \ecal{L}^{\otimes m})} & {F(\ecal{F})\otimes F(\ecal{L}^{\otimes m})} \\
	{F(\ecal{F}\otimes \ecal{L}^{\otimes n})} & {F(\ecal{F}) \otimes F(\ecal{L}^{\otimes n})}
	\arrow["{\alpha_{\ecal{F}, \ecal{L}^{\otimes m}}}", from=1-1, to=1-2]
	\arrow["{F(1 \otimes \varphi)}", from=1-1, to=2-1]
	\arrow["{1 \otimes F(\varphi)}", from=1-2, to=2-2]
	\arrow["{\alpha_{\ecal{F}, \ecal{L}^{\otimes n}}}"', from=2-1, to=2-2]
\end{tikzcd}\]
commutes. We do this by viewing the composition $\alpha_{\ecal{F}, \ecal{L}^{\otimes n}} \circ F(1 \otimes \varphi) \circ \alpha_{\ecal{F}, \ecal{L}^{\otimes m}}^{-1}$ as a natural transformation from the functor $\tau_{F(\ecal{L}^{\otimes m})} \circ F \to \tau_{F(\ecal{L}^{\otimes n})} \circ F$, which since $F$ is equivalence, is equivalent to the data of a natural transformation $\tau_{F(\ecal{L}^{\otimes m})} \to \tau_{F(\ecal{L}^{\otimes n})}$. By Lemma \ref{lemma-hombetweentensorfunctors}, we see that this morphism of functors is the one determined by the morphism $F(\varphi) : F(\ecal{L}^{\otimes m}) \to F(\ecal{L}^{\otimes n})$, which says that the diagram commutes for all $\ecal{F}$. Note here that $F(\ecal{L}^{\otimes i})$ is an invertible $\ecal{O}_{\mathcal{Y}}$-module for all $i$ since $\ecal{O}_{\mathcal{Y}} \cong F(\ecal{O}_{\mathcal{X}}) \cong F(\ecal{L}^{\otimes -i} ) \otimes F(\ecal{L}^{\otimes i})$, so that lemma indeed applies. 

The rest of the argument is formal. There is a unique way to compatibly define $\alpha_{\ecal{F}, \ecal{E}}$ whenever $\ecal{E}$ is a finite direct sum of positive tensor powers of $\ecal{L}$ and this rule is natural in $\ecal{F}$ and $\ecal{E}$ and satisfies the associativity axiom for computing $F(\ecal{F} \otimes \ecal{E}_1 \otimes \ecal{E}_2)$. Then an arbitrary $\ecal{G} \in \operatorname{Coh} \mathcal{X}$ admits a resolution
$
\ecal{E}_2 \to \ecal{E}_1 \to \ecal{G} \to 0
$
and we define $\alpha_{\ecal{F}, \ecal{G}}$ for $\ecal{F} \in \operatorname{Coh}\mathcal{X}$ to be the unique dashed arrow making the diagram
\[\begin{tikzcd}
	{F(\ecal{F}\otimes \ecal{E}_2)} & {F(\ecal{F}\otimes \ecal{E}_1)} & {F(\ecal{F}\otimes \ecal{G})} & 0 \\
	{F(\ecal{F})\otimes F(\ecal{E}_2)} & {F(\ecal{F})\otimes F(\ecal{E}_1)} & {F(\ecal{F})\otimes F(\ecal{G})} & 0
	\arrow[from=1-1, to=1-2]
	\arrow["{\alpha_{\ecal{F}, \ecal{E}_2}}", from=1-1, to=2-1]
	\arrow[from=1-2, to=1-3]
	\arrow["{\alpha_{\ecal{F}, \ecal{E}_1}}", from=1-2, to=2-2]
	\arrow[from=1-3, to=1-4]
	\arrow["{\alpha_{\ecal{F}, \ecal{G}}}", dashed, from=1-3, to=2-3]
	\arrow[from=2-1, to=2-2]
	\arrow[from=2-2, to=2-3]
	\arrow[ from=2-3, to=2-4]
\end{tikzcd}\]
commute. Note that both rows are exact since $F$ is an equivalence. By standard arguments comparing resolutions, the definition of $\alpha_{\ecal{F}, \ecal{G}}$ does not depend on the choice of resolution and $\alpha_{\ecal{F}, \ecal{G}}$ is natural in the second variable. It is also strictly easier to see that $\alpha_{\ecal{F}, \ecal{G}}$ is natural in the first variable. For associativity, suppose $\ecal{G}, \ecal{G}' \in \operatorname{Coh}\mathcal{X}$. Choose surjections $\ecal{E} \twoheadrightarrow \ecal{G}$ and $\ecal{E}' \twoheadrightarrow \ecal{G}'$ where $\ecal{E}, \ecal{E}'$ are isomorphic to finite direct sums of positive tensor powers of $\ecal{L}$. We must check that the square 
\[\begin{tikzcd}
	{F(\ecal{F} \otimes \ecal{G} \otimes \ecal{G}')} & {F(\ecal{F}) \otimes F(\ecal{G} \otimes \ecal{G}')} \\
	{F(\ecal{F} \otimes \ecal{G}) \otimes F(\ecal{G}')} & {F(\ecal{F}) \otimes F(\ecal{G}) \otimes F(\ecal{G}')}
	\arrow[from=1-1, to=1-2]
	\arrow[from=1-1, to=2-1]
	\arrow[from=1-2, to=2-2]
	\arrow[from=2-1, to=2-2]
\end{tikzcd}\]
commutes. Form the cube where the back face is the diagram above and the front face is the same square with $\ecal{G}, \ecal{G}'$ replaced by $\ecal{E}, \ecal{E}'$, which we know commutes. We also know all the other squares commute by definition of $\alpha$. It follows that the original square commutes because $F(\ecal{F} \otimes \ecal{E} \otimes \ecal{E}') \twoheadrightarrow F(\ecal{F} \otimes \ecal{G} \otimes \ecal{G}')$, as needed. 

Now, considering the unit compatibility axioms for monoidal functors, there is a unique unit isomorphism $\eta: \ecal O_\cal Y \to F(\ecal O_X)$ so that $(F, \eta, \alpha)$ is a (strong) monoidal functor.

Finally, to see $(F, \eta, \alpha)$ is a symmetric monoidal functor, let $B^\cal X$ and $B^\cal Y$ denote the braidings on $\operatorname{Coh}(\cal X)$ and $\operatorname{Coh}(\cal Y)$, respectively. Note that the braidings equip the identity functor with the structures of monoidal functors into the reverse monoidal categories. Since  
\[
\alpha'_{\ecal F, \ecal G} = B^\cal Y_{F(\ecal G),F(\ecal F)} \circ \alpha_{\ecal G, \ecal F} \circ F(B^\cal X_{\ecal F, \ecal G})
\]
defines a monoidal structure on $F$ (with the same unit isomorphism $\eta: \ecal O_\cal Y \to F(\ecal O_X)$), it suffices to show $\alpha'_{\ecal F, \ecal L} = \beta_{\ecal F}$ for any $\ecal F \in \operatorname{Coh}(\cal X)$ and hence $\alpha = \alpha'$ by uniqueness. This is then equivalent to showing that the natural transformation $\beta \inv \circ \alpha'_{-, \ecal L}: F\circ \tau_\ecal L \to F \circ \tau_\ecal L$ is the identity. Since $F$ is fully faithful, there is a unique natural transformation $\gamma:\tau_\ecal L \to \tau_\ecal L$ such that $F \circ \gamma = \beta \inv \circ \alpha'_{-, \ecal L}$. Hence, by Lemma \ref{lemma-hombetweentensorfunctors}, it suffices to show $\beta_{\ecal O_\cal X}\inv \circ \alpha'_{\ecal O_\cal X, \ecal L} = F\circ \gamma_{\ecal O_X}$ is the identity, which is equivalent to $\alpha_{\ecal O_\cal X, \ecal L} = \alpha_{\ecal O_\cal X, \ecal L}'$ and follows by the left unit compatibility axiom for monoidal functors since they have the same unit isomorphisms.
\end{proof}

 \subsection{A lemma on Fourier--Mukai transforms}
 \label{subsection-lemmaonFM}

 Let $k$ be a field.

 \begin{lemma}
     Let $\mathcal{G}$ be a tame gerbe over a finite field extension $k'$ of $k$. Let $\mathcal{X}$ be a Noetherian algebraic stack over $k$. Let $P$ be a perfect complex on $\mathcal{G}$. Let $K \in D^b_{\operatorname{Coh}}(\mathcal{G} \times _k \mathcal{X})$. Then 
     $$
     \operatorname{Supp}(Rpr_{2, *}(Lpr_1^*(P)^\vee \otimes^{\mathbb{L}}K))  \subset pr_2(\operatorname{Supp}K) 
     $$
     as subsets of $|\mathcal{X}|$. Furthermore, we have equality if $P$ is a compact generator of $D_{qc}(\mathcal{G})$ (for example, if $P$ is the direct sum of all the simple coherent sheaves on $\mathcal{G}$). 
 \end{lemma}

 \begin{proof}
      We note first that it suffices to show this with $\mathcal{X}$ an affine scheme by flat base change along any flat morphism $U \to \mathcal{X}$ with $U$ an affine scheme.
      
      Let us prove $\subset$. If $u \notin pr_2(\operatorname{Supp}K)$, then since $pr_2$ is closed, there is an open neighborhood $V$ of $u$ in $U$ such that the restriction of $K$ to $\mathcal{G} \times _k V$ is zero, and hence so is the restriction of  $Lpr_1^*(P) \otimes^{\mathbb{L}}K$ to $\mathcal{G}\times _k V$, and also its derived pushforward to $V$. By flat base change again, this coincides with the restriction to $V$ of $Rpr_{2, *}(Lpr_1^*(P)^\vee \otimes^{\mathbb{L}}K)$, so that $u \notin \operatorname{Supp}Rpr_{2, *}(Lpr_1^*(P)^\vee \otimes^{\mathbb{L}}K)$, as needed.
     
     Now assume $P$ is a compact generator of $D_{qc}(\mathcal{G})$ and we'll show $\supset$.  If a point $u \in U$ is not in the support of  $Rpr_{2, *}(Lpr_1^*(P)^\vee \otimes^{\mathbb{L}}K)$, then there is an affine open neighborhood $u \in V \subset U$ such that the restriction of this object to $V$ is zero. By flat base change again, the object $K|_{\mathcal{G} \times _k V} \in D^b_{\operatorname{Coh}}(\mathcal{G} \times _k V)$ satisfies $R \Gamma (\mathcal{G} \times _k V, Lpr_1^*(P)^\vee \otimes^{\mathbb{L}} K|_{\mathcal{G} \times _k V}) = R \operatorname{Hom}(Lpr_1^*(P)|_{\mathcal{G} \times _k V}, K|_{\mathcal{G} \times _k V}) = 0$. But since $\mathcal{G} \times _k V \to \mathcal{G}$ is affine, the object $Lpr_1^*(P)|_{\mathcal{G} \times _k V}$ is a compact generator of $D_{qc}(\mathcal{G} \times _k V)$, so this implies $K|_{\mathcal{G} \times _k V} = 0$ and so $u \notin pr_2(\operatorname{Supp}(K))$. 
 \end{proof}

 \begin{lemma}
 \label{lemma-supportofthekernel}
     Let $\mathcal{X}, \mathcal{Y}$ be tame stacks which are proper over a field $k$. Let $K \in D^b_{\operatorname{Coh}}(\mathcal{X} \times _k \mathcal{Y})$. Let $x$ be a closed point of $\mathcal{X}$ and $y$ a closed point of $\mathcal{Y}$. The following are equivalent.
     \begin{enumerate}
    \item There is a generalized closed point sheaf $\kappa(x,\xi)$ on $\mathcal{X}$ such that $y$ is in the support of $\Phi_K(\kappa(x,\xi))$.
    \item There is a closed point $z$ of $\operatorname{Supp}K$ such that $pr_1(z) = x$ and $pr_2(z) = y$. \qedhere
     \end{enumerate}
 \end{lemma}

\begin{remark}
\label{remark-algebraicallyclosedcaseoffouriermukai}
    Suppose $k = \bar{k}$, so we can identify closed points of the topological space of a finite type tame stack over $k$ with (isomorphism classes of) $k$-points, and thus closed points of $\mathcal{X} \times \mathcal{Y}$ are in bijection with pairs consisting of a closed point of $\mathcal{X}$ and a closed point of $\mathcal{Y}$. Then the lemma says a closed point $p = (x,y)$ of $\mathcal{X} \times _k \mathcal{Y}$ is in the support of $K$ if and only if there exists a generalized closed point sheaf $\kappa(x,\xi)$ on $\mathcal{X}$ supported at $x$ such that $y$ is in the support of $\Phi_K(\kappa(x, \xi))$. 
\end{remark}

 \begin{proof}
    Write $\mathcal{G}_x$ for the residual gerbe of $\mathcal{X}$ at $x$. For a generalized closed point sheaf $\kappa(x,\xi)$ we compute
     $$
     \Phi_K(\kappa(x,\xi)) \cong R(\mathcal{G}_x \times _k \mathcal{Y} \to \mathcal{Y})_*(L(\mathcal{G}_x \times_k \mathcal{Y} \to \mathcal{X} \times _k \mathcal{Y})^*K \otimes^{\mathbb{L}}L(\mathcal{G}_x \times _k \mathcal{Y} \to \mathcal{G}_x)^*\xi) 
     $$
     (where $\xi$ is the simple coherent sheaf on $\mathcal{G}_x$ whose pushforward is $\kappa(x,\xi)$). By the previous lemma, we see that if a closed point $y$ in $|\mathcal{Y}|$ is in the support of this object, then there is a closed point of $\mathcal{G}_x \times _k \mathcal{Y}$ lying over $y$ which is in the support of $L(\mathcal{G}_x \times _k \mathcal{Y} \to \mathcal{X} \times _k \mathcal{Y})^*K$. This gives a closed point of $\mathcal{X} \times _k \mathcal{Y}$ lying over $x$ and $y$ which is in the support of $K$. Conversely, we see that the support of $\Phi_K(\bigoplus _{\xi} \kappa(x, \xi))$ contains all closed points $y$ which are the image of a point of the support of $L(\mathcal{G}_x \times _k \mathcal{Y} \to \mathcal{X} \times _k \mathcal{Y})^*K$, from which we conclude.
 \end{proof}
 
\subsection{Reconstruction for stacks: Weak form}
\label{subsection-weakreconstructionforstacks}


In this section, we prove Theorems \ref{thm-bo-tamestack-introversion} and \ref{theorem-autoequivsintroversion}. We take the following setup: $\mathcal{X}, \mathcal{Y}$ are smooth, proper, tame stacks over a field $k$ of pure dimension $d>0$, with trivial generic stabilizer and nowhere torsion canonical bundles $\omega_{\mathcal{X}}, \omega_{\mathcal{Y}}$. Denote $X, Y$ the coarse spaces of $\mathcal{X}, \mathcal{Y}$. Let $K \in D^b_{\operatorname{Coh}}(\mathcal{X} \times _k \mathcal{Y})$ and assume that the Fourier--Mukai transform $F = \Phi_K : D^b_{\operatorname{Coh}}(\mathcal{X}) \to D^b_{\operatorname{Coh}}(\mathcal{Y})$ is an equivalence of categories. Indeed, \cite{peng2024equivalences}*{Theorem 1.1} states any $k$-linear, exact equivalence in our setup arises in this way.

\begin{prop}
\label{prop-pointlikegenericallyschemenowheretorsioncanonical}
The functors $F$ and $F^{-1}$ satisfy condition (i) of \ref{subsection-conditions}.
\end{prop}

\begin{proof}
    It suffices by symmetry to show $F$ satisfies condition (i). By Lemma \ref{lemma-pointsarepointlike}, any closed point sheaf $\kappa(x,\xi)$ is a point-like object of $D^b_{\operatorname{Coh}}(\mathcal{X})$. Therefore, $F(\kappa(x, \xi))$ is a point-like object of $D^b_{\operatorname{Coh}}(\mathcal{Y})$. By Lemma \ref{lemma-pointlikeandnowheretorsion}, we see that the support of $F(\kappa(x, \xi))$ is a single closed point. Moreover, the collection of all closed point sheaves $\kappa(x,\xi)$ supported at a given closed point $x$ is $\operatorname{Ext}^1$-connected by Lemma \ref{lemma-genericallyaschemeextconnected} and Proposition \ref{prop-faithfulnormalbundle}, hence remains so after applying $F$. This implies that all of the $F(\kappa(x,\xi))$ (with $x$ fixed) are supported at the some closed point of $\mathcal{Y}$, as needed.
\end{proof}

By Lemma \ref{lemma-supports}, we obtain a well-defined bijection between the set of closed subsets of $|\mathcal{X}|$ and the set of closed subsets of $|\mathcal{Y}|$ with the defining property that
$$
\operatorname{Supp}M \mapsto \operatorname{Supp}F(M).
$$
In particular, we obtain a bijection between the sets of closed points of $\mathcal{X}$ and $\mathcal{Y}$ given by
\begin{equation}
\label{equn-bijectiononclosedpoints}
x \leftrightarrow y\iff \text{For all } M \in D^b_{\operatorname{Coh}, \{x\}}(\mathcal{X})\text{, we have }\operatorname{Supp}F(M) = \{y\}.
\end{equation}
Let $\mathcal{U} \subset \mathcal{X}$ be the open substack found in Proposition \ref{prop-structureawayfromcodim2}. Recall in particular that $\omega_{\mathcal{X}}|_{\mathcal{U}}$ is point-wise $\otimes$-generating, see Definition \ref{defn-pointwisetensorgenerating}.


Let $\mathcal{V} \subset \mathcal{Y}$ be the corresponding open substack. By Corollary \ref{corollary-passagetoopens}, we obtain a 2-commutative square
\[\begin{tikzcd}
	{D^b_{\operatorname{Coh}}(\mathcal{X})} & {D^b_{\operatorname{Coh}}(\mathcal{Y})} \\
	{D^b_{\operatorname{Coh}}(\mathcal{U})} & {D^b_{\operatorname{Coh}}(\mathcal{V})}
	\arrow["F",from=1-1, to=1-2]
	\arrow[from=1-1, to=2-1]
	\arrow[from=1-2, to=2-2]
	\arrow["F_{\mathcal{U}}", from=2-1, to=2-2]
\end{tikzcd}\]
where the horizontal arrows are $k$-linear, exact equivalences. 

\begin{prop}\label{prop-uniform-shift}
    There exists an integer $m$ such that, if $G = (-[m])\circ F_{\mathcal{U}}$, the functors $G, G^{-1}$ satisfy condition (ii) of \ref{subsection-conditions}.
\end{prop}
\begin{proof}
  Let $(y, \eta)$ be a generalized closed point of $\mathcal{V}$. We know that $F^{-1}(\kappa(y, \eta))$ is a point-like object of $D^b_{\operatorname{Coh}}(\mathcal{X})$ supported at a single closed point in $\mathcal{U}$. By Lemma \ref{prop-classifypointlikecycliccase}, there exist an integer $n = n(y, \eta)$ and a generalized closed point $(x,\xi)$ of $\mathcal{U}$ such that $F^{-1}(\kappa(y, \eta)) \cong \kappa(x, \xi)[n]$. Thus $F^{-1}$ takes generalized closed point sheaves supported on $\mathcal{V}$ to shifts of generalized closed point sheaves supported on $\mathcal{U}$. We claim that also $F$ sends generalized closed point sheaves supported on $\mathcal{U}$ to shifts of generalized closed point sheaves supported on $\mathcal{V}$. Equivalently, we have to show that every generalized closed point sheaf $\kappa(x, \xi)$ with $x \in \mathcal{U}$ occurs as a shift of $F^{-1}(\kappa(y, \eta))$ for some generalized closed point $(y, \eta)$ of $\mathcal{V}$. We know that $F$ induces a bijection between the sets of closed points of $\mathcal{U}$ and of $\mathcal{V}$, so we see that there exists a generalized closed point $(y, \eta)$ of $\mathcal{V}$ such that $F^{-1}(y, \eta)$ is a shift of $\kappa(x, \xi')$ for some $\xi'$ possibly different from $\xi$. But by choice of $\mathcal{U}$, we know that $\kappa(x,\xi)$ can be obtained from $\kappa(x, \xi')$ by repeatedly tensoring with $\omega_{\mathcal{X}}$, and since $F^{-1}$ commutes with the Serre functor, we see that by repeatedly tensoring with $\omega_{\mathcal{Y}}$ we can find a generalized closed point sheaf $\kappa(y, \eta')$ which maps via $F^{-1}$ to a shift of $\kappa(x,\xi)$, as needed.

  Finally, we need to show that the integers $n = n(y, \eta)$ above do not depend on the choice of generalized closed point $(y, \eta)$. To this end, we claim that $F^{-1}(\ecal{O}_{\mathcal{Y}})|_{\mathcal{U}}\cong \ecal{L}[m]$ for some vector bundle $\ecal{L}$ on $\mathcal{U}$ and integer $m$. For $(y, \eta)$ a generalized closed point of $\mathcal{V}$, we compute
  $$
  \operatorname{Ext}^i(\ecal{O}_{\mathcal{Y}}, \kappa(y, \eta))= \operatorname{Ext}^i(F^{-1}(\ecal{O}_{\mathcal{Y}}), F^{-1}(\kappa(y, \eta)) = \operatorname{Ext}^i(F^{-1}(\ecal{O}_{\mathcal{Y}}), \kappa(x, \xi)[n(y, \eta)]) = \operatorname{Ext}^{i+n(y, \eta)}(F^{-1}(\ecal{O}_{\mathcal{Y}}), \kappa(x, \xi))
  $$
  which is non-zero for at most one $i$. Moreover, by the previous paragraph, all generalized closed points sheaves $\kappa(x, \xi)$ supported on $\mathcal{U}$ occur as a shift of $F^{-1}(\kappa(y, \eta))$ for some generalized closed point $(y, \eta)$ of $\mathcal{V}$. 
  Also, for a fixed closed point $y$ of $\mathcal{Y}$, since $\{\kappa(y,\eta)\}$ is $\ext^1$-connected, so is $\{F^{-1}(\kappa(y,\eta))\}$. Therefore, by Lemma \ref{lemma-uniform-shift-ext1}, for fixed $y$, the integer $n(y,\eta)$ does not depend on $\eta$.
   Thus by Lemma \ref{lemma-characterizingvectorbundles}, we conclude that $F^{-1}(\ecal{O}_{\mathcal{Y}})|_{\mathcal{U}}$ is a vector bundle, and that all the integers $n(y, \eta)$ are equal to $m$, as claimed. This completes the proof: We replace $F$ with $(-[m]) \circ F$ and this works. Note that for $(x,\xi)$ a generalized closed point of $\mathcal{U}$ and $(y, \eta)$ a generalized closed point of $\mathcal{V}$ we have $F_{\mathcal{U}}(\kappa(x,\xi)) \cong \kappa(y, \eta)$ if and only if $F(\kappa(x,\xi)) \cong \kappa(y, \eta)$ by the diagram defining $F_{\mathcal{U}}$.
\end{proof}

\begin{prop}
\label{prop-standardonthebigopen}
    There exist a $k$-isomorphism $\mathcal{V} \to \mathcal{U}$ and a line bundle $\ecal{L}$ on $\mathcal{V}$ such that the restriction of $F_{\mathcal{U}}$ to $\operatorname{Coh}\mathcal{U}$ is isomorphic to the functor 
    \[
    \ecal{F} \mapsto Lf^*\ecal{F} \otimes^{\mathbb{L}} \ecal{L}[m]. \qedhere 
    \]
\end{prop}

\begin{proof}
    Replace $F$ by $(-[m]) \circ F$ where $m$ was found in the previous proposition. Then $F_{\mathcal{U}}$ and $F_{\mathcal{U}}^{-1}$ satisfy condition (ii) of \ref{subsection-conditions}, so by Lemma \ref{lemma-cohtocoh}, $F_{\mathcal{U}}, F_{\mathcal{U}}^{-1}$  restrict
 to quasi-inverse equivalences of categories $\operatorname{Coh}\mathcal{U} \leftrightarrow \operatorname{Coh}\mathcal{V}$. We know that $F$ commutes with the Serre functors on $\mathcal{X}$ and $\mathcal{Y}$, and hence that there are natural isomorphisms $F(K \otimes^{\mathbb{L}}\omega_{\mathcal{X}}) \cong F(K) \otimes^{\mathbb{L}}\omega_{\mathcal{Y}}$ for $K \in D^b_{\operatorname{Coh}}(\mathcal{X})$, which by the last sentence of Corollary \ref{corollary-passagetoopens}, descend to natural isomorphisms
 $$
 F_{\mathcal{U}}(K \otimes^{\mathbb{L}}\omega_{\mathcal{X}}|_{\mathcal{U}}) \cong F_{\mathcal{U}}(K) \otimes^{\mathbb{L}}\omega_{\mathcal{Y}}|_{\mathcal{V}}
 $$
 for $K \in D^b_{\operatorname{Coh}}(\mathcal{U})$.
 Furthermore, $F(\ecal{O}_{\mathcal{U}})$ is a line bundle $\ecal{L}$ on $\mathcal{V}$ by Lemma \ref{lemma-vecttovect}, as $\cal U$ has generically trivial stabilizer. Replacing $F_{\mathcal{U}}$ by $(-\otimes^{\mathbb{L}}\ecal{L}^{-1}) \circ F_{\mathcal{U}}$ we obtain an equivalence of abelian categories
 $$
 F_{\mathcal{U}} : \operatorname{Coh}\mathcal{U} \to \operatorname{Coh}\mathcal{V}
 $$
 such that $\ecal{O}_{\mathcal{U}}\mapsto \ecal{O}_\mathcal{V}$ and there are natural isomorphisms $$F_{\mathcal{U}}(\ecal{F} \otimes\omega_{\mathcal{X}}|_{\mathcal{U}}) \cong F_{\mathcal{U}}(\ecal{F}) \otimes\omega_{\mathcal{Y}}|_{\mathcal{V}}$$
 for $\ecal{F} \in \operatorname{Coh}\mathcal{U}$. 

 Since the coarse space of $\mathcal{U}$ is a scheme, we may write $\mathcal{U} = \bigcup_{i = 1}^n \mathcal{U}_i$ with $\mathcal{U}_i \subset \mathcal{U}$ open with affine coarse space. Let $\mathcal{V}_i \subset \mathcal{V}$ be the corresponding opens. Then the same arguments show that there are 2-commutative squares
 $$
 \begin{tikzcd}
 \operatorname{Coh}\mathcal{U} \ar[r, "F_{\mathcal{U}}"] \ar[d] &\operatorname{Coh}\mathcal{V} \ar[d] \\
 \operatorname{Coh}\mathcal{U}_i \ar[r, "F_{\mathcal{U}_i}"] &\operatorname{Coh}\mathcal{V}_i
 \end{tikzcd}
 $$
 with $F_{\mathcal{U}_i}$ equivalences of $k$-linear abelian categories taking $\ecal{O}_{\mathcal{U}_i}$ to $\ecal{O}_{\mathcal{V}_i}$ such that there are natural isomorphisms 
 $$F_{\mathcal{U_i}}(\ecal{F} \otimes\omega_{\mathcal{X}}|_{\mathcal{U}_i}) \cong F_{\mathcal{U}_i}(\ecal{F}) \otimes\omega_{\mathcal{Y}}|_{\mathcal{V}_i}$$
 for $\ecal{F} \in \operatorname{Coh}\mathcal{U}_i$. Furthermore, by Lemma \ref{lemma-pointwiseonditionfortensorgeneration}, every coherent sheaf on $\mathcal{U}_i$ is a quotient of a direct sum of positive tensor powers of $\omega_{\mathcal{X}}|_{\mathcal{U}_i}$. By Proposition \ref{prop-itsmonoidal} and Remark \ref{remark-lgoestom} (applied to $F_{\mathcal{U}_i}^{-1}$), there is a unique structure of monoidal functor on $F_{\mathcal{U}_i}$ which is compatible with the natural isomorphisms above. By \cite[Theorem 1.1]{Hall-Rydh-coherent}, there are unique $k$-isomorphisms $f_i : \mathcal{V}_i \to \mathcal{U}_i$ such that $F_{\mathcal{U}_i} \cong f_i^*$ as a monoidal functor. By this uniqueness applied to the overlaps, we see that there is a unique $k$-isomorphism $\mathcal{V} \to \mathcal{U}$ such that $f|_{\mathcal{V}_i} = f_i$ for every $i$. Furthermore, there are natural isomorphisms $f^*\ecal{F} \cong F_{\mathcal{U}}(\ecal{F})$ for $\ecal{F} \in \operatorname{Coh}\mathcal{U}$ which are obtained by glueing from the natural isomorphisms $f^*\ecal{F}|_{\mathcal{V}_i} \cong f_i^*(\ecal{F}|_{\mathcal{U}_i})\cong F_{\mathcal{U}_i}(\ecal{F}|_{\mathcal{U}_i}) \cong F_{\mathcal{U}}(\ecal{F})|_{\mathcal{V}_i}$. This completes the proof.
 \end{proof}

When $F$ is a Fourier--Mukai transform, we will show that there is an isomorphism of coarse spaces $X \cong Y$ compatible with the isomorphism $\mathcal{U} \to \mathcal{V}$.

\begin{lemma}
    Suppose there is an object $K \in D^b_{\operatorname{Coh}}(\mathcal{X} \times _k \mathcal{Y})$ such that $F \cong \Phi_K$ is the Fourier--Mukai transform with respect to $K$. Then the fibers of $\operatorname{Supp}K$ over closed points of $|\mathcal{X}|$ or of $|\mathcal{Y}|$ are finite. Furthermore, its intersection with the open $\mathcal{U} \times _k \mathcal{V} \subset \mathcal{X} \times _k \mathcal{Y}$ is equal (set-theoretically) to the graph of the isomorphism $f : \mathcal{V} \to \mathcal{U}$.
\end{lemma}

\begin{proof}
Both assertions can be checked after base change to the algebraic closure of $k$. Note that all hypotheses on $\mathcal{X}, \mathcal{Y}$ continue to hold after base chance. In particular, since $F \cong \Phi_K$ is a Fourier--Mukai equivalence, by the calculus of kernels, the base change $K_{\bar{k}}$ induces an equivalence between $D^b_{\operatorname{Coh}}(\mathcal{X}_{\bar{k}})$ and $D^b_{\operatorname{Coh}}(\mathcal{Y}_{\bar{k}})$. Also, the dualizing sheaves of $\mathcal{X}_{\bar{k}}, \mathcal{Y}_{\bar{k}}$ remain nowhere torsion by Lemma \ref{lemma-checkafterbasechange} and \cite[Theorem 3.1(1)]{hall2024generalizedbondalorlovfaithfulnesscriterion} which shows that $(\omega_{\mathcal{X}})_{\bar{k}}$ is the dualizing complex on $\mathcal{X}_{\bar{k}}$. Thus we may assume $k = \bar{k}$. 

Then we can identify closed points with $k$-points as in Remark \ref{remark-algebraicallyclosedcaseoffouriermukai}, and by Lemma  \ref{lemma-supportofthekernel}, we see that the closed points of the support of the kernel are the set of pairs of closed points $x$ of $\mathcal{X}$ and $y$ of $\mathcal{Y}$ such that $y$ is in the support of $\Phi_K(\kappa(x,\xi))$ for some generalized closed point $(x, \xi)$ supported at $x$. Thus $(x,y) \in \operatorname{Supp}K$ if and only if $\{x\}$ and $\{y\}$ correspond under the bijection between the closed sets in $|\mathcal{X}|$ and in $|\mathcal{Y}|$. Thus the closed points in $\operatorname{Supp}K$ map bijectively to the closed points in $|\mathcal{X}|$ and in $|\mathcal{Y}|$ under the projections. Furthermore, if $x$ is a closed point in $\mathcal{U}$, then by definition of the isomorphism $f : \mathcal{V} \to \mathcal{U}$, we see that $(x,y) \in \operatorname{Supp}K$ if and only if $x = f(y)$ if and only if $y = f^{-1}(x)$. Thus $\operatorname{Supp}K \cap (\mathcal{U} \times _k \mathcal{V})$ and the graph of $f$ have the same closed points, and therefore agree as closed subsets of $|\mathcal{U} \times _k \mathcal{V}|,$ as needed. 
\end{proof}

\begin{prop}
\label{prop-isoofcoarsespaces}
    There is a $k$-isomorphism of coarse spaces $Y \to X$ which agrees with $f^{-1}$ on the dense open of $\mathcal{V}$ which is representable by a scheme. Furthermore, on closed points, the isomorphism agrees with the bijection (\ref{equn-bijectiononclosedpoints}).
\end{prop}

\begin{proof}
    Let $Z \subset X \times _k Y$ be the reduced closed subspace corresponding to the support of $K$. It follows from the previous lemma that $Z \to X$ and $Z \to Y$ are finite (being proper and quasi-finite) and birational since they agree with the graph of $f$ over the dense opens of $\mathcal{U}$ and $\mathcal{V}$ which are representable by schemes. Since $X, Y$ are normal by Lemma \ref{lemma-coarsespaceisnormal}, the projections $Z \to X$ and $Z \to Y$ are isomorphisms, and the composition $Y \to Z \to X$ is the desired isomorphism. If $x, y$ are closed points of $\mathcal{X}, \mathcal{Y}$, then under this isomorphism we have $x \leftrightarrow y$ if and only if there is a (necessarily unique) closed point $p$ of $\operatorname{Supp}K$ lying over both $x$ and $y$, and by Lemma \ref{lemma-supportofthekernel}, this is equivalent to $x \leftrightarrow y$ under (\ref{equn-bijectiononclosedpoints}).
\end{proof}

\begin{theorem}\label{theorem-bo-tame-stack}
   There are a $k$-isomorphism $f : \mathcal{Y} \to \mathcal{X}$, a line bundle $\ecal{L}$ on $\mathcal{Y}$, an integer $n$, and an open $\mathcal{V} \subset \mathcal{Y}$ whose complement has codimension $\geq 2$ such that:
   \begin{enumerate}
        \item There are natural isomorphisms
    $F(\ecal{F})|_{\mathcal{V}} \cong (f^*\ecal{F} \otimes \ecal{L}[n])|_{\mathcal{V}} $
    for coherent sheaves $\ecal{F}$ on $\mathcal{X}$. 
        \item For every object $M \in D^b_{\operatorname{Coh}}(\mathcal{X})$ whose support is a single closed point $x$, the support of $F(M)$ is equal to $f^{-1}(x)$. 
   \end{enumerate}
   Moreover, $f, \ecal{L},$ and $n$ are uniquely determined by  the conditions (i) and (ii). 
\end{theorem}

\begin{proof}
There is an isomorphism $X \cong Y$ which agrees with the isomorphism $\mathcal{U} \cong \mathcal{V}$ on their shared open, and $\mathcal{U}, \mathcal{V}$ are opens whose complements have codimension $\geq 2$. Hence there is an isomorphism $\mathcal{X} \cong \mathcal{Y}$ which agrees with both of them by Proposition \ref{prop-extendingmorphisms}. We abuse notation and call this isomorphism $f : \mathcal{Y} \to \mathcal{X}$. Similarly, the line bundle $\ecal{L}$ on $\mathcal{V}$ found earlier extends to a line bundle on $\mathcal{Y}$ by Lemma \ref{lemma-hartogs}, which we still call $\ecal{L}$. Now we must check that (i) and (ii) hold. Part (i) follows immediately from Proposition \ref{prop-standardonthebigopen}. Next, (ii) is by Proposition \ref{prop-isoofcoarsespaces} and the definition of the bijection (\ref{equn-bijectiononclosedpoints}).

For uniqueness of $f, \ecal{L}, n$, note that $n, \ecal{L}$ are uniquely determined by $F(\ecal{O}_\mathcal{X})|_{\mathcal{V}} \cong \ecal{L}[n]|_{\mathcal{V}}$ since the codimension of the complement of $\mathcal{V}$ is $\geq 2$, see Lemma \ref{lemma-hartogs}. Finally, the isomorphism $f$ is determined on closed points by the rule (ii). That uniquely determines $f$ on the dense open which is representable by a scheme by \cite[\href{https://stacks.math.columbia.edu/tag/0G05}{Tag 0G05}]{stacks-project}, and that in turn determines $f$ on all of $\mathcal{Y}$ according to Lemma \ref{lemma-separationforstacks}.
\end{proof}

In contrast to the case of schemes (Theorem \ref{theorem-reconstruction-for-schemes}), Theorem \ref{theorem-bo-tame-stack} does not imply all the autoequivalences of $D^b_{\operatorname{Coh}} (\cal X)$ are standard. In order to clarify how they may fail to be standard, let us first observe the following.
\begin{construction}
\label{construction-localsubgroup}
     Let $\aut D^b_{\operatorname{Coh}}(\cal X)$ denote the group of (Fourier--Mukai) autoequivalences of $D^b_{\operatorname{Coh}}(\cal X)$ and $\operatorname{Std}(\cal X) := (\aut(\cal X) \ltimes \pic \cal X) \times \bb Z[1]$ denote the subgroup of standard autoequivalences. Now, we have a well-defined surjective group homomorphism
    \[
    \rho: \aut D^b_{\operatorname{Coh}}(\cal X) \surj \operatorname{Std}(\cal X)
    \]
    given by sending $F$ to $Lf^*(-)\otimes^{\mathbb{L}} \ecal L[n]$, where $(f, \ecal L, n)$ is uniquely determined by $F$ in Theorem \ref{theorem-bo-tame-stack}. Indeed, since in the notation of Proposition \ref{prop-standardonthebigopen}, one has $\Phi_\cal V \circ \Psi_\cal V = (\Phi \circ \Psi)_\cal V$, the uniqueness in Theorem \ref{theorem-bo-tame-stack} shows that $\rho$ is a group homomorphism. It is also clear that $\rho$ is a split surjection. 
\end{construction}
\begin{definition}
    We say an autoequivalence $\Phi \in \ker \rho =: \operatorname{Loc}(\cal X)$ is \textit{local}. By construction, an autoequivalence $\Phi$ is local if and only if 
    (i) there are natural isomorphisms 
    $$
    F(\ecal{F})|_{\mathcal{U}} \cong \ecal{F}_{\mathcal{U}}
    $$
    for $\ecal{F} \in \operatorname{Coh}\mathcal{X}$ where $\mathcal{U} \subset \mathcal{X}$ is the open found in Proposition \ref{prop-structureawayfromcodim2}; and (ii) for any closed point $x \in \cal X$, 
    \[
    \Phi(D^b_{\operatorname{Coh},\{x\}}(\cal X)) = D^b_{\operatorname{Coh},\{x\}}(\cal X). 
    \]  
    By construction, we have the decomposition
    \[
    \aut D^b_{\operatorname{Coh}}(\cal X) = \operatorname{Loc}(\cal X) \rtimes \operatorname{Std}(\cal X). \qedhere 
    \]
\end{definition}
In our setup, $\operatorname{Loc}(\cal X)$ is indeed determined locally at closed points. 


\begin{lemma}\label{lemma-local-rigid}
Let $\Phi \in \operatorname{Loc}(\cal X)$. If $\Phi$ and $\Phi\inv$ satisfy condition (ii) of \S\ref{subsection-conditions}, then $\Phi \cong \id_{D^b_{\operatorname{Coh}}(\cal X)}$.
\end{lemma}
\begin{proof}
    By Lemma \ref{lemma-cohtocoh} and Corollary \ref{corollary-reflexivetoreflexive}, $\Phi$ restricts to an exact autoequivalence 
    \[
    \Phi^\heartsuit:\operatorname{Coh}(\cal X) \to \operatorname{Coh}(\cal X)
    \]
    that preserves reflexive coherent sheaves. We claim $\Phi^\heartsuit$ is naturally isomorphic to $\id_{\operatorname{Coh}(\cal X)}$. Let $j: \cal U \inj \cal X$ denote the open immersion. Since $\Phi \in \operatorname{Loc}(\cal X)$, we have a natural isomorphism $j^*\circ \Phi^\heartsuit \cong j^*: \operatorname{Coh}(\cal X) \to \operatorname{Coh}(\cal U)$. Recall that $\operatorname{Ref}$ denotes the category of reflexive coherent sheaves. Now, as in \cite[\href{https://stacks.math.columbia.edu/tag/0EBJ}{Tag 0EBJ}]{stacks-project}, the functors $j^*, j_*$ induce quasi-inverse equivalences $\operatorname{Ref}(\mathcal{X}) \cong \operatorname{Ref}(\mathcal{U})$ (indeed, one can check smooth locally that $j_*$ preserves reflexive coherent sheaves and the unit and co-unit of $(j^*, j_*) : \operatorname{Ref}(\cal X) \leftrightarrow \operatorname{Ref}(\cal U)$ 
    are isomorphisms). 
    Therefore, we obtain a natural isomorphism $\Phi^\heartsuit |_{\operatorname{Ref}(\cal X)}\cong \id_{\operatorname{Ref}(\cal X)}: \operatorname{Ref}(\cal X) \to \operatorname{Ref}(\cal X)$. Using Lemma \ref{lemma-resolutionbyreflexives},
    we can extend the natural isomorphism to $\Phi^\heartsuit \cong \id_{\operatorname{Coh}(\cal X)}: \operatorname{Coh}(\cal X) \to \operatorname{Coh}(\cal X)$ by using resolutions by reflexive coherent sheaves as in the last part of the proof of Proposition \ref{prop-itsmonoidal}. Now, by the comment right after \cite{peng2024equivalences}*{Definition 6.1}, $\operatorname{Coh}(\cal X)$ has an almost ample set in the sense of \cite{CanonacoStellari2014} and we apply \cite{CanonacoStellari2014}*{Proposition 3.3} to see it extends to a natural isomorphism $\Phi \cong \id_{D^b_{\operatorname{Coh}}(\cal X)}$. 
\end{proof}
\begin{prop}\label{prop-local-determination}
    The restriction homomorphism $r: \operatorname{Loc}(\cal X) \to \prod_{x} \aut D^b_{\operatorname{Coh},\{x\}}(\cal X)$, where the product is taken over closed points of $\cal X \setminus \cal U$, is injective.
\end{prop}
\begin{proof}
    By construction, $\ker r$ consists of autoequivalences that preserve every generalized closed point sheaf and restrict to the identity over $\mathcal{U}$. Hence, $\ker r = \{\id_{D^b_{\operatorname{Coh}}(\cal X)}\}$ by Lemma \ref{lemma-local-rigid}. 
\end{proof}

In the next section, we will give a sufficient condition for the group $\operatorname{Loc}(\mathcal{X})$ to be trivial. Outside of this case, it is still sometimes possible to determine the group $\operatorname{Loc}(\mathcal{X})$ by analyzing the images of $\operatorname{Loc}(\mathcal{X}) \to \operatorname{Aut} D^b_{\operatorname{Coh}, x}(\mathcal{X})$ for closed points $x$. We will see an example of this in \ref{subsection-stackysurfaceexamples}.

\subsection{Reconstruction for stacks: Strong form} \label{subsection-reconstruction-strong-form}

In this section, we prove Theorem \ref{thm-canonicalgeneratesintroversion}. The proof has significant overlap with the previous section, so we sometimes just indicate the necessary modifications. Let $\mathcal{X}$ be a smooth proper connected tame stack over a field $k$ of dimension $d > 0$ with trivial generic stabilizer. Assume furthermore 
that its coarse space is a scheme, and that $\omega_{\mathcal{X}}$ is point-wise $\otimes$-generating (see Definition \ref{defn-pointwisetensorgenerating}).
(so that we can take $\cal U = \cal X$ in Proposition \ref{prop-structureawayfromcodim2}}). Let $\mathcal{Y}$ be a smooth proper connected tame stack over $k$ with trivial generic stabilizer. Let $F : D^b_{\operatorname{Coh}}(\mathcal{X}) \to D^b_{\operatorname{Coh}}(\mathcal{Y})$ be a $k$-linear equivalence of triangulated categories. 

\begin{lemma}
\label{lemma-samedimesionforstacks}
    The dimension of $\mathcal{Y}$ is $d$.
\end{lemma}

\begin{proof}
    Let $\mathcal{Y}$ have dimension $e$. Let $\kappa(x, \xi)$ be a generalized closed point sheaf on $\mathcal{X}$. Then for some $n> 0$, we have $\ecal{F} \otimes \omega_{\mathcal{X}}^{\otimes n} \cong \ecal{F}$, see Lemma \ref{lemma-pointsarepointlike}. Hence denoting $S_{\mathcal{X}}, S_{\mathcal{Y}}$ the Serre functors,
    $$
   F(\ecal{F}) \cong  F(S_{\mathcal{X}}^n(\ecal{F})[-nd])\cong S^n_{\mathcal{Y}}(F(\ecal{F}))[-nd] = F(\ecal{F})\otimes \omega_{\mathcal{Y}}^{\otimes n}[n(e-d)].
    $$
    Considering top and bottom cohomology sheaves of $F(\ecal{F})$ as in Lemma \ref{lemma-samedimension}, we see that $d = e$, as needed.
\end{proof}

\begin{lemma}
    The functors $F$ and $F^{-1}$ satisfy condition (i) of \ref{subsection-conditions}. 
\end{lemma}

\begin{proof}
    By Lemma \ref{lemma-pointsarepointlike}, we see that for any generalized closed point sheaf $\kappa(y, \eta)$ on $\mathcal{Y}$, the object $F^{-1}(\kappa(y, \eta))$ is a point-like object of $D^b_{\operatorname{Coh}}(\mathcal{X}).$ By Proposition \ref{prop-classifypointlikecycliccase}, there is an integer $n$ such that $F^{-1}(\kappa(y, \eta)) \cong \kappa(x, \xi)[n]$ for some generalized closed point sheaf $\kappa(x, \xi)$ on $\mathcal{X}$. Consider now the set of all generalized closed point sheaves supported at $y$. It is $\operatorname{Ext}^1$-connected by Lemma \ref{lemma-genericallyaschemeextconnected} and Proposition \ref{prop-faithfulnormalbundle}. Hence its image under $F^{-1}$ is $\operatorname{Ext}^1$-connected, from which we see that all the generalized closed point sheaves in this set are supported at the same closed point $x$, hence $F^{-1}$ satisfies condition (i).
    
   To show $F$ satisfies property (i), first note that for every closed point $x$ of $\mathcal{X}$, there exists a generalized closed point sheaf $\kappa(y, \eta)$ such that $F^{-1}(\kappa(y, \eta))$ is a shift of a generalized closed point sheaf $\kappa(x, \xi)$ supported at $x$. This follows from the previous paragraph and the fact that the generalized point sheaves are a spanning class, see \ref{subsection-generalizedpoints}. Now it follows from the assumptions that every closed point sheaf $\kappa(x, \xi')$ supported at $x$ can be obtained from $\kappa(x, \xi)$ by repeatedly tensoring with $\omega_{\mathcal{X}}$. Since $F$ commutes with Serre functors, we see that $F$ of any such $\kappa(x, \xi')$ is isomorphic to a shift of $\kappa(y, \eta)\otimes \omega_{\mathcal{Y}}^{\otimes i}$ for some integer $i$, and this is a generalized closed point sheaf supported at $y$. Since $x$ and $\kappa(x, \xi')$ were arbitrary, this completes the proof.
\end{proof}

By Lemma \ref{lemma-supports}, there is a well-defined bijection between the set of closed subsets of $\mathcal{Y}$ and the set of closed subsets of $\mathcal{X}$ given by
\begin{equation}
\label{equn-thebijectionstrong}
\operatorname{Supp}M \mapsto \operatorname{Supp}F(M).
\end{equation}

\begin{lemma}
    There exists an integer $m$ such that, if $G = (-[m]) \circ F$, then the functors $G, G^{-1}$ satisfy condition (ii) of \ref{subsection-conditions}. 
\end{lemma}

\begin{proof}
    The proof of Proposition \ref{prop-uniform-shift} shows this. In fact, we have showed in the proof of the previous lemma that $F$ and $F^{-1}$ send generalized closed point sheaves to shifts of generalized closed point sheaves, and we need only show that these shifts are uniform, which follows from the second paragraph of the proof of Proposition \ref{prop-uniform-shift}.
\end{proof}


\begin{prop}
\label{prop-standardonbigopenstrong}
    There exists a $k$-isomorphism $\mathcal{Y} \to \mathcal{X}$ and a line bundle $\ecal{L}$ on $\mathcal{Y}$ such that the restriction of $F$ to $\operatorname{Coh}(\mathcal{X})$ is isomorphic to the functor $\ecal{F} \mapsto Lf^*\ecal{F} \otimes ^{\mathbb{L}} \ecal{L}[m].$
\end{prop}

\begin{proof}
    Exactly the same as the proof of Proposition \ref{prop-standardonthebigopen} with $\mathcal{V}$ replaced by $\mathcal{Y}$ and $\mathcal{U}$ replaced by $\mathcal{X}$.
\end{proof}

\begin{theorem}
\label{thm-strongreconstructionforstacks}
    There exist a $k$-isomorphism $\mathcal{Y} \to \mathcal{X}$, a line bundle $\ecal{L}$ on $\mathcal{Y}$, and an integer $m$ such that $F \cong Lf^*(-) \otimes ^{\mathbb{L}} \ecal{L}[m].$
\end{theorem}

\begin{proof}
    By \cite[Section 2]{peng2024equivalences}, the category $\operatorname{Coh}(\mathcal{X})$ has an almost ample set. Hence by \cite[Proposition 3.3]{CanonacoStellari2014}, if $F, F'$ are exact functors from $D^b_{\operatorname{Coh}}(\mathcal{X})$ to a triangulated category, $F$ is fully faithful, and $F|_{\operatorname{Coh}(\mathcal{X})} \cong F'|_{\operatorname{Coh}(\mathcal{X})}$, then $F \cong F'$.  Thus applying Proposition \ref{prop-standardonbigopenstrong}, we are done.
\end{proof}

\subsection{Examples of stacks with nowhere torsion canonical bundles}
\label{subsection-stackyexamples}

We work over a fixed field $k$. 

\begin{example}
\label{example-weightedprojectivestackcanonical}
Let $a_0, \dots , a_n$ be integers $\geq 1$. The weighted projective stack $\mathcal{P}(a_0, \dots , a_n)$ has canonical bundle $\omega_{\mathcal{P}(a_0, \dots , a_n)} = \ecal{O}_{{\mathcal{P}(a_0, \dots , a_n)}}(-\sum a_i)$: This follows from the computation of the cohomology of the line bundles $\ecal{O}_{\mathcal{P}(a_0, \dots , a_n)}(d)$ exactly as in the case of projective space, see \cite{auroux2008mirror}*{\S2.3} for instance. The line bundle $\omega_{\mathcal{P}(a_0, \dots , a_n)}$ is nowhere torsion by Example \ref{example-weightedprojectivestack}. Thus $\mathcal{P}(a_0, \dots , a_n)$ can be reconstructed from its derived category in the weak sense of Section \ref{subsection-weakreconstructionforstacks}. If moreover $\gcd(\sum_i a_i, a_j) = 1$ for all $j$, then the canonical bundle is point-wise $\otimes$-generating by Example \ref{example-weightedprojectivestacktensorgen}, so $\mathcal{P}(a_0, \dots , a_n)$ can be reconstructed in the strong sense of Section \ref{subsection-reconstruction-strong-form}. 
\end{example}

A scheme $X/k$ is said to have tame quotient singularities if $X$ admits an \'etale cover $\{U_i \to X\}_i$ where each $U_i$ is the (coarse) quotient $V_i / G_i$ where $V_i$ is a smooth $k$-scheme and $G_i/k$ is a finite linearly reductive group scheme. If $k$ has characteristic zero, this is equivalent to $X$ having quotient singularities. Satriano showed in \cite{Satriano2012} that for a scheme $X/k$ with tame quotient singularities, there is a canonically associated smooth tame stack $\mathcal{X}/k$ with coarse space $X$ and such that the coarse space morphism $\pi : \mathcal{X} \to X$ is an isomorphism over an open whose complement has codimension $\geq 2$. This $\pi$ is called the canonical stack morphism associated with $X$. Note that for such an $X$, $X_{\bar{k}}$ is normal (see \ref{lemma-coarsespaceisnormal}), so the smooth locus $X_{sm} \subset X$ has complement with codimension $\geq 2$.

\begin{lemma}
    Let $X$ be a proper variety over $k$ with tame quotient singularities. Let $\pi : \mathcal{X} \to X$ be the canonical stack morphism associated with $X$. 
    \begin{enumerate}
         \item There is an integer $n > 1$ and a line bundle $\ecal{L}$ on $X$, we have $\pi^*\ecal{L} \cong \omega_{\mathcal{X}}^{\otimes n}$.
        \item $X$ is $\mathbb{Q}$-Gorenstein: For this $n$ and $\ecal{L}$, we have $\ecal{L}|_{X_{sm}} \cong \omega_{X_{sm}}^{\otimes n}$.
       
        \item $\omega_{\mathcal{X}}$ is nowhere torsion on $\mathcal{X}$ if and only if $\ecal{L}$ is nowhere torsion on $X$. \qedhere
    \end{enumerate}
\end{lemma}

\begin{remark}
    We can obtain a more elegant statement by extending the definition of nowhere torsion to $\mathbb{Q}$-line bundles (elements of $\operatorname{Pic}\otimes \mathbb{Q}$) using Corollaries \ref{corollary-powersofnowheretorsiononscheme} and \ref{corollary-powersofnowheretorsiononstack}. Then the lemma says $X$ is $\mathbb{Q}$-Gorenstein, so $\omega_X \in \operatorname{Pic}(X) \otimes \mathbb{Q}$, 
    and $\omega_X$ is nowhere torsion if and only if so is $\omega_{\mathcal{X}}$.
\end{remark}

\begin{proof}
    By Remark \ref{remark-testforgoodnessonstack}, there is an integer $n > 0$ so that $\omega_{\mathcal{X}}^{\otimes n}$ descends to a line bundle $\ecal{L}$ on $X$. We have $\pi^!(\omega_X) = \omega_{\mathcal{X}}$, see \ref{subsection-dualityforstacks}, and $\pi$ is an isomorphism over an open $U \subset X_{sm} \subset X$ whose complement in $X$ has codimension $\geq 2$. Thus we see (using that $\pi^!$ commutes with flat base change and $\pi^!$ for an isomorphism is $L\pi^*$, see \cite[Theorem 3.1(1)]{hall2024generalizedbondalorlovfaithfulnesscriterion}) that $\omega_{\mathcal{X}}$ agrees with $L\pi^*\omega_X^\bullet$ when restricted to $\pi^{-1}(U)$. Thus we see $\ecal{L}|_U \cong \omega_U ^{\otimes n}$ since they agree after pullback along the isomorphism $\pi^{-1}(U) \to U$. If $U \subsetneq X_{sm}$, then use Lemma \ref{lemma-hartogs} to obtain (ii). Then (i) holds by construction, and (iii) is by Remark \ref{remark-testforgoodnessonstack}.
\end{proof}

Next we give a criterion for the canonical bundle of a root stack to be nowhere torsion.

\begin{lemma}\label{lemma-canonicalofroot}
    Let $\mathcal{X}$ be a smooth proper tame stack and $\mathcal{D} \subset \mathcal{X}$ a smooth divisor. Let $\pi: \mathcal{Y} = \sqrt[r]{(\mathcal{X}, \mathcal{D})} \to \mathcal{X}$ be the root stack morphism where $r\geq 1$ is an integer. Let $\mathcal{E} \subset \mathcal{Y}$ be the tautological $r$th root of $\pi^{-1}(\mathcal{D})$.
    \begin{enumerate}
        \item We have $\omega_{\mathcal{Y}} \cong \pi^*\omega_{\mathcal{X}} \otimes \ecal{O}_{\mathcal{Y}}((r-1)\mathcal{E})$
        \item $\omega_{\mathcal{Y}}$ is nowhere torsion on $\mathcal{Y}$ if and only if $\omega_{\mathcal{X}}^{\otimes r} \otimes \ecal{O}_{\mathcal{X}}((r-1)\mathcal{D})$ is nowhere torsion on $\mathcal{X}$. \qedhere
    \end{enumerate}
\end{lemma}

\begin{proof}
    Part (i) follows immediately from Lemma \ref{lemma-rootstacks}. Then by Corollary \ref{corollary-powersofnowheretorsiononstack}, we have that $\omega_{\mathcal{Y}} \cong  \pi^*\omega_{\mathcal{X}} \otimes \ecal{O}_{\mathcal{Y}}((r-1)\mathcal{E})$ is nowhere torsion if and only if $\omega_{\mathcal{Y}}^{\otimes r} = \pi^*(\omega_{\mathcal{X}}^{\otimes r} \otimes \ecal{O}_{\mathcal{X}}((r-1)\mathcal{D})$ is so. By Lemma \ref{lemma-properquasifinitepullback}, this is nowhere torsion if and only if $\omega_{\mathcal{X}}^{\otimes r} \otimes \ecal{O}_{\mathcal{X}}((r-1)\mathcal{D})$ is nowhere torsion, as needed.
\end{proof}

\begin{corollary}
    Let $X$ be a smooth proper variety over $k$, $D \subset X$ a smooth divisor, and $r \geq 1$ an integer. Let $\pi: \mathcal{Y} = \sqrt[r]{(X, D)} \to X$ be the associated root stack morphism. If $\omega_X^{\otimes r} \otimes \ecal{O}_X((r-1)D)$ is a nowhere torsion line bundle on $X$, 
    then $\mathcal{Y}$ is strongly determined by  $D^b_{\operatorname{Coh}}(\mathcal{Y})$ in the sense of Theorem \ref{thm-strongreconstructionforstacks}.
\end{corollary}

\begin{proof}
   This follows from Lemma \ref{lemma-canonicalofroot}. See the proof of Proposition \ref{prop-structureawayfromcodim2} to see why $\omega_{\mathcal{Y}}$ is point-wise $\otimes$-generating.
\end{proof}

\begin{example} \ 
\begin{enumerate} 
    \item Let $X = \mathbb{P}^n_k$ and let $D \subset X$ be a smooth hypersurface of degree $d$. Let $r \geq 1$ be an integer and set $\mathcal{Y} = \sqrt[r]{(X,D)}$. Then $\omega_{\mathcal{Y}}$ is nowhere torsion if and only if 
    $$\omega_{X}^{\otimes r} \otimes \ecal{O}_X((r-1)D) = \ecal{O}_{\mathbb{P}^n_k}(-r(n+1) + (r-1)d)$$ is nowhere torsion, if and only if $(r-1)d \neq r(n+1)$. If this holds, then $D^b_{\operatorname{Coh}}(\mathcal{Y})$ is strongly determined by its derived category.
    \item If $X$ is a smooth proper variety over $k$, $\omega_X$ is ample, and $D \subset X$ is an ample smooth divisor, then for any integer $r \geq 1$, the root stack $\sqrt[r]{(X,D)}$ is strongly determined by its derived category. \qedhere
\end{enumerate}
\end{example}

\begin{lemma}\label{lemma-good-canonical-on-quotient-tame}
    Let $X$ be a proper Gorenstein scheme over $k$. Suppose that a finite 
    linearly reductive group scheme $G$ acts on $X$. Then, $\omega_{[X/G]}$ is nowhere torsion if and only if $\omega_X$ is nowhere torsion.
\end{lemma}
\begin{proof}
    Note that $X \to [X/G]$ is finite syntomic, being a $G$-torsor, hence $[X/G]$ is Gorenstein since $X$ is. There is a Cartesian square
    $$
    \begin{tikzcd}
        X \ar[r, "p"] \ar[d, "f"] & \operatorname{Spec}k \ar[d, "g"]\\
        {[X/G]} \ar[r] & BG_k,
    \end{tikzcd}
    $$
    in which every arrow is flat, proper, and tame. By the properties of the $( -)^!$ functors proved in \cite[Theorem 3.1]{hall2024generalizedbondalorlovfaithfulnesscriterion}, we obtain
    $$
    \omega_X = f^!(\omega_{[X/G]}) = f^*\omega_{[X/G]}\otimes^{\mathbb{L}} f^!(\ecal{O}_{[X/G]}) = Lf^*\omega_{[X/G]} \otimes^{\mathbb{L}}Lp^*(g^!\ecal{O}_{BG_k}) = f^*\omega_{[X/G]}.
    $$
    Here we have $g^!(\ecal{O}_{BG_k}) \cong \ecal{O}_{\operatorname{Spec}k}$ because $g$ is finite syntomic, so $g^!(\ecal{O}_{BG_k})$ is a line bundle on $\operatorname{Spec}k$, hence isomorphic to the structure sheaf. Since $f$ is a finite morphism, we conclude by Lemma \ref{lemma-properquasifinitepullback}. 
\end{proof}
\begin{example} \label{example-multiple-crepant}
    Set $\iota([u:v]) = [-u:v]$ to be the involution on $\bb P^1$ and set $M = (\bb P^1)^3$. Suppose $\operatorname{char} k \neq 2$. Set 
    \[
    G = (\bb Z/2\bb Z)^2 \cong \{\id, (\iota, \iota, \id), (\iota, \id, \iota), (\id, \iota, \iota)\} \subset \aut M. 
    \]
   Consider the corresponding quotients $\cal X = [M/G]$ and $X = M/G$. Then, $\cal X$ has a nowhere torsion canonical bundle by Lemma \ref{lemma-good-canonical-on-quotient-tame}. Also, $X$ is covered by eight affine charts, which correspond to the eight $G$-invariant affine charts of $M$ and are isomorphic to $[\bb A_{x,y,z}^3/G]$ with the restricted $G$-action, and therefore each affine chart of $X$ is isomorphic to
   \[
   U = \spec k[x^2,y^2,z^2,xyz] = \spec k[a,b,c,d]/\bra{abc-d^2} = \spec R.
   \]
   In \cite{donagi2017global}*{\S5.1}, they construct four crepant resolutions of $U$, a central crepant resolution and three outer crepant resolutions, using a toric method. Now, we claim the central resolution in loc cit is given by blowing up the reduced singular locus of $U$, cut out by $I = \bra{ab,ac,bc,2d} = \bra{ab,ac,bc,d}$. Indeed, using the formula as in \cite[\href{https://stacks.math.columbia.edu/tag/0804}{Tag 0804}]{stacks-project}, $\operatorname{Bl}_IU$ is covered by the four affine charts
   \begin{align*}
       \spec R[ac/ab,bc/ab,d/ab]  &= \spec k[a,b,d/ab] \\
       \spec R[ab/ac,bc/ac,d/ac]  & = \spec k[a,c,d/ac] \\
       \spec R[ab/bc,ac/bc,d/bc] & = \spec k[b,c,d/bc] \\
       \spec R[ab/d,ac/d,bc/d] & = \spec k[ab/d,ac/d,bc/d]. 
   \end{align*}
   Setting $w = b_1 = d_{11} = d/ab$, $c_{11} = ab/d$, $d_1 = ac/d$, and $t = bc/d$, we have the gluing relations $tw = 1$, $c_{11}d_{11} = 1$, and $b_1 d_1 = 1$ matching the ones in \cite{donagi2017global}*{\S5.1} and hence $\operatorname{Bl}_I(U)$ agrees with the central crepant resolution. As described in \cite{donagi2017global} through the fan description, $\operatorname{Bl}_IU$ has three $(-1,-1)$-curves (i.e., $\bb P^1$'s whose normal bundles have splitting type $(-1,-1)$), meeting at a single point, and each curve gives an Atiyah flop to an outer crepant resolution.
   
   Now, we globalize this construction to obtain a desired example. Let $\Sigma \subset X$ be the reduced singular locus (which consists of the cube formed by twelve $\bb P^1$'s). Then, the blow-up
   \[
   \rho: Y = \operatorname{Bl}_\Sigma X \to X
   \]
   agrees with the blow-up $\operatorname{Bl}_IU \to U$ on each standard affine chart $U \subset X$, as $\Sigma \cap U = V(I)$ by construction. Now, by \cite{BridgelandKingReid2001}*{Lemma 3.1, 3.2}, we have $\omega_{Y} = \rho^* \rho_* \omega_{Y} = \rho^* \omega_X$ and hence $\rho$ is crepant. Now, let $Y'$ be the Atiyah flop of one of the $(-1,-1)$-curve in $\operatorname{Bl}_I U \subset Y$. Then, since the configuration of $(-1,-1)$-curves changes, $Y \not \cong Y'$. Since all the crepant resolutions are $K$-equivalent to each other by definition, by \cite{BridgelandKingReid2001}*{Theorem 1.2} and \cite{Kawamata2002DEquivalenceAK}*{Theorem 1.7}, we have 
   \[
   D^b_{\operatorname{Coh}}(\cal X) \simeq D^b_{\operatorname{Coh}}(Y) \simeq D^b_{\operatorname{Coh}}(Y'),
   \]
   where $\omega_\cal X$ is nowhere torsion while $\omega_Y$ and $\omega_{Y'}$ are somewhere torsion by considering proper curves in the exceptional loci in $Y, Y'$ viewed as crepant resolutions of $X$. 
\end{example}
\subsection{Autoequivalences on tame surfaces}
\label{subsection-stackysurfaceexamples}
In this section, we work on a smooth proper tame surface $\cal S$ over a field $k$ with nowhere torsion canonical bundle and isolated stacky points $x_1, \dots, x_N$ to illustrate our main results. Set $U = \cal S \setminus \{x_1, \dots x_N\}$, which is a smooth surface. By Proposition \ref{prop-local-determination}, the restriction 
\[
r: \operatorname{Loc}(\cal S) \to \prod_{i} \aut D^b_{\operatorname{coh},\{x_i\}}(\cal S)
\]
is injective. 
In certain cases, we can fully determine $\im r$. We will now carry this out when the stacky points are all \'etale locally of the following form.

\begin{example}
\label{example-ansingsspherical}
    Let $\mu_{n+1}$ act on $\bb A^2$ with weights $(1,-1)$. Then the residual gerbe $B\mu_{n+1} = [0/\mu_{n+1}] \hookrightarrow [\bb A^2/\mu_{n+1}]$ has simple coherent sheaves $\chi_i$ corresponding to the elements $i \in \bb X(\mu_{n+1}) = \mathbb{Z}/{n+1}$. It has normal bundle $\ecal{N}_0 = \chi_1 \oplus \chi_{-1}$ which has trivial determinant, and by Lemma \ref{lemma-ext of skyscrapers}, we compute that for $\ecal{F} = \kappa(0, \chi_i)$, we have
    \[
    \operatorname{Ext}^i(\ecal{F}, \ecal{F}) = \begin{cases}
        k & \text{ if $i = 0,2$} \\
        0 & \text{ otherwise.} 
    \end{cases} \qedhere
    \]
\end{example}

Suppose $\mathcal{X}$ is an algebraic stack and $\mathcal{Z} \subset \mathcal{X}$ is a closed substack. We say a flat morphism of algebraic stacks $f: \mathcal{U} \to \mathcal{X}$ is a \emph{neighborhood} of $(\mathcal{X}, \mathcal{Z})$ if the base change $f^{-1}(\mathcal{Z}) \to \mathcal{Z}$ is an isomorphism. This implies that for $n \geq 1$, $f$ is a neighborhood of the closed substack $\mathcal{Z}_{n}$ defined by the ideal sheaf $\ecal{I}_{\mathcal{Z}}^{n}$. In particular, the conormal sheaves of $\mathcal{Z}$ in $\mathcal{X}$ and of $f^{-1}(\mathcal{Z})$ in $\mathcal{U}$ are identified under the isomorphism $f^{-1}(\mathcal{Z}) \to \mathcal{Z}$. We will say $f$ is a neighborhood of a closed subset $T \subset |\mathcal{X}|$ if $f$ is a neighborhood of the reduced closed substack with $T$ as its underlying set, and we say a pair of flat morphisms $f: \mathcal{U} \to \mathcal{X}, g : \mathcal{U} \to \mathcal{Y}$ of algebraic stacks is a \emph{common neighborhood} of the pairs $(\mathcal{X}, \mathcal{Z}), (\mathcal{Y}, \mathcal{W})$ if $f$ is a neighborhood of $(\mathcal{X}, \mathcal{Z})$, $g$ is a neighborhood of $(\mathcal{Y}, \mathcal{W}),$ and $f^{-1}(\mathcal{Z}) = g^{-1}(\mathcal{W})$ as closed substacks of $\mathcal{U}$.

\begin{definition}
    We say a stacky point $x\in \cal S$ is of \textit{$A_n$-type} if the pairs $(\mathcal{S}, x)$ and $([\bb A^2/\mu_{n+1}], 0)$ (where $\mu_{n+1}$ acts as in the example above) admit a common \'etale neighborhood which is representable, quasi-compact, and separated. 
    In particular, the residual gerbe $\mathcal{G}_x$ is $B\mu_{n+1}$.  
   Since the normal bundles of $x$ and $0$ are identified, the exact same computation as in Example \ref{example-ansingsspherical} shows that for any simple coherent sheaf $\chi_i$ on $\mathcal{G}_x$ corresponding to the element $i \in \bb Z/(n+1)\bb Z = \bb X(G_x)$, the generalized closed point sheaf $\kappa(x, \chi_i)$ is spherical (In particular, $\kappa(x, \chi_i) \otimes \omega_{\mathcal{X}} = \kappa(x, \chi_i \otimes \wedge^2 \ecal{N}_x) = \kappa(x, \chi_i)$ using Lemma \ref{lemma-adjunctionformula} applied to the inclusion of the residual gerbe). Let $B_x \subset D^b_{\operatorname{Coh},\{x\}}(\cal S)$ to be the subgroup generated by spherical twists associated to the objects $\kappa(x, \chi_i)$.  
\end{definition}

For such a point $x$, it is shown in \cite[Example 5.6]{hall2017perfect} that the pullbacks to a common neighborhood $(\mathcal{U}, u)$ as in the definition induce equivalences of categories $D^b_{\operatorname{Coh}, \{x\}}(\mathcal{S}) = D^b_{\operatorname{Coh}, \{u\}}(\mathcal{U}) = D^b_{\operatorname{Coh}, \{0\}}([\bb A^2/\mu_{n+1}])$. Since the equivalences are induced by pullback, we see that the generalized point sheaves $\kappa(x, \chi_i)$ and $\kappa(0, \chi_i)$ correspond under this equivalence, as do the associated spherical twists.

Now, we summarize the classification of autoequivalences for $A_n$-type surfaces obtained in \cite{A_nsings}. First, let us recall the setup in loc cit. 
\begin{construction}
    Suppose $x \in \cal S$ is a stacky point of $A_n$-type. Then, there is an equivalence $D^b_{\operatorname{Coh}, \{x\}} (\cal S) \simeq D^b_{\operatorname{Coh},\{0\}}([\bb A^2/\mu_{n+1}])$ by definition. Let $Y \to \bb A^2/\mu_{n+1}$ be the minimal resolution of singularities and set $Z = C_1 \cup \cdots\cup C_n$ to be the chain of exceptional $(-2)$-curves $C_i$. As in \cite{A_nsings}, we normalize a local McKay correspondence
    \[
    \sf{MK}: D^b_{\operatorname{Coh}, Z}(Y) \simeq D^b_{\operatorname{Coh,\{x\}}} (\cal S)
    \]
    to map $\ecal O_{C_i}(-1)$ to $\kappa(x, \chi_i)$, where $\chi_1, \dots, \chi_n$ are the non-trivial characters of $\mu_{n+1}$. Moreover, $\omega_Z$ corresponds to $\kappa(x, \chi_0)[-1]$, where $\chi_0$ is the trivial character. Thus, writing
    \[
    B = \bra{T_{\ecal O_{C_l(-1)}}, T_{\omega_Z} \mid 1 \leq l \leq n} \subset \aut D^b_{\operatorname{Coh},Z}(Y),
    \]
    we have $B_x = \sf{MK} \circ B \circ \sf{MK}\inv$. 
\end{construction}

\begin{prop}[Ishii--Uehara]\label{prop-Ishii-Ueada}
    Assume $k = \bb C$ and suppose that a stacky point $x \in \cal S$ is of $A_n$-type. Then, for any $\Phi \in \aut D^b_{\operatorname{Coh},\{x\}}(\cal S)$, there exist $T \in B_x$, $\sigma \in S_{n+1}$, and $m \in \bb Z$ such that 
    \[
    T\circ \Phi (\kappa(x, \chi_i)) \cong \kappa(x, \chi_{\sigma(i)})[m]. \qedhere  
    \]
\end{prop}
\begin{proof}
    By \cite{A_nsings}*{Corollary 2.2}, there exists $\Psi \in \bra{B, \iota^*, [m] \mid m \in \bb Z}$ such that 
    \[
    \Psi \circ \sf{MK} \inv \circ \Phi \circ \sf{MK} (\ecal R) \cong \ecal R
    \]
    for any line bundle $\ecal R$ on any subchain of $Z$, where $\iota \in \aut(Y)$ is an involution such that $\iota(C_i) = C_{n+1-i}$. In particular, we may take $\ecal R = \ecal O_{C_i}(-1)$ and $\ecal R = \omega_Z$. By \cite{HuyBook}*{Lemma 8.21} and $\iota^* \ecal O_{C_i}(-1) = \ecal O_{C_{n+1-i}}(-1)$, we can write $\Psi = (\iota^*)^\epsilon \circ \tilde T [- m]$ for $\epsilon \in \{0,1\}$, $\tilde T \in B$, and $m \in \bb Z$. Hence, by $\Psi \circ \sf{MK} \inv \circ \Phi \circ \sf{MK} (\ecal R) \cong \ecal R$, we get the isomorphism $\sf{MK} \circ \tilde T \circ \sf{MK}\inv \circ \Phi \circ \sf{MK}(\ecal R) \cong \sf{MK}  \circ (\iota^*)^\epsilon (\ecal R)[m]$, noting $\iota$ is an involution. Therefore, setting $T = \sf{MK} \circ \tilde T \circ \sf{MK}\inv \in B_x$ and $\sigma$ to be the permutation corresponding to the involution $\iota$ and applying the isomorphism above to $\ecal R = \ecal O_{C_i}(-1)$ or $\ecal R = \omega_Z$,
    we get
    \[
    T \circ \Phi (\kappa(x, \rho_i)) =  \kappa(x, \rho_{\sigma(i)})[m]. \qedhere 
    \]
\end{proof}
\begin{theorem}\label{theorem-localautoforan}
    Assume $k = \bb C$ and suppose that each stacky point $x_i$ of $\cal S$ is of $A_{n_i}$-type with $n_i > 0$. Then, we have
    \[
    \im r = \prod_i B_{x_i} \subset \prod_{i} \aut D^b_{\operatorname{Coh},\{x_i\}}(\cal X).
    \]
    In particular,
    \[
    \aut D^b_{\operatorname{Coh}}(\cal S) \cong \prod_i B_{x_i} \rtimes \operatorname{Std}(\cal S). \qedhere 
    \]
\end{theorem}
\begin{proof}
    First of all, $\prod_i B_{x_i} \subset \im r$ is clear. For the converse, take $\Phi \in \operatorname{Loc}(\cal S)$. By Proposition \ref{prop-Ishii-Ueada}, the restriction $\Phi_i$ of $\Phi$ to $\aut D^b_{\operatorname{Coh},\{x_i\}}(\cal X)$ satisfies that $T_i \circ \Phi(\kappa(x_i,\chi)) = \kappa(x_i,\sigma_i(\chi))[m_i]$ for some $\sigma \in S_{n_i+1}$, $T_i \in B_{x_i}$, and $m_i \in \bb Z$. Since spherical twists associated to $x_i$ act trivially on sheaves supported at $x \neq x_i$, by setting $T = T_1 \circ \cdots \circ T_N$, we see that
    \[
    T\circ \Phi (\kappa(x_i,\chi)) = \kappa(x_i,\sigma_i(\chi))[m_i]
    \]
    for any $i$ and $T \circ \Phi (\kappa(x)) = \kappa(x)$ for any $x \in U$ as $\Phi \in \operatorname{Loc}(\cal S)$. Now, we claim $F:= T\circ \Phi \cong \id_{D^b_{\operatorname{Coh}}(\cal S)}$ (and hence $r(\Phi) = (T_1, \dots, T_N) \in \prod_i B_{x_i}$ as desired). First, we see that $m_ i = 0$. Indeed, for any $i$,
    \[
    \ext^j(F\inv(\ecal O_\cal S), \kappa(x_i, \sigma(\chi))) \cong \ext^{j+m_{i}}(\ecal O_\cal S, \kappa(x_i, \chi))
    \]
    is possibly non-zero only if $j+m_i = 0$ and for any $x \in U$, we have that $\ext^j(F\inv(\ecal O_\cal S), \kappa(x)) \cong \ext^j(\ecal O_\cal S, \kappa(x))$ is possibly non-zero only if $j = 0$. Therefore, by Lemma \ref{lemma-characterizingvectorbundles}, $F\inv(\ecal O_\cal S) \cong \ecal E[m]$ for some vector bundle $\ecal E$ and $m \in \bb Z$. Since $F\inv$ restricts to an identity on $U$, we see that $m = 0$ and thus $m_i = 0$. Hence, by Lemma \ref{lemma-local-rigid}, $F \cong \id_{D^b_{\operatorname{Coh}}(\cal S)}$. 
\end{proof}
\begin{remark}
    To the authors' knowledge, the computations of local autoequivalences for $D_n, E_6, E_7, E_8$ are still unknown. 
\end{remark}

In the rest of the section, we study the weighted projective stack $\cal P(1,1,2)$ which has a single isolated stacky point of type $A_1$ and provides a clear example of the results so far. We first begin by describing the McKay correspondence geometrically. This seems to be well-known to experts, but let us record details for completeness. 
\begin{lemma}\label{lemma-local-identification-rootblowup}
    Let $n \geq 1$ be an integer. Let the group scheme $\mu_n$ act on $\mathbb{A}^2_k$ with weights $(1,1),$ that is, by $u \cdot (x,y) = (ux, uy)$. Then there is a $2$-commutative diagram of algebraic stacks
    $$
    \begin{tikzcd}
        \mathcal{X} \ar[r, "r"] \ar[d, "b"] & \operatorname{Bl}_{(0,0)}(\mathbb{A}^2_k/\mu_n) \ar[d] \\
        {[\mathbb{A}^2_k/\mu_n]} \ar[r] & \mathbb{A}^2_k / \mu_n
    \end{tikzcd}
    $$
    in which $b$ is the blow up in the closed substack $B\mu_{n, k} \cong [(0,0)/\mu_n] \subset [\mathbb{A}^2_k/\mu_n] $ and $r$ is the $n$th root stack with respect to the exceptional divisor. Moreover, the exceptional divisor of $b$ is equal to the tautological $n$th root of the exceptional divisor of $\operatorname{Bl}_{(0,0)}(\mathbb{A}^2_k/\mu_n)$ as closed substacks of $\mathcal{X}$.
\end{lemma}

Note that the right vertical arrow can be characterized as the minimal resolution of singularities of $\mathbb{A}^2_k/\mu_n$. 

\begin{proof}
    The coarse moduli space $\mathbb{A}^2_k/\mu_n$ is the spectrum of the ring of invariants $k[x^n, x^{n-1}y, \dots , xy^{n-1}, y^n]$. The blow up of $\mathbb{A}^2_k/\mu_n$ in the closed point $(0,0)$ defined by the maximal ideal $(x^n, x^{n-1}y, \dots , xy^{n-1}, y^n) \subset k[x^n, x^{n-1}y, \dots , xy^{n-1}, y^n]$ is covered by the two affine opens 
    $$
    \operatorname{Spec}k[x^n, x^{n-1}y/x^n, \dots , xy^{n-1}/x^n, y^n/x^n] = \operatorname{Spec}k[x^n, y/x],
    $$
    on which the exceptional is defined by $x^n = 0$, and 
    $$
    \operatorname{Spec}k[x^n/y^n, x^{n-1}y/y^n, \dots , xy^{n-1}/y^n, y^n] = \operatorname{Spec}k[x/y, y^n]
    $$
    on which the exceptional is defined by $y^n = 0$. In particular, the affine charts of the blow up corresponding to the other generators $x^iy^j$ of the maximal ideal are contained in the union of these two. The $n$th root stack with respect to the exceptional divisor is then a union of the two open substacks
    $$
    [\operatorname{Spec}k[x,y/x]/\mu_n]
    $$
    on which $\mu_n$ acts on $x$ with weight $1$ and on $y/x$ with weight zero, and 
    $$
    [\operatorname{Spec}k[x/y, y]/\mu_n]
    $$
    on which $\mu_n$ acts on $y$ with weight $1$ and on $x/y$ with weight zero. The tautological $n$th root is given by the closed substack $x = 0$ on the first open and $y = 0$ on the second. The blow up of $[\mathbb{A}^2_k/\mu_n]$ in the closed $[(0,0)/\mu_n]$ has an identical local description, and the exceptional divisor corresponds to the tautological $n$th root. Using their universal properties, one obtains globally defined morphisms between the blow up and the root stack, and this local check shows that they are isomorphisms, as needed.
\end{proof}
We can globalize the diagram as follows:
\begin{corollary} \label{corollary-roofforhirzebruch}
    Let $\bb P(1,1,n)$ be the (non-stacky) weighted projective space and $\bb F_n = \bb P_{\bb P^1}(\ecal O_{\bb P^1} \oplus \ecal O_{\bb P^1}(-n)) = \operatorname{Bl}_{[0:0:1]} \bb P(1,1,n)$ the Hirzebruch surface. Then, there is a $2$-commutative diagram of algebraic stacks
    $$
    \begin{tikzcd}
        \mathcal{Z} \ar[r, "r"] \ar[d, "b"] & \bb F_n \ar[d] \\
        {\cal P(1,1,n)} \ar[r] & \bb P(1,1,n)
    \end{tikzcd}
    $$
    in which $b$ is the blow up in the closed substack $B\mu_{n, k} \cong [[0\!:\!0\!:\!1]/\mu_n] \subset \cal P(1,1,n)$ and $r$ is the $n$th root stack with respect to the exceptional divisor. Moreover, the exceptional divisor of $b$ is equal to the tautological $n$th root of the exceptional $(-n)$-curve of $\bb F_n$ as closed substacks of $\mathcal{Z}$.
\end{corollary}
\begin{proof}
    Over the standard affine open chart of $\bb P(1,1,n)$ containing $[0\!:\!0\!:\!1]$, the diagram above agrees with the one in Lemma \ref{lemma-local-identification-rootblowup}, which induces the birational map $\operatorname{Bl}_{B\mu_{n,k}}\cal P(1,1,n) \ratmap \sqrt[n]{(\bb F, C)}$. On the other hand, over $\bb P(1,1,n) \setminus \{[0\!:\!0\!:\!1]\}$, all the maps are the identity. Since the birational map above is also identity over $\bb P(1,1,n) \setminus \{[0\!:\!0\!:\!1]\}$, we glue them to get the identification $\operatorname{Bl}_{B\mu_{n,k}}\cal P(1,1,n) = \sqrt[n]{(\bb F, C)} =: \cal Z$ making the diagram above commute. The latter claim also follows from the local description since the exceptional divisor is in the local chart. 
\end{proof}
To get the McKay correspondence, let us also record some computations. 
\begin{lemma} \label{lemma-cohomologycompforhirzebruch} Let $C$ be the $(-n)$-curve in $\bb F_n$, $F$ a fiber of $\bb F_n \to \bb P^1$, and $\cal E$ the exceptional divisor of $b$ in $\cal Z$. 
    \begin{enumerate}
        \item $r^* \ecal O_{\bb F_n} (C) = \ecal O_\cal Z(n\cal E)$.
        \item $r^* \ecal O_{\bb F_n} (F) = b^* \ecal O_{\cal P(1,1,n)}(1) \otimes \ecal O_\cal Z(-\cal E)$. 
        \item $\bb R b_*\ecal O_\cal Z \cong \ecal O_{\cal P(1,1,n)}$.
        \item $\bb R b_*\ecal O_\cal Z(\cal E) \cong \ecal O_{\cal P(1,1,n)}$. \qedhere 
    \end{enumerate}
\end{lemma}
\begin{proof} Part (i) follows by the construction of a root stack. For (ii), note that $$(\bb F_n \to \bb P(1,1,n))^*\ecal O_{\bb P(1,1,n)}(n) = \ecal O_{\bb F_n}(C + nF).$$ Thus, $r^* \ecal O_{\bb F_n}(C+nF) = b^*\ecal O_{\cal P(1,1,n)}(n)$ and hence part (i) gives $r^* \ecal O_{\bb F_n}(nF) = b^*\ecal O_{\cal P(1,1,n)}(n) \otimes \ecal O_\cal Z(-n\cal E)$. Now, since $\pic(\cal Z) = (\pic(\bb F_n) \oplus \bb Z \cdot \cal E)/\bra{n\cal E - C}$ is torsion-free, we have $$r^* \ecal O_{\bb F_n}(F) = b^*\ecal O_{\cal P(1,1,n)}(1) \otimes \ecal O_\cal Z(-\cal E).$$ Next, (iii) holds more generally for the blow up of an algebraic stack in a regularly immersed closed substack (this reduces by flat base change \cite{hall2017perfect}*{Theorem 2.6} immediately to the case of schemes, which is well-known). Finally, to show (iv), we have a short exact sequence 
\[
0 \to \ecal O_\cal Z \to \ecal O_\cal Z (\cal E) \to \ecal O_\cal E(\cal E) \to 0,
\]
so it suffices to show $\bb R b_* \ecal O_\cal E(\cal E) = 0$ by (iii). Since $B\mu_n \subset \mathcal{P}(1,1,n)$ is a regularly immersed  closed substack of codimension $2$, the exceptional divisor $\mathcal{E}$ is a projective bundle of relative dimension $1$ over $B\mu_n$, and $\ecal{O}_\mathcal{E}(\mathcal{E})$ is the relative $\ecal{O}(-1)$, so $\bb R b_* \ecal O_\cal E(\cal E) = 0$ by flat base change and the usual computation of the cohomology of projective bundles. 
\end{proof}
The following gives a geometric description of the McKay functor described in \cite{auroux2008mirror}*{Theorem 2.29}. 
\begin{prop}\label{prop-akoequiv}
    Under the setup of Corollary \ref{corollary-roofforhirzebruch}, the functor 
    \[
    \Phi:= \bb Rb_*(\bb L r^*(-) \otimes \ecal O_\cal Z(\cal E)):D^b_{\operatorname{Coh}}(\bb F_n) \to D^b_{\operatorname{Coh}}(\bb P(1,1,n))
    \]
    is fully faithful. In particular, it sends the full strong exceptional collection $$\ecal O_{\bb F_n}, \ecal O_{\bb F_n}(F), \ecal O_{\bb F_n}(C + n F), \ecal O_{\bb F_n}(C + (n+1) F)$$ to the strong exceptional collection
    \[
    \ecal O_{\cal P(1,1,n)}, \ecal O_{\cal P(1,1,n)}(1), \ecal O_{\cal P(1,1,n)}(n),\ecal O_{\cal P(1,1,n)}(n+1). 
    \]
    Moreover, when $n = 2$, it is an equivalence. 
\end{prop}
\begin{proof}
    First of all, the claims on the exceptional collections follow directly by \cite{auroux2008mirror}*{Theorem 1.2, Proposition 2.2} and Lemma \ref{lemma-cohomologycompforhirzebruch}. Thus, to conclude by \cite{auroux2008mirror}*{Lemma 2.17}, it suffices to show $\Phi$ is fully faithful on the collection $\ecal O_{\bb F_n}, \ecal O_{\bb F_n}(F), \ecal O_{\bb F_n}(C + n F), \ecal O_{\bb F_n}(C + (n+1) F)$. Let $\ecal L, \ecal L'$ be line bundles from the collection. Then, since $b$ and $r$ are the identities on $\bb F_n \setminus C \cong \cal Z \setminus \cal E \cong \cal P(1,1,n) \setminus B\mu_n$, the map $\hom_{\bb F_n}(\ecal L, \ecal L') \to \hom_{\cal P(1,1,n)}(\Phi(\ecal L), \Phi(\ecal L'))$ is injective since if $\Phi(s) = 0$, then $s|_{\bb F_n \setminus C} = 0$ and hence $s = 0$. Since the dimensions of the Hom sets are the same \cite{auroux2008mirror}*{p.20}, the map is an isomorphism as desired. When $n = 2$, the strong exceptional collection 
    \[
    \ecal O_{\cal P(1,1,n)}, \ecal O_{\cal P(1,1,n)}(1), \ecal O_{\cal P(1,1,n)}(2),\ecal O_{\cal P(1,1,n)}(3)
    \]
    is full, so $\Phi$ is an equivalence. 
\end{proof}
\begin{corollary}\label{corollary-hirzemckay}
    Under the setup of Corollary \ref{corollary-roofforhirzebruch}, for any line bundle $\ecal L$ on $\cal Z$, the functor 
    \[
    \bb Rb_*(\bb L r^*(-) \otimes \ecal L): D^b_{\operatorname{Coh}}(\bb F_n) \to D^b_{\operatorname{Coh}}(\bb P(1,1,n))
    \]
    is fully faithful, and moreover equivalence if $n = 2$. In particular, when $\ecal L = \ecal O_\cal Z$, the functor is $\bb Rb_* \circ \bb Lr^*$. 
\end{corollary}
\begin{proof}
    Since $\pic(\cal Z) = (\pic(\bb F_n) \oplus \bb Z \cdot \cal E)/\bra{n\cal E - C}$, we may write $\ecal L \cong \ecal O_\cal Z(k\cal E) \otimes r^* \ecal O_{\bb F_n}(lF) \cong b^*\ecal O_{\cal P(1,1,n)}(l) \otimes \ecal O_\cal Z((k-l)\cal E)$. Then, by Lemma \ref{lemma-cohomologycompforhirzebruch} we indeed have 
    \[
     \Phi \cong (- \otimes_{\cal P(1,1,n)} \ecal O_{\cal P(1,1,n)}(-(k-1)) ) \circ \bb Rb_*(\bb L r^*(-) \otimes \ecal L) \circ (- \otimes_{\bb F_n} \ecal O_{\bb F_n}((k-l - 1)F)). \qedhere 
    \]
\end{proof}
\begin{remark}
    Since $\bb L b^*$ and $\bb Lr^*$ are fully faithful, in order to show $\bb Rb_* \bb L r^*$ is fully faithful, it is tempting to see if the essential image of $\bb Lr^*$ lies in the essential image of $\bb Lb^*$ (up to a line bundle twist) on which $\bb Rb_*$ is fully faithful. However, for example, the image of the former is
    \[
    \bra{\ecal O_\cal Z, b^* \ecal O_{\cal P(1,1,n)}(1) \otimes \ecal O_\cal Z(-\cal E), b^*\ecal O_{\cal P(1,1,n)}(n),b^*\ecal O_{\cal P(1,1,n)}(n+1) \otimes \ecal O_\cal Z(-\cal E)}
    \]
    and in particular $\bb Lb^* \ecal O_{\cal P(1,1,n)}(1) \otimes \ecal O_\cal Z(-\cal E) \not \in \bb Lb^*(D^b_{\operatorname{Coh}}(1,1,n))$, but $\ecal O_\cal Z(\cal E) \not \in \bb Lb^*(D^b_{\operatorname{Coh}}(1,1,n))$, so twisting by $\ecal O_\cal Z(\cal E)$ (and by other line bundles) does not work. 
\end{remark}
The following cleanly illustrates the theory we have developed so far. 
\begin{example}\label{example-f2}
Use the notation of Example \ref{example-spherical-twist-on-p112}. By Proposition \ref{prop-akoequiv} or Corollary \ref{corollary-hirzemckay}, we have a $k$-linear exact equivalence 
\[
D^b_{\operatorname{Coh}}(\bb F_2) \simeq D^b_{\operatorname{Coh}}(\cal P(1,1,2)).
\]
Although $\omega_{\cal P(1,1,2)}$ is nowhere torsion, this does not contradict Theorem \ref{theorem-bo-tame-stack} since $\omega_{\bb F_2}$ restricts trivially to the unique $(-2)$-curve $C \subset \bb F_2$ and thus $\omega_{\bb F_2}$ is somewhere torsion. Let us explicitly observe what is happening here in parallel to the arguments in \S\ref{subsection-weakreconstructionforstacks}. First, the restrictions of $b$ and $r$ to $\cal Z \setminus \cal E$ induce an isomorphism
\[
f: \cal V: = \bb F_2 \setminus C \to \bb P(1,1,2) \setminus B{\mu_2}=:\cal U. 
\]
Here, $\cal U$ corresponds to the open substack given in Proposition \ref{prop-structureawayfromcodim2}, while $\cal V$ does not satisfy the codimension $2$ condition. Now, every equivalence $\Phi$ constructed in Corollary \ref{corollary-hirzemckay} with $\ecal L = \ecal O_\cal Z(k\cal E)\otimes r^*\ecal O_{\bb F_2}(lF)$ induces a functor $D^b_{\operatorname{Coh}}(\cal V) \to D^b_{\operatorname{Coh}}(\cal U) $ given by $f_*(- \otimes \ecal O_{\bb F_2}(lF)|_\cal V)$. In particular, we have
\[
\Phi(D^b_{\operatorname{Coh},\{p\}}(\bb F_2)) = D^b_{\operatorname{Coh}\{f(p)\}} \text{ \ \ for $p \in \bb F_2 \setminus C$} \quad \text{and} \quad \Phi(D^b_{\operatorname{Coh}, C}(\bb F_2)) = D^b_{\operatorname{Coh}, B\mu_2}(\cal P(1,1,2))
\]
and spherical twists arise away from $\cal V \cong \cal U$. Indeed, assuming $k = \bb C$, Theorem \ref{theorem-localautoforan} applies and
\[
\operatorname{Loc}(\cal P(1,1,2)) = \bra{T_{\kappa(x, \chi_0)}, T_{\kappa(x,\chi_1)}}. \qedhere 
\]
\end{example}

By Example \ref{example-weightedprojectivestackcanonical}, for a weighted projective surface $\cal P(a_0,a_1,a_2)$, we have that $\omega_{\cal P(a_0,a_1,a_2)}$ is point-wise $\otimes$-generating if and only if $\gcd (a_0 +a_1 +a_2, a_i) = 1$ for all $i$. In particular, if $n$ is odd, then $\omega_{\cal P(1,1,n)}$ is point-wise $\otimes$-generating and we can apply Theorem \ref{thm-strongreconstructionforstacks} to see $\operatorname{Loc}(\cal P(1,1,n)) = 1$. 

\textcolor{red}{}
\begin{question}
    If $\omega_{\cal P(1,1,n)}$ is not point-wise $\otimes$-generating (equivalently, if $n$ is even), then is $\operatorname{Loc}(\cal P(1,1,n))$ nontrivial? Note that if $n \neq 2$, then the unique stacky point is not of type $A_n$. 
\end{question}
It would be interesting to understand local autoequivalences at stacky points in smooth tame surfaces more generally. 

\medskip\noindent\textbf{Acknowledgements.} We would like to thank Ruoxi Li and Feiyang Lin for discussions about projective bundles over abelian varieties.  DI is grateful to John S. Nolan for a discussion about the derived McKay correspondence for Hirzebruch surfaces. 
NO is grateful to James Hotchkiss for a discussion about Bondal--Orlov Reconstruction for $\mathbb{G}_m$-gerbes, and is partially supported by the National Science Foundation under Award No.
2402087.

\medskip\noindent\textbf{AI Usage Disclosure.} DI used Gemini through Google's automated search suggestions and ChatGPT for further reference searches. Gemini and ChatGPT provided DI with \cite{lazic2013around}*{Example 5.1} and relevant constructions as examples of ill-behaviors of the canonical bundle of non-klt varieties and DI modified the construction there to the examples in \S\ref{subsubsection-abundance} by himself. ChatGPT helped DI to find examples of affine crepant resolutions with multiple flops \cites{cacciatori2009dbranes, donagi2017global}, but the globalization in Example \ref{example-multiple-crepant} was done by himself. NO used Gemini through Google's automated search suggestions for reference searches.

\bibliography{bib}
\end{document}

%% file: style.tex
\usepackage{stix} 
\usepackage[english]{babel} 
\usepackage{amsmath,amsfonts,amsthm,amscd}
\usepackage{ascmac}
\usepackage{appendix}
\usepackage{bm}
\usepackage{cases}
\usepackage{changepage}
\usepackage{enumitem}
\usepackage{etoolbox}
\usepackage{float}
\usepackage{geometry}
\usepackage{graphicx}
\usepackage[utf8]{inputenc}
\usepackage [autostyle, english = american]{csquotes}
\usepackage{scalerel}
\usepackage{ytableau}
\usepackage{tikz-cd} 
\usetikzlibrary{patterns}
\usepackage{tocloft}
\usepackage{setspace}
\usepackage[color,all]{xy}
\usepackage[cal=euler]{mathalfa}
\usepackage{url}
\usepackage{CJKutf8}

\usepackage[alphabetic,backrefs]{amsrefs}
\usepackage{hyperref}
\hypersetup{colorlinks,linkcolor=blue,citecolor=red}

\DeclareMathOperator{\aut}{Aut}

\DeclareMathOperator{\bl}{Bl}

\DeclareMathOperator{\codim}{codim}

\DeclareMathOperator{\End}{End}

\DeclareMathOperator{\ext}{Ext}

\let \hom \relax
\DeclareMathOperator{\hom}{Hom}

\let \im \relax
\DeclareMathOperator{\im}{Im}

\let \ker \relax
\DeclareMathOperator{\ker}{Ker}

\DeclareMathOperator{\perf}{Perf}

\DeclareMathOperator{\pic}{Pic}

\DeclareMathOperator{\rank}{rank}

\DeclareMathOperator{\spec}{Spec}

\DeclareMathOperator{\supp}{Supp}
\DeclareMathOperator{\sym}{Sym}

\makeatletter
\newcommand{\address}[1]{\gdef\@address{#1}}
\newcommand{\email}[1]{\gdef\@email{\url{#1}}}
\newcommand{\website}[1]{\gdef\@website{\url{#1}}}
\newcommand{\@endstuff}{\par\vspace{\baselineskip}\noindent\small
\begin{tabular}{@{}l}\scshape{Daigo Ito} \\ \scshape\@address\\\textrm{E-mail address:} \@email \\\textrm{Website:} \@website\end{tabular}\\
\begin{tabular}{@{}l}\scshape{Noah Olander} \\ \scshape{Department of Mathematics, University of California, Berkeley, Evans Hall, Berkeley, CA 94720}\\\textrm{E-mail address:} \url{nolander@berkeley.edu} \\\textrm{Website:} \url{https://noaholander.github.io/}\end{tabular}}
\AtEndDocument{\@endstuff}
\makeatother
\address{Department of Mathematics, Columbia University, 2990 Broadway, New York, NY 10021}
\email{di2260@columbia.edu}
\website{https://daigoi.github.io/}

\newcommand {\bb}{\mathbb}
\renewcommand {\cal}{\mathcal}
\newcommand{\ecal}{\mathscr}
\renewcommand {\epsilon}{\varepsilon}
\newcommand {\fr}{\mathfrak}

\newcommand {\bra}[1]{\langle{#1}\rangle}

\newcommand*{\DashedArrow}[1][]{\mathbin{\tikz [baseline=-0.25ex,-latex, dashed,#1] \draw [#1] (0pt,0.5ex) -- (1.3em,0.5ex);}}
\newcommand {\ratmap}{\DashedArrow[->,densely dashed    ]}

\newcommand {\inj}{\hookrightarrow}
\newcommand {\inv}{^{-1}}
\newcommand {\iso}{\cong}

\newcommand {\surj}{\twoheadrightarrow}
\newcommand {\tens}{\otimes}

\newsavebox{\pullbacks}
\sbox\pullbacks{%
\begin{tikzpicture}%
\draw (0,0) -- (1ex,0ex);%
\draw (1ex,0ex) -- (1ex,1ex);%
\end{tikzpicture}}

\newcommand {\ab}{\mathsf{Ab}}

\newcommand {\func}{\mathsf{ Func}}

\newcommand {\id}{{\rm id}}

\newcommand {\op}{\mathsf{op}}

\newcommand {\set}{\mathsf{ Set}} 
\let \sf \relax 
\newcommand{\sf}{\mathsf}

\theoremstyle{plain}
\newtheorem{theorems}{Theorem}[section] 
\newtheorem{claims}{Claim}[theorems]
\newtheorem{conjectures}[theorems]{Conjecture}
\newtheorem{corollaries}[theorems]{Corollary}
\newtheorem{lemmas}[theorems]{Lemma} 
\newtheorem{props}[theorems]{Proposition}

\newtheorem{penmdef}[claims]{Definition}

\theoremstyle{definition}
\newtheorem{constructions}[theorems]{Construction}
\newtheorem{definitions}[theorems]{Definition} 
\newtheorem{axioms}[theorems]{Axiom} 
\newtheorem{examples}[theorems]{Example}     
\newtheorem{notations}[theorems]{Notation}         

\newtheorem{question}[theorems]{Question}
\newtheorem{obss}[theorems]{Observation}
\newtheorem{penmlem}[claims]{Lemma}
\newtheorem{penmthm}[claims]{Theorem}
\newtheorem{penmcor}[claims]{Corollary}
\newtheorem{penmeg}[claims]{Example} 
\newtheorem{inclaims}[claims]{Claim}

\theoremstyle{remark}
\newtheorem{remarks}[theorems]{Remark}

\newtheorem{penmrem}[claims]{Remark}

\makeatletter
\newtheoremstyle{indented}
  {1pt}
  {1pt}
  {\addtolength{\@totalleftmargin}{1.5em}
   \addtolength{\linewidth}{-1.5em}
   \parshape 1 1.5em \linewidth}
  {}
  {\bfseries}
  {.}
  {.5em}
  {}
\makeatother

\theoremstyle{indented}
\newtheorem{pinddef}[claims]{Definition} 
\newtheorem{pindlem}[claims]{Lemma}
\newtheorem{pindthm}[claims]{Theorem}
\newtheorem{pindcor}[claims]{Corollary}
\newtheorem{pindeg}[claims]{Example} 
\newtheorem{pindrem}[claims]{Remark} 
\newtheorem{pindq}[claims]{Question} 
\newtheorem{pindc}[claims]{Conjecture} 
\newtheorem{pindclaim}[claims]{Claim}

\newenvironment{theorem}
{
	\pushQED{\qed}\begin{theorems}}
	{\popQED\end{theorems}}

\newenvironment{prop}
{
	\pushQED{\qed}\begin{props}}
	{\popQED\end{props}}

\newenvironment{notation}
{
	\pushQED{\qed}\begin{notations}}
	{\popQED\end{notations}}

\newenvironment{corollary}
{
	\pushQED{\qed}\begin{corollaries}}
	{\popQED\end{corollaries}}

\newenvironment{definition}
{
	\pushQED{\qed}\begin{definitions}}
	{\popQED\end{definitions}}

\newenvironment{lemma}
{
	\pushQED{\qed}\begin{lemmas}}
	{\popQED\end{lemmas}}

\newenvironment{remark}
{
	\pushQED{\qed}\begin{remarks}}
	{\popQED\end{remarks}}
	
\newenvironment{example}
{
	\pushQED{\qed}\begin{examples}}
	{\popQED\end{examples}}

\newenvironment{construction}
{
	\pushQED{\qed}\begin{constructions}}
	{\popQED\end{constructions}}

